\documentclass{article}
\usepackage[margin=1in]{geometry}
\usepackage{graphicx}
\usepackage[colorlinks=true, allcolors=black]{hyperref}
\usepackage{amsmath,amsthm,amsfonts,amssymb,amscd,mathrsfs}
\usepackage{enumerate}
\usepackage{xcolor}
\usepackage{enumitem}
\usepackage{bbm}
\usepackage{float}
\usepackage{fancyvrb}
\usepackage{booktabs}
\usepackage{subcaption}

\usepackage[ruled,vlined]{algorithm2e}
\def\algbackskip{\hskip-\ALG@thistlm}
\newtheorem{theorem}{Theorem}[section]
\newtheorem{definition}[theorem]{Definition}

\newtheorem{proposition}[theorem]{Proposition}
\newtheorem{lemma}[theorem]{Lemma}

\newtheorem{assumption}[theorem]{Assumption}
\newtheorem*{theorem*}{Problem}
\theoremstyle{remark}
\newtheorem{remark}[theorem]{Remark}

\newcommand{\R}{{\mathbb R}}
\newcommand{\E}{{\mathbb E}}
\newcommand{\A}{{\mathcal A}}
\newcommand{\h}{{\mathcal H}}
\newcommand{\g}{{\mathcal L}}
\newcommand{\inprod}[2]{\left\langle #1, #2 \right\rangle}
\newcommand{\nor}[1]{\left\lVert#1\right\rVert}
\newcommand{\modu}[1]{\left\lvert #1 \right\rvert}
\newcommand{\paren}[1]{\left( #1 \right)}
\newcommand{\brac}[1]{\left[ #1 \right]}
\newcommand{\curlbrac}[1]{\left\{ #1 \right\}}

\title{Neural Actor-Critic Methods for Hamilton--Jacobi--Bellman PDEs: Asymptotic Analysis and Numerical Studies}
\date{\today}
\author{Samuel N. Cohen\footnote{Mathematical Institute, University of Oxford. Email: cohens@maths.ox.ac.uk.}, Jackson Hebner\footnote{Corresponding author; Mathematical Institute, University of Oxford. Email: hebner@maths.ox.ac.uk.}, Deqing Jiang\footnote{Mathematical Institute, University of Oxford. Email: jiang\_deqing\_math@pku.edu.cn.}, Justin Sirignano\footnote{Mathematical Institute, University of Oxford. Email: justin.sirignano@maths.ox.ac.uk.}}
\begin{document}

\maketitle

\begin{abstract}
    We mathematically analyze and numerically study an actor-critic machine learning algorithm for solving high-dimensional Hamilton-Jacobi-Bellman (HJB) partial differential equations from stochastic control theory. The architecture of the critic (the estimator for the value function) is structured so that the boundary condition is always perfectly satisfied (rather than being included in the training loss) and utilizes a biased gradient which reduces computational cost. The actor (the estimator for the optimal control) is trained by minimizing the integral of the Hamiltonian over the domain, where the Hamiltonian is estimated using the critic. We show that the training dynamics of the actor and critic neural networks converge in a Sobolev-type space to a certain infinite-dimensional ordinary differential equation (ODE) as the number of hidden units in the actor and critic $\rightarrow \infty$. Further, under a convexity-like assumption on the Hamiltonian, we prove that any fixed point of this limit ODE is a solution of the original stochastic control problem. This provides an important guarantee for the algorithm's performance in light of the fact that finite-width neural networks may only converge to local minimizers (and not optimal solutions) due to the non-convexity of their loss functions. In our numerical studies, we demonstrate that the algorithm can solve stochastic control problems accurately in up to 200 dimensions. In particular, we construct a series of increasingly complex stochastic control problems with known analytic solutions and study the algorithm's numerical performance on them. These problems range from a linear-quadratic regulator equation to highly challenging equations with non-convex Hamiltonians, allowing us to identify and analyze the strengths and limitations of this neural actor-critic method for solving HJB equations.
\end{abstract}

\section{Introduction}

\subsection{Stochastic optimal control} \label{problem_statement}
Stochastic control theory has at its core the class of Hamilton--Jacobi--Bellman (HJB) partial differential equations (PDEs). These PDEs are often fully nonlinear and difficult to solve, but if a sufficiently regular solution is known, then by the verification theorem, the PDE solution is the true value function of the problem and the minimizer of the Hamiltonian is the true optimal control \cite{pham2009continuous}. These partial differential equations are thus of immense practical interest despite their often pathological nonlinearity.

For high-dimensional stochastic control problems, the relevant HJB equation becomes intractable for traditional numerical PDE techniques such as finite difference schemes. However, deep learning architectures excel at learning high-dimensional functions, and so there has recently been an explosion of interest in using neural networks to solve stochastic control problems. For example, \cite{han2016deep, hure2021deep, jiao2024} solve a time-discretized version of the control problem with machine learning, \cite{beck2019, hure2020deep} develop machine learning methods based on backwards stochastic differential equation formulations, and \cite{al2022extensions, cheridito2025} use machine learning to minimize the squared PDE operator of the HJB equation. Machine learning has also been used for multi-agent control scenarios. For example, \cite{carmona2021deep} summarizes progress in using machine learning for (competitive) mean field games and \cite{hofgard2024, pham2022mean} investigate machine learning for (collaborative) mean field controls. A good summary of the state of the field for these situations is provided in \cite{hu2024}. These techniques have seen great success and are vastly more computationally tractable than naively scaling traditional numeric techniques.

In this paper, we consider a continuous-time, continuous state space, and continuous action space stochastic control problem where both the drift and diffusion coefficients are controlled. The mathematical tools we use in this paper can easily be extended to multi-agent settings and to settings containing jumps; we will focus on a classical control problem for simplicity of presentation. The underlying filtered probability space $\{\Omega_\mathbb{P}, \mathcal{G}, (\mathcal{F}_t)_{t \geq 0}, \mathbb{P}\}$ supports a $d'$-dimensional Brownian motion $W$ and satisfies the usual conditions. The state space $\Omega\subset \R^d$ is open and precompact and the action space is $A \subset \R^{k}$. For a fixed control $u \in \mathcal{U}^\mathcal{F}$, where $\mathcal{U}^\mathcal{F}$ is the space of admissible controls adapted to the filtration $(\mathcal{F}_t)_{t\geq0}$, the corresponding controlled stochastic differential equation is given by\footnote{We may later on omit the superscript $u$ of a controlled process if the identity of the control is clear from context.}
\begin{align}
    dX_t^u = b(X_t^u, u_t)dt + \Phi(X_t^u, u_t)dW_t, \quad X_0^u=x \in \Omega \subset \mathbb{R}^d,
\end{align}
where $b: \Omega \times A \to \mathbb{R}^d$ and $\Phi: \Omega \times A \to \mathbb{R}^{d \times d'}$ are continuous. We define the exit time $$\tau_u = \inf \{t : X_t^u \not\in \Omega \},$$ which then allows the introduction of a cost functional
\begin{align}
    V^u(x) = \E \brac{\int_0^{\tau_u} c(X_s^u, u_s) e^{-\gamma s} ds + g(X_{\tau_u}^u)e^{-\gamma\tau_u} \Big\rvert X_0^u = x},
\end{align}
where $c : \Omega \times A \to \R$ and $g: \partial{\Omega} \to \R$ are continuous functions representing the running and terminal costs, respectively, and $\gamma \geq 0$ is a discounting factor.\footnote{The above problem is sometimes called a \textit{static} or \textit{time-homogeneous} stochastic control setup on account of the coefficients $b$ and $\Phi$ not varying in time or there being a deterministic time cutoff $T > 0$. However, we can easily incorporate these factors by ``lifting" the diffusion into a $d+1$ dimensional state space $\Omega' = \Omega \times [0,T]$ and changing $b$ and $\Phi$ to account for a deterministic time dimension. For simplicity, we consider only the static problem statement from now on.} The value function is then 
\begin{equation}\label{value}
    V(x) = \inf_{u \in \mathcal{U}^\mathcal{F}} V^u(x). 
\end{equation}
Under mild regularity assumptions, the verification theorem of classical stochastic control theory tells us that any function $V \in C^2(\Omega) \cap C^0(\bar\Omega)$ solving
\begin{align}\label{hjb_full}
    \begin{split}
    \inf_{a \in A} \Big \{ \sum_{i}b_i(x,a)\partial_i V(x)+\frac{1}{2}\sum_{i,j}{\Phi\Phi^\intercal}_{ij}(x,a)\partial^2_{ij} V(x) + c(x,a) \Big\}- \gamma V(x)=0,\,\,\,\, &x  \in \Omega,\\
    V(x)=g(x),\,\,\,\, &x \in \partial\Omega.
    \end{split}
\end{align}
is the value function, as defined in \eqref{value}.

The above PDE \eqref{hjb_full} is the HJB equation of interest. Its properties and formal arguments for its derivation are provided in \cite{krylov1980, pham2009continuous, yong_stochastic_1999}. One may show that if $(\mathcal{F}_t)_{t\geq0}$ is the augmented filtration generated by $W$, then the value function \eqref{value} is the same as when the infimum is restricted to only admissible feedback controls $\mathcal{U} \subset \mathcal{U}^\mathcal{F}$, i.e.~those which are measurable functions of the state variable $X$. For ease of notation, we define for each feedback $u \in \mathcal{U}$ the PDE operator 
\begin{align*}
    \mathcal{L}^u V(x) :&= \sum_{i}b_i(x,u(x))\partial_i V(x)+\frac{1}{2}\sum_{i,j}{\Phi\Phi^\intercal}_{ij}(x,u(x))\partial^2_{ij} V(x) + c(x,u(x)) - \gamma V(x) \\
    &= b(x,u(x))\cdot \nabla V(x) + \frac{1}{2}\textrm{Tr}\paren{\Phi\Phi^\intercal(x,u(x)) \textrm{Hess}V(x)} + c(x,u(x)) - \gamma V(x).
\end{align*}
It is then easy to see that the PDE can be represented as
\begin{equation}\label{HJB_with_operator}
    \inf_{u \in \mathcal{U}}\mathcal{L}^u V(x) = 0 \text{ for $x\in \Omega$ \quad and \quad } V(x) = g(x) \text{ for $x \in \partial\Omega$}.
\end{equation}

\subsection{Challenges for standard computational methods}\label{difficulties}
As already mentioned, traditional numerical PDE techniques such as finite difference schemes are computationally intractable in even moderately high dimensions. One might thus try to solve \eqref{hjb_full} using standard deep learning-based PDE solving techniques such as the Deep Galerkin Method \cite{sirignano2018dgm} or Physics Informed Neural Networks \cite{raissi2019physics}. The corresponding optimization problem then becomes that of finding the parameters $\phi$ of a neural network $f_\phi$ that minimize the loss
\begin{align*}
    J(\phi) = \nor{\inf_{a \in A} \curlbrac{\sum_{i}b_i(x,a)\partial_i f_\phi(x)+\frac{1}{2}\sum_{i,j}{\Phi\Phi^\intercal}_{ij}(x,a)\partial^2_{ij} f_\phi(x) + c(x,a)} - \gamma f_\phi(x)}_{\mu(\Omega)}^2 + \nor{f_\phi-g}_{\nu(\partial \Omega)}^2,
\end{align*}
where $\mu$ and $\nu$ are chosen measures (usually Lebesgue) on the interior and boundary of the domain $\Omega$ and the norms are the corresponding $L^2$ ones. However, this formulation encounters several obstacles. Chief among these is that directly performing this optimization task by (deterministic) gradient descent requires computing the gradient term
\begin{equation*}
    \nabla_\phi \inf_{a \in A} \curlbrac{\sum_{i}b_i(x,a)\partial_i f_\phi(x)+\frac{1}{2}\sum_{i,j}{\Phi\Phi^\intercal}_{ij}(x,a)\partial^2_{ij} f_\phi(x) + c(x,a)}.
\end{equation*}
This infimum may not have derivatives in $\phi$, and even if it does, standard autodifferentiation/back-propagation cannot find them.\footnote{Formally, as taking the infimum over an uncountable set is not an operation that can be easily computed, let alone be stored in an autodifferentiator's computational graph,  we cannot take gradients through this step of evaluating the PDE operator.} Further, even if we do somehow learn the value function up to some satisfactorily small loss, we have not yet extracted from it an optimal or near-optimal control.

Note, however, that for a fixed feedback control $u\in \mathcal{U}$, the value function $V^u$ of the corresponding controlled process $X_t^u$ satisfies, by the Feynman--Kac theorem, the following linear PDE
\begin{align}
\begin{split} \label{FC_linear_hjb}
    \sum_{i}b_i(x,u(x))\partial_i V^u(x)+\frac{1}{2}\sum_{i,j}{\Phi\Phi^\intercal}_{ij}(x,u(x))\partial^2_{ij} V^u(x) + c(x,u(x)) - \gamma V^u(x)=0,\,\,\,\, &x  \in \Omega,\\
    V^u(x)=g(x),\,\,\,\, &x \in \partial\Omega.
\end{split}
\end{align}
This potentially high-dimensional linear PDE is solvable with a cornucopia of deep learning methods such as DGM \cite{sirignano2018dgm, jiang2023global}, PINNs \cite{raissi2019physics}, Q-PDE \cite{JMLR:v24:22-1075}, and deep BSDE methods \cite{han2018solving}, among others. These techniques have the advantages of computational tractability and often theoretical guarantees of convergence.

On the other hand, by the verification theorem, we know an optimal control $u^*$ of a stochastic control problem achieves the infimum of the Hamiltonian at each $x \in \Omega$ for the true value function $V$, i.e.
\begin{equation}
    H(u^*, V)(x) = \inf_{u \in \mathcal{U}} \curlbrac{\sum_{i}b_i(x,u(x))\partial_i V(x)+\frac{1}{2}\sum_{i,j}{\Phi\Phi^\intercal}_{ij}(x,u(x))\partial^2_{ij} V(x) + c(x,u(x))}.
\end{equation}
Thus, if we have a neural network $Q_\phi$ that accurately estimates the value function $V$, we can extract a near-optimal control from it by training a neural network $U_\theta$ to minimize the (estimated) Hamiltonian
\begin{align}\label{hamiltonian}
    H(U_\theta, Q_\phi)(x) = \sum_{i}b_i(x,U_\theta(x))\partial_i Q_\phi(x)+\frac{1}{2}\sum_{i,j}{\Phi\Phi^\intercal}_{ij}(x,U_\theta(x))\partial^2_{ij} Q_\phi(x) + c(x,U_\theta(x)).
\end{align}
Note that throughout this paper we will consistently use $Q_\phi$ and $U_\theta$ to denote neural networks tasked with learning HJB solutions and controls, respectively. Here, the subscripts $\phi$ and $\theta$ represent the neural network parameters. 

Of course, we can iterate the above two steps:
\begin{itemize}
    \item Given a candidate near-optimal feedback control $U_\theta \in\mathcal{U}$, train the neural network $Q_\phi$ to approximate $V^{U_\theta}$ by solving \eqref{FC_linear_hjb}.
    \item Given a candidate value function $Q_\phi$, train the neural network $U_\theta$ to approximate the corresponding optimal feedback control $u\in\mathcal{U}$ by minimizing \eqref{hamiltonian}.
\end{itemize}
One then hopes that the candidate value function and candidate control converge to their true counterparts as the number of iterations $\rightarrow \infty$. This inspires the idea of actor-critic algorithms; the algorithm we consider in this paper is of this class. Here, the ``actor" is the neural network $U_\theta$ trained to learn the optimal control $u^*$ and the ``critic" is the neural network $Q_\phi$ trained to learn the value function $V$.

There are several options for training an actor $U_\theta$ to minimize \eqref{hamiltonian} in some sense. For example, we could use gradient descent to approximate the parameters $\theta$ that achieve
\begin{align*}
    \min_\theta \int_\Omega H(U_\theta, Q_\phi)(x)d\mu(x).
\end{align*}
This is the most common approach and the one we will analyze in this paper. However, there are several other techniques. For a summary of the literature on actor-critic algorithms in stochastic control theory, see Section \ref{sec:review}.

We briefly remark that it is not strictly necessary in an actor-critic algorithm to represent both the actor and critic with neural networks -- even for continuous action spaces $A$. For example, the recent paper \cite{ito2021neural} proposes and proves the superlinear convergence of an actor-critic algorithm for solving the Hamilton--Jacobi--Bellman--Isaacs problem in which the actor minimizing the appropriate analogue of \eqref{hamiltonian} is defined theoretically as the pointwise minimizer and is calculated individually each time an $x \in \Omega$ is sampled. However, this approach is computationally expensive and not tractable in high dimensions.

\subsection{Features of this paper}

We present, mathematically analyze, and numerically study a specific actor-critic deep learning algorithm for solving the HJB equation \eqref{hjb_full} when the state space and/or action space are high-dimensional. In particular, we make use of the Q-PDE algorithm from \cite{JMLR:v24:22-1075} to train the critic (i.e.~solve \eqref{FC_linear_hjb}) and apply a form of gradient descent on the integral of the Hamiltonian over $\Omega$ to train the actor (i.e.~minimize \eqref{hamiltonian}). The actor-critic algorithm we define has several advantages:
\begin{itemize}
    \item Following \cite{mcfall2009artificial}, the critic neural network is structured so that the boundary condition $V\rvert_{\partial\Omega} = g$ is automatically satisfied. This means that there is no need to add a separate loss term for the boundary, simplifying training and analysis.
    \item The algorithm works without ever needing to sample controlled SDE paths or calculate any stochastic integrals. Instead, the user has full freedom to determine where learning is prioritized by selecting a fixed measure of their choice to define the actor and critic loss functionals. This avoids the concern that regions of the state space may not be explored frequently enough and hence not be learned.%The algorithm may still be modified to sample training points from the controlled SDE, however, if such behavior is desirable.
\end{itemize}
We leverage these properties to rigorously prove that for actor and critic neural networks initialized in the neural tangent kernel (NTK) regime, the training dynamics of the actor and critic outputs converge to the solution of an infinite-dimensional ordinary differential equation \eqref{limit_ODE} as their hidden layer widths $N, N^* \to \infty$ (cf.~Theorem \ref{thm_limit_ode}). If the Hamiltonian of the HJB equation satisfies a convexity-like assumption in the sense that its global minimum in the control is determined by a first-order condition, we can further prove that any fixed point of the limit ODE is the stochastic control problem's true value function paired with an optimal feedback control (cf.~Theorem \ref{fixed_point_verification}). The convergence analysis of finite-width neural networks is challenging because they have complicated dynamics and highly non-convex loss functions that introduce local minima not inherent to the problem. The infinite-dimensional limit ODE \eqref{limit_ODE}, on the other hand, has a simple form that allows for tractable mathematical analysis.

On the numeric side, our contributions are:
\begin{itemize}
    \item We demonstrate that our algorithm can accurately solve fully nonlinear HJB equations in extremely high dimensions. For example, we solve in 200 dimensions a form of the linear-quadratic regulator problem where the diffusion term depends on both the state and control.
    \item We develop a general framework for writing arbitrarily difficult HJB equations with known closed-form solutions and use it to construct problems to benchmark our actor-critic method. These problems range in difficulty from being only slightly harder than the linear-quadratic regulator to being much more difficult than those usually considered in the literature.
    \item We show that the algorithm performs very well for a wide range of stochastic control problems. However, problems that violate the convexity-like assumption on the Hamiltonian may have non-optimal fixed points towards which the actor-critic training flow might converge. We demonstrate this phenomenon happening in one example (cf.~Problem 2A) and failing to happen in another, slightly modified version where the Hamiltonian was made to satisfy the assumption (cf.~Problem 2B). Thus, Theorem \ref{fixed_point_verification} -- and the assumptions it requires -- have practical significance.
    \item We provide an optional modification to the algorithm to increase numerical stability in cases where the controlled SDE is extremely sensitive to the choice of actor (cf.~Problem 3).
    \item We demonstrate our algorithm learning controls in variants of the Merton problem where restraints on the control space or modifications to the wealth dynamics prevent the existence of a general closed-form solution, with Monte Carlo simulations as our benchmark.
\end{itemize}

\subsection{Neural Q-Learning for PDEs}

Throughout this paper, we frequently make references to the Q-PDE algorithm from \cite{JMLR:v24:22-1075}. The intuition behind the algorithm is to simplify parameter update gradients by reducing dependence on the PDE operator in a way that is analogous to the deep Q-learning algorithm from reinforcement learning.\footnote{See \cite[p.~12]{JMLR:v24:22-1075} for a detailed discussion on why this analogy makes sense.} More specifically, like with the (constrained) Deep Galerkin Method, we are interested in finding the parameters $\phi$ of a neural network $f_\phi$ that minimize the $L^2(\mu)$ error of the PDE operator for some chosen measure $\mu$, i.e. we want to solve $$\min_\phi \int_\Omega[\g f_\phi(x)]^2 d\mu(x).$$ Instead of using the gradient
\begin{equation}
    \frac{d\phi}{dt} = -2\int_\Omega \g f_\phi(x) \nabla_\phi \g f_\phi(x) d\mu(x),
\end{equation}
we may use the biased gradient
\begin{equation}\label{qpde}
    \frac{d\phi}{dt} = -2\int_\Omega \g f_\phi(x) \nabla_\phi (- f_\phi(x)) d\mu(x).
\end{equation}
This especially makes sense if $\g : \h^2 \to L^2$ is strongly negatively monotone, i.e. there exists a $k > 0$ such that $$\inprod{f_1-f_2}{\g f_1 - \g f_2}_{L^2} \leq -k \nor{f_1 - f_2}_{L^2}^2$$ for all $f_1, f_2 \in \h^2,$ although everything is of course still well defined even if this is not the case. Under some regularity assumptions, using gradient clipping and hard constraints, and with the assumption that $\g : \h^2 \to L^2$ is strongly negatively monotone, \cite{JMLR:v24:22-1075} proves that the output of $f_\phi$ trained with the Q-PDE algorithm in the infinite-width limit is given by a deterministic infinite-dimensional ODE flow, and this flow converges ergodically to the solution of the PDE. We will build upon the Q-PDE algorithm for the neural actor-critic algorithm studied in this paper.

\section{Review of Actor-Critic Methods for the HJB Equation}\label{sec:review}

Actor-critic methods are a class of algorithms that emerged from reinforcement learning \cite{konda1999actor, sutton2018reinforcement} and are based on the classical policy iteration method in stochastic control \cite{howard1960}. These algorithms perform two tasks iteratively: i) Given a control $u$, approximate its value function $V^u$; ii) Given an approximation of $V^u$, update the control $u \to u'$ so that $u'$ outperforms $u$ or is in some way closer to an optimal control $u^*$. The term ``actor" refers to the control $u$ (and/or the neural network representing it) and the term ``critic" refers to the value function estimate $Q \approx V^u$ of the current control (and/or the neural network representing it). For the rest of this paper, we may use ``actor" and ``control" interchangeably, and similarly for ``critic" and ``value function".
     
The discussion in Section \ref{difficulties} means that problems in continuous-time stochastic control naturally lend themselves to neural actor-critic methods, and so several papers proposing these types of approaches have already been suggested (and this is not to count the very many papers in discrete-time stochastic control). In particular, we briefly outline the algorithms proposed in \cite{zhou2021actor, zhou2022, al2022extensions, cheridito2025} which (to the authors' knowledge) are the closest analogues to our approach.

\subsection{Variance Reduced Least-Squares Temporal Difference}

In \cite{zhou2021actor}, the authors propose an extension of the Least Squares Temporal Difference (LSTD) method from discrete-time reinforcement learning to a continuous-time stochastic control setting. Namely, given a fixed ``look-forward" time $T > 0$ and a control $u \in \mathcal{U}$, they define the temporal difference functional
\begin{align*}
    \mathrm{TD}^u(Q_\phi) &= \int_0^{T\wedge \tau} e^{-\gamma s}c(X_s,u(X_s))ds - \int_0^{T\wedge\tau} e^{-\gamma s}\nabla Q_\phi(X_s)^\intercal\Phi(X_s, u(X_s))dW_s \\
    &\quad+ e^{-\gamma (T\wedge\tau)}Q_\phi(X_{T\wedge\tau}) - Q_\phi(X_0).
\end{align*}
One may show that if $Q_\phi = V^u$, i.e. $Q_\phi$ is the true critic for a fixed actor $u$, then $\mathrm{TD}^u(Q_\phi) = 0$ almost surely. The authors thus initialize a critic $Q_\phi$ and an actor $U_\theta$, and train $Q_\phi$ to approximate $V^{U_\theta}$ by minimizing the sum of temporal difference $L^2$ loss and the boundary value $L^2$ loss, i.e.
\begin{equation*}
    \mathrm{minimize} \quad \E_{X_0 \sim \mu} \nor{\mathrm{TD}^{U_\theta}(Q_\phi)}^2 + \eta\E_{X\sim\mathrm{Unif}(\partial\Omega)}(Q_\phi(X) - g(X))^2,
\end{equation*}
where $\mu$ is the initial distribution and $\eta > 0$ is a constant that encourages $Q_\phi$ to learn the boundary value.\footnote{Technically, the authors introduce two neural networks: one to learn $V^{U_{\theta}}$ and the other to learn $\nabla V^{U_\theta}$. This is principally a computational trick to ensure stability and plays no significant role in the theoretical underpinning of the algorithm.} The temporal difference loss is calculated from a batch of Monte Carlo samples of the SDE controlled by the current actor $U_\theta$ and the boundary value loss is calculated by Monte Carlo integration.

For the actor training, the authors aim to directly minimize the cost functional $$J^{U_\theta}(Q_\phi) = \E_{X_0 \sim \mu, u}\brac{\int_0^{T\wedge\tau}c(X_s,U_\theta(X_s))e^{-\gamma s}ds + Q_{\phi}(X_{T\wedge\tau})e^{-\gamma (T\wedge\tau)}}.$$ By using It\^{o}'s lemma and excluding difficult to calculate terms such as the functional derivatives $\delta V^{U_\theta}(Q_\phi)/\delta U_\theta$ and $\delta \tau/\delta U_\theta$, the authors derive an expression for a biased functional gradient approximating $\delta J^{U_\theta}(Q_\phi)/\delta\theta.$ Then, similar to the critic training, they simulate a batch of controlled SDE paths via Monte Carlo sampling and take a training step to minimize the cost functional $J^{U_\theta}(Q_\phi)$. The actor and critic training then iterate.

The authors implement the algorithm for four PDEs in up to 20 dimensions. The performance is acceptable, but suffers from the very high computational cost of numerically computing stochastic integrals in high dimensions. There are some other difficulties introduced by the simulated path approach, too, such as the need to estimate the exit time $\tau$ and exit point on $\partial\Omega$ of a path known at only finitely many points.

\subsection{Learning a regularized control}

In \cite{zhou2022}, the authors present a unified theory of continuous-time reinforcement learning. In particular, their form of reinforcement learning relies on relaxed controls, where instead of searching for a single optimal control function $u: [0,T] \times \Omega \to A$, they seek a function $\pi: [0,T] \times \Omega \to \mathcal{P}(A)$ that outputs a probability distribution over the set of possible actions. They then expand the probability space and filtration the stochastic control problem is supported on so that $(\pi(t,X_t^\pi))_{0\leq t \leq T}$ is conditionally independent of $\mathcal{F}^{X_t^\pi}$. This approach allows for adding a regularizer term $p : \mathcal{P}(A) \to \R$ such as differential entropy to the cost functional of a given policy, giving $$J^\pi(t,x) = \E\brac{\int_t^T e^{-\gamma(s-t)} \brac{c(s,X^\pi_s,a^\pi_s) + p(\pi(\cdot \mid s,X^\pi_s))}ds + e^{-\gamma(T-t)} g(X^\pi_T) \bigg| X_t^\pi = x},$$
where $\E$ represents the expectation with respect to a probability space supporting both the Brownian motion $W$ driving the diffusion and also the random actions at each time $t$ sampled from $\pi(t,X_t^\pi) \in \mathcal{P}(A).$ Numerous regularity assumptions are imposed on the full setup, including that the family of relaxed policies $\pi$ is optimized over must be described by a finite set of parameters $\theta$ and must be absolutely continuous with respect to the Lebesgue measure on $A$. The regularization has several practical benefits, such as increased stability of the optimal relaxed policy with respect to the control problem parameters and forcing the control to ``explore" the action space in some sense. For a rigorous derivation and explanation of the former benefit in the slightly different case of a finite action space $A$, see \cite{zhang2021}.

The particular contribution of \cite{zhou2022} is to derive an analog of the classical policy gradient lemma, i.e.~a closed-form expression of the gradient $\nabla_\theta J(t,x ; \pi_\theta)$ of the value functional for each space-time point, and to use this in conjunction with martingale optimality and orthogonality conditions for the value function to write two actor-critic algorithms. In particular, their offline actor-critic algorithm (Algorithm 1 in the paper) could be used to solve stochastic control problems similar to those we consider in this paper. It works by simulating controlled SDE paths to calculate estimates of the relaxed policy gradient and the inner product of the value process against a test function, taking steps in the direction of the former and separate gradient descent steps to minimize the square of the latter.

However, this approach has several drawbacks. Most notably, the parametrization of $\pi(\cdot, \cdot; \theta)$ must be determined a priori, and there is not always a clear or canonical choice. Further, the theory behind the critic training relies on enforcing a certain martingale orthogonality condition against \textit{all} test functions in a very large class of stochastic processes, but practically we must select only one or a few, and there is no obvious best choice.

\subsection{Deep Galerkin Method extended to HJB equation}

In \cite{al2022extensions}, the authors propose an actor-critic extension of the Deep Galerkin Method for finite time horizon HJB equations which they call DGM-PIA (Deep Galerkin Method -- Policy Improvement Algorithm). The algorithm is based on the linearization trick described in Section \ref{difficulties}, with DGM used to train the critic and minimizing the integral of the Hamiltonian to train the actor. That is to say, the authors initialize neural networks $Q_\phi(t,x)$ and $U_\theta(t,x)$ and alternate between solving the linear critic PDE for a fixed actor $U_\theta$
\begin{align*}
    \sum_{i}b_{x_i}(t,x,U_\theta(t,x))\partial_{x_i} Q_\phi(t,x)+\frac{1}{2}\sum_{i,j}{\Phi\Phi^\intercal}_{x_i x_j}(t,x,U_\theta(t,x))\partial^2_{x_ix_j} Q_\phi(t,x) + c(t,x,U_\theta(t,x)) + \partial_t Q_\phi(t,x)&=0,\\
    Q_\phi(T,x)&=g(x)
\end{align*}
by minimizing over $\phi$ the critic loss
\begin{equation*}
    L_\textrm{critic}(\phi; \theta) = \nor{\g^{U_\theta(\cdot,\cdot)}Q_\phi(\cdot,\cdot)}_{\nu_1([0,T]\times\Omega)}^2 + \nor{Q_\phi(T,\cdot) - g}_{\nu_2(\Omega)}^2
\end{equation*}
and then solving for an optimal feedback control given a fixed critic $Q_\phi$ by minimizing over $\theta$ the actor loss
\begin{equation*}
    L_{\mathrm{actor}}(\theta;\phi) = \int_{[0,T]\times\Omega}\brac{\g^{U_\theta(\cdot, \cdot)}Q_\phi(\cdot, \cdot) - \partial_tQ_\phi(\cdot,\cdot)} d\nu_1.
\end{equation*}

The authors employ this algorithm for a few stochastic control examples including the Merton problem, an optimal execution problem, and an LQR problem. They also extend the algorithm to systems of HJB equations and mean-field games. However, none of their examples are high-dimensional, and more importantly they do not provide any theoretical analysis of the training dynamics or convergence of the algorithm. Note also that the critic neural network $Q_\phi$ is not necessarily structured so as to satisfy the boundary condition $Q_\phi(T, \cdot) = g$ automatically, but instead is encouraged to learn it via an $L^2$ loss.

\subsection{Deep learning for stochastic control with jumps}

In \cite{cheridito2025}, the authors introduce two numerical schemes for solving stochastic control problems with controlled jumps in high dimensions. The first, called GPI--PINN 1, is a natural extension of the DGM--PIA algorithm introduced in \cite{al2022extensions}. The structure is the same, but they add a jump-related term in the PDE operator 
\begin{align*}
    \g^u V(x,t) = -\partial_t V(t,x) + c(t,x,u(t,x)) + b(t,x,u(t,x)) &\cdot \nabla_xV(t,x) + \frac{1}{2}\mathrm{Tr}\brac{\Phi\Phi^T(t,x,u(t,x)) \mathrm{Hess}_x V(t,x)} \\ &+ \lambda(t,x,u(t,x))\E^\mathcal{Z}\brac{V(t,x+\gamma(t,x,Z,u(t,x))) - V(t,x)},
\end{align*}
and similarly for the Hamiltonian.
Here, jumps follow a controlled Poisson structure, with $\lambda$ being the intensity and $\mathcal{Z}$ being the unchanging distribution of the random jumps $Z$. In fact, if the jump intensity is uniformly zero, DGM--PIA and GPI--PINN 1 coincide. The presence of the jump-related term in the PDE operator and Hamiltonian is motivated by the verification theorem for controlled jump processes.

However, the combination of needing third derivatives of the critic when evaluating $\nabla_\phi \g^{U_\theta} Q_\phi(t,x)$ and computing an expectation $\E^\mathcal{Z}$ at each sampled space-time point when training with DGM is expensive. To get around this, the authors provide GPI--PINN 2, which is similar to GPI--PINN 1 except that each $$\E^\mathcal{Z}\brac{Q_\phi(t_n,x_n+\gamma(t_n,x_n,Z,U_\theta(t_n,x_n))) - Q_\phi(t_n,x_n)}$$ is estimated by sampling a single $Z = z_n \sim \mathcal{Z}$ and that the critic update follows a non-constrained Q-PDE-like \cite{JMLR:v24:22-1075} biased gradient
$$\phi_{k+1} - \phi_{k} = \eta_1 \int_{\Omega\times [0,T]} \g^{U_\theta} Q_\phi(t,x) \nabla_\phi Q_\phi(t,x) d\mu(x)dt - \eta_2 \int_{\partial \Omega} \paren{Q_\phi(T,x) - g(x)}\nabla_\phi Q_\phi(T, x) d\nu(x).$$
This is generally more effective and computationally efficient than GPI--PINN 1. While developed independently, it is similar to the method we will introduce in this paper, with the two main differences being that GPI--PINN 2 allows for jumps and uses a soft rather than hard constraint to force the boundary values to match. The latter difference, we believe, is surprisingly fundamental -- with the use of a hard constraint allowing for more accurate results and greater analytic insight. Thus, the results in this paper serve as a proof of the wide-network convergence of the training dynamics of GPI--PINN 2 with this modification.

\subsection{Iterative diffusion optimization}

We remark here that \cite{zhou2021actor, zhou2022} rely on simulating controlled SDE paths to estimate a loss functional. This places them in the class of \textit{iterative diffusion optimization} algorithms, of which a systematic theoretical analysis is provided in \cite{nusken2022}. Our algorithm is not an iterative diffusion optimization algorithm, which gives it more flexibility and ensures that the actor and critic are encouraged to learn equally well throughout the whole domain $\Omega$.

\subsection{Results for theoretical gradient flows}

Particular actor-critic gradient flows are analyzed in \cite{feng2025,zhou2025}. We mention these works only briefly because while they do prove convergence theorems for the gradient flows they consider, they do not provide machine learning/neural network algorithms which could be used to computationally implement the flows. Our paper specifically studies actor-critic algorithms implemented with neural networks and proves the convergence of one such algorithm to a limiting flow upon which further analysis can then be developed. To the best of our knowledge, we are the first to derive the exact limiting dynamics of a neural actor-critic algorithm for solving stochastic control problems. In addition, the mathematical approach we take can be easily extended to algorithms which solve HJB equations more general than \eqref{hjb_full}, e.g.~those with jumps or multiple agents.

\section{Our Algorithm and its Analysis}

We now introduce our algorithm. Recall that we are interested in the following stochastic control problem: there exists a controlled process $X_t$ defined on $\{\Omega_\mathbb{P}, \mathcal{G}, (\mathcal{F}_t)_{t \geq 0}, \mathbb{P}\}$ with state space $\Omega\subset \R^d$, action space $A \subset \R^k$, and drift and volatility coefficients $b: \Omega \times A \to \mathbb{R}^d$ and $\Phi: \Omega \times A \to \mathbb{R}^{d \times d'}$ such that
\begin{align*}
    dX_t = b(X_t, u_t)dt + \Phi(X_t, u_t)dW_t, \quad X_0=x \in \Omega \subset \mathbb{R}^d,
\end{align*}
and the corresponding HJB equation is
\begin{align*}
    \inf_{a \in A} \curlbrac{ \sum_{i}b_i(x,a)\partial_i V(x)+\frac{1}{2}\sum_{i,j}{\Phi\Phi^\intercal}_{ij}(x,a)\partial^2_{ij} V(x) + c(x,a)}- \gamma V(x)=0,\,\,\,\, &x  \in \Omega,\\
    V(x)=g(x),\,\,\,\, &x \in \partial\Omega.
\end{align*}
We denote the linearized PDE operator for any admissible feedback control $u \in \mathcal{U}$ by $\g^u$, i.e. for $f \in \mathcal{D}(\g^u)$ we have
\begin{align*}
        \mathcal{L}^u f(x):=\sum_{i}b_i(x,u(x))\partial_i f(x)+\frac{1}{2}\sum_{i,j}{\Phi\Phi^\intercal}_{ij}(x,u(x))\partial^2_{ij} f(x) - \gamma f(x) + c(x,u(x)).
\end{align*}

In addition to the conditions given in Section \ref{problem_statement}, we add the following assumptions to our problem. These are helpful in the mathematical analysis of the actor and critic networks' training dynamics. The last two assumptions in particular allow for the critic network to be hard-constrained (i.e.~constructed in a way that always satisfies the boundary condition), thus simplifying training and allowing for a neural tangent kernel-style analysis.
\begin{assumption}\label{assumption_bounded_second_derivative}
    The functions $b_i, \Phi \Phi^\intercal_{ij}, c$ are in $C_b^{0,2}(\overline\Omega \times A; \R)$ for all indices $i,j$.
\end{assumption}
\begin{assumption}[Auxiliary function $\eta$]\label{assumption_eta}
    There exists a known, explicit auxiliary function $\eta \in C^3_b(\R^d)$ such that $\eta(x) > 0$ for $x\in\Omega$ and $\eta(x) = 0$ for $x \in \partial\Omega$.
\end{assumption}
\begin{assumption}[Interpolation $\overline{g}$ of boundary condition]\label{assumption_g}
    There exists a known, explicit interpolation function $\overline{g} \in C^2_b(\overline\Omega)$ of the boundary condition $g \in C^0(\partial \Omega)$ such that $\overline{g}\rvert_{\partial\Omega} = g$.
\end{assumption}
\begin{remark}
    In \cite{JMLR:v24:22-1075}, additional assumptions are made that the boundary $\partial\Omega$ is $C^{3,\alpha}$ for some $\alpha \in (0,1)$ and $\nabla \eta(x) \neq 0$ for $x \in \partial \Omega$. These are used in \cite[Theorem 33]{JMLR:v24:22-1075} to prove that the image of the kernel $\mathcal{B}$, which is defined in \eqref{kernel_B_op_def} and plays a role in the infinite-dimensional limit ODE flow \eqref{limit_ODE}, is dense in $\h^2_{(0)}$\footnote{We define $\h^2_{(0)} \subset \h^2$ as the space of Sobolev functions that have a boundary trace of 0.}. Such assumptions are not required for the results in this paper because we are concerned with the convergence of the training dynamics of our algorithm to the limit flow \eqref{limit_ODE} containing $\mathcal{B}$ (cf.~Theorem \ref{thm_limit_ode}) rather than the convergence of \eqref{limit_ODE} itself to a solution $V \in \h_{(0)}^2 + \overline{g}$ of the PDE \eqref{hjb_full} (which would be the analogue of \cite[Theorem 47]{JMLR:v24:22-1075}).
\end{remark}
\subsection{The algorithm}
We define two neural networks. The actor, tasked with modeling the (feedback) control, is a fully-connected one-hidden-layer feedforward neural network $U_\theta^N: \overline\Omega \to \mathbb{R}$ given by\footnote{To ease notation, from now on we will without any serious loss of generality consider $A = \R$. All subsequent mathematics can be done in a higher-dimensional Euclidean space $\R^k$ by changing the notation in the obvious ways. If $A$ is not $\R^k$, we can apply a function that maps $\R^k$ to $A$ in all subsequent places where the control appears. This may occasionally cause problems such as losing the ability to take actions in the boundary of $A$, but is generally numerically harmless due to the high regularity of the problems we are considering.}
\begin{align} \label{actor_network}
    U^N_\theta(x):=\frac{1}{N^\beta}\sum_{i=1}^N h^i \sigma(v^i \cdot x+z^i),
\end{align}
where $\beta \in \paren{\frac{1}{2}, 1}$ is a scaling parameter. Here, we represent the parameters $\{h^i,v^i,z^i\}_{i=1}^N$ by a vector $\theta$. The critic network, on the other hand, is $Q_\phi^{N^*}: \overline{\Omega} \to \mathbb{R}$ given by
\begin{equation}\label{q_def}
    Q^{N^*}_\phi(x):= Z^{N^*}_\phi(x) \eta(x) + \overline{g}(x),
\end{equation}
where $Z^{N^*}_\phi$ is a separate neural network
\begin{align} \label{critic_network}
    Z_\phi^{N^*}(x):=\frac{1}{{N^*}^{\beta}} \sum_{i=1}^{N^*} c^i \sigma(w^i \cdot x + b^i).
\end{align}
Similarly, we represent critic parameters $\{c^i,w^i,b^i\}_{i=1}^{N^*}$ by a vector $\phi$. Recall that $\eta(x) \in (0,C]$ for $x\in\Omega$ and $\eta(x) = 0$ for $x \in \partial\Omega$. The structure of \eqref{critic_network} therefore forces the critic network $Q_\phi^{N^*}$ to perfectly match the known boundary condition while still giving it the freedom to learn any value in the interior of $\Omega$.

For simplicity, we assume that the actor and critic networks use the same activation function $\sigma$ and are initialized from a common parameter distribution satisfying the following conditions.
\begin{assumption} \label{assumption_neural_net}
The common activation function $\sigma\in C_b^4(\mathbb{R})$ is non-constant.
There exists a universal constant $C > 0$ such that the initialization of the parameters $\theta$, $\phi$ satisfies:
\begin{itemize}
    \item The parameter sets $\{(c_0^i$, $w_0^i$, $b_0^i)$, $(h_0^j$, $v_0^j$, $z_0^j)\}_{i\in\{1,\dots,N^*\}, j\in \{1,\dots,N\}}$ are all independent and identically distributed.
    \item The parameters $c_0^i,w_0^i,b_0^i$ are independent.
    \item The distribution of $c_0^i$ satisfies $\modu{c^i_0} \leq C$ and $\mathbb{E}\brac{c^i_0}=0$.
    \item We have the moment bounds $\mathbb{E}\brac{|b_0^i|} \leq C$ and $\mathbb{E}\brac{\modu{w_0^{i,j}}^3} \leq C$ for $j \in \{1,\dots,d\}$.
\end{itemize}
\end{assumption}

To help with the theoretical analysis of the algorithm as $N,N^* \to \infty$, we employ smooth gradient clipping. In particular, we use a family of truncation functions of the following type (adapted from \cite{JMLR:v24:22-1075}).
\begin{definition}\label{definition_smooth_truncation}
    We call a set of functions $\{\psi^N\}_{N \geq 1} : \R \to \R$ a family of smooth truncation functions with parameter $\delta \in \paren{0, \frac{1-\beta}{4}}$ if for each $N \geq 1$
    \begin{itemize}
        \item $\psi^N \in C^2_b(\R)$ is increasing
        \item $\modu{\psi^N}$ is bounded by $2N^\delta$
        \item $\psi^N(x) = x$ for $x \in [-N^\delta, N^\delta]$
        \item $\modu{(\psi^N)'}$ is bounded by 1
        \item $F^N := (\psi^N)\cdot(\psi^N)'$ is Lipschitz with uniform constant for all $N \geq 1$
    \end{itemize}
\end{definition}

\begin{remark}
    For concreteness, we mention the following family of smooth truncation functions. Let $\delta = \frac{1-\beta}{5}$ and $$\psi^N(x) = \int_0^x \xi^N(y)dy, \quad \text{where} \quad \xi^N(x) = \begin{cases}
        e^{-(\modu{x}-N^\delta)^2} \hspace{0.5em} &\text{if } \modu{x} \geq N^\delta, \\ 1 \hspace{0.5em} &\text{if } \modu{x} < N^\delta.
    \end{cases}$$
\end{remark}

\begin{remark}
    It is straightforward to see that a family of smooth truncation functions with parameter $\delta > 0$ obeys the following identities:
    \begin{align}
        \modu{\psi^N(x) - x} &\leq \modu{x} \mathbf{1}_{\curlbrac{\modu{x} \geq N^\delta}}, \\
        \modu{F^N(x)-x} &\leq 2\modu{x}\mathbf{1}_{\curlbrac{\modu{x} \geq N^\delta}}.
    \end{align}
\end{remark}

The use of smooth gradient clipping in the setup below allows us to derive the training dynamics of our algorithm in the infinite-width limit where $N,N^* \to \infty$. We also note that gradient clipping is known to typically improve training performance \cite{Zhang2020Why}. However, empirically, the algorithm also works well even if each $\psi^N$ is just the identity map.

Each training iteration is over a length of training time $\Delta t$, during which both the actor and critic update. First, given the current actor $U_{\theta_t}^N$, the critic $Q_{\phi_t}^{N^*}$ takes a step towards learning the corresponding value function $V^{U_{\theta_t}^N}$ by reducing the linearized PDE operator loss
\begin{align} \label{critic_loss}
    L_\mathrm{critic}(\phi_{t}; \theta_t) &= \int_\Omega \modu{\mathcal{L}^{U_{\theta_t}^N} Q_{\phi_t}^{N^*}(x)}^2 d\mu(x)
\end{align}
via updating $\phi_t \leftarrow \phi_{t+\Delta t}$ according to a clipped Q-PDE gradient, where $\mu$ is a chosen finite measure with full support that is absolutely continuous with respect to the Lebesgue measure. After this, the actor $U_{\theta_{t}}^N$ takes a step towards learning the best control corresponding with the critic $Q_{\phi_{t+\Delta t}}^{N^*}$ by reducing the actor loss
\begin{align}\label{actor_loss} \begin{split}
    L_\mathrm{actor}(\theta_{t}; \phi_{t+\Delta t}) &= \int_\Omega H\paren{U_{\theta_{t}}^N, Q_{\phi_{t+\Delta t}}^{N^*}}(x)d\mu(x), %\\
    \end{split}
\end{align}
via updating $\theta_t \leftarrow \theta_{t+\Delta t}$ according to a clipped gradient, where $H\paren{U_{\theta_{t}}^N, Q_{\phi_{t+\Delta t}}^{N^*}}$ is the pointwise estimated Hamiltonian. The process then repeats until some terminal condition (such as a maximum time $\Lambda$) is achieved. This is an approximation of the continuous-time gradient flow where the actor and critic update simultaneously.

Note that calculating the gradient in $\theta$ of the actor loss in \eqref{actor_loss} gives
\begin{align}
    \begin{split}
        \nabla_\theta L_\mathrm{actor}(\theta_{t}; \phi_{t}) &=\nabla_{\theta}\Big[\int_\Omega H(U^N_{\theta_t},Q^{N^*}_{\phi_t})d\mu \Big]\\
        &=\int_\Omega \bigg (\sum_i \partial_u b_i(x,U^N_{\theta_t}(x))\partial_{i} Q^{N^*}_{\phi_t}(x)+\frac{1}{2}\sum_{i,j} \partial_{u}(\Phi\Phi^\intercal)_{ij}(x, U_{\theta_t}^N(x))\partial_{ij}Q_{\phi_t}^{N^*}(x)\\ &\quad + \partial_u c(x,U^N_{\theta_t}(x))\bigg) \nabla_{\theta}U^N_{\theta_t}(x) d\mu(x) \\
        &:= \int_\Omega \partial_u H(U_{\theta_t}^N, Q_{\phi_t}^{N^*})(x) \nabla_{\theta}U^N_{\theta_t}(x) d\mu(x),
    \end{split}
\end{align}
which we clip to give the quantity
\begin{align} \label{actor_gradient}
    \begin{split}
        G_u(\theta_t; \phi_t)&=\int_\Omega  \psi^{N}\paren{\partial_u H(U_{\theta_t}^N, Q_{\phi_t}^{N^*})(x)} \nabla_{\theta}U^N_{\theta_t}(x) d\mu(x).
    \end{split}
\end{align}
On the other hand, the critic parameters $\phi$ update according to a clipped Q-PDE gradient (notice that \eqref{critic_gradient} below is a clipped version of \eqref{qpde})
\begin{align}
    \begin{split}\label{critic_gradient}
    G_q(\theta_t;\phi_t)&=\int_{\Omega}\brac{\paren{\psi^{N^*} \circ {\mathcal{L}^{U_{\theta_t}^N}Q^{N^*}_{\phi_t}}} \cdot  \paren{(\psi^{N^*})' \circ \mathcal{L}^{U_{\theta_t}^N}Q^{N^*}_{\phi_t}}}\nabla_\phi\paren{-Q_{\phi_t}^{N^*}}d\mu \\
    &= \int_{\Omega}F^{N^*}\paren{\mathcal{L}^{U_{\theta_t}^N}Q^{N^*}_{\phi_t}(x)} \nabla_\phi \paren{-Q_{\phi_t}^{N^*}(x)}d\mu(x).
    \end{split}
\end{align}
With these, we set the limiting training dynamics
\begin{align}\label{theta_gradient}
    \frac{d\theta_t}{dt}&= -\alpha_t N^{2\beta-1} G_u(\theta_t; \phi_t)
\end{align}
and
\begin{align}\label{phi_gradient}
    \frac{d \phi_t}{dt}&= -\omega_t(N^*)^{2\beta-1} G_q(\theta_t; \phi_t),
\end{align}
where $\alpha_t, \omega_t \geq 0$ are bounded, continuous, non-negative learning rate schedulers such as $\alpha_t = \frac{1}{1+t}$ or $\omega_t = 1$. The rates may even be defined dynamically, e.g. $$\alpha_t = \frac{1}{1+t+\nor{\g^{U_{\theta_t}^N}Q_{\phi_t}^{N^*}}_{L^2}^2} \qquad \text{ or } \qquad \alpha_t = \frac{1}{1+\nor{\g^{U_{\theta_t}^N}Q_{\phi_t}^{N^*}}_{L^2}},$$ although for simplicity we limit ourselves to schedules dependent only on time.

Since the learning rates can be zero, this framework covers both \textit{online} learning (where the actor and critic in theory train simultaneously, which is implemented in practice by having the actor and critic both update in the same time step, as described above) and \textit{offline} learning (where the actor or critic must learn for a long time or achieve some satisfactorily low loss before switching). The online case occurs when $\alpha_t, \omega_t > 0$ and the offline case when $\alpha_t > 0$ if and only if $\omega_t = 0$.

We now finally define the algorithm as an approximation of the training dynamics \eqref{theta_gradient} and \eqref{phi_gradient} using the discretization scheme discussed above. Note that in practice, \eqref{actor_gradient} and \eqref{critic_gradient} are calculated by Monte Carlo sampling, which is capable of accurately estimating these integrals in very high dimensions while being vastly less computationally expensive than exact evaluation. However, for our theoretical analysis, we will use the exact and continuously-updating gradients.

\begin{algorithm}[H] \label{algo_ac}
\SetAlgoLined
 \textbf{Parameters:} Hyperparameters of the actor and critic networks; Smooth truncation function $\psi^N$; State space $\Omega$ and action space $A$; SDE coefficients $b,\Phi$; Cost function $c$ and interpolated boundary condition $\overline g$; Auxiliary function $\eta$; Reference measure $\mu$; Numbers of sample points $\mathfrak{m}_\text{critic}$ and $\mathfrak{m}_\text{actor}$ from reference measure; Learning rate schedulers $\alpha_t, \omega_t$; Terminal time $\Lambda$; time increment $\Delta t$\;
 \textbf{Initialize:} Neural networks $Q_{\phi_0}^{N^*}$ and $U_{\theta_0}^N$; time $t = 0$\;
 ----------------------------------------------------------- \\
 \textit{While $t < \Lambda$}:\\
  ----------------------------------------------------------- \\
\textbf{Critic step}: \\
\quad Fix actor network $U_{\theta_t}^N$. Take one step towards minimizing the associated PDE loss \eqref{critic_loss} by updating the parameters $\phi_t^{N^*} \leftarrow \phi_{t+\Delta t}^{N^*}$ in accordance with a Monte Carlo estimate of \eqref{phi_gradient} calculated from $\mathfrak{m}_\text{critic}$ randomly sampled points\;
\textbf{Actor step}: \\
\quad Fix critic network $Q_{\phi_{t+\Delta t}}^{N^*}$. Take one step towards minimizing the integral-Hamiltonian \eqref{actor_loss} by updating the parameters $\theta_{t}^N \leftarrow \theta_{t+\Delta t}^N$ in accordance with a Monte Carlo estimate of \eqref{theta_gradient} calculated from $\mathfrak{m}_\text{actor}$ randomly sampled points\;
Update time $t \leftarrow t+\Delta t$\;
\caption{Actor-Critic Algorithm for the HJB Equation}
\end{algorithm}

We remark that while the dynamics of the system correspond to using stochastic gradient descent as an optimizer, Adam is also a good choice for practical implementations.

\subsection{Convergence to limiting training dynamics}
We now turn to the limiting dynamics of $Q_{\phi_t}^{N^*}$ and $U_{\theta_t}^N$ under  the exact gradients \eqref{theta_gradient} and \eqref{phi_gradient} as $N,N^* \to \infty$. In particular, we will show in Theorem \ref{thm_limit_ode} that the pair $(Q_{\phi_t}^{N^*}, U_{\theta_t}^N)$ approaches the solution of an infinite-dimensional ordinary differential equation.

For finite $N, N^*$, we have by the chain rule that for any $y \in \overline\Omega$, 
\begin{align} \begin{split} \label{456}
    \frac{d U_{\theta_t}^N}{dt}(y)&= \frac{d\theta_t}{dt}\cdot \nabla_\theta U_{\theta_t}^N(y)\\
    &= -\alpha_t N^{2\beta-1} \int_\Omega \psi^N\paren{\partial_u H(U_{\theta_t}^N, Q_{\phi_t}^{N^*})(x)} \paren{\nabla_{\theta} U_{\theta_t}^N(x)\cdot \nabla_{\theta} {U_{\theta_t}^N(y)}} d\mu(x)
\end{split}\end{align}
and 
\begin{align}
    \begin{split} \label{457}
    \frac{d Q_{\phi_t}^{N^*}}{dt}(y)&= \dfrac{d\phi_t}{dt}\cdot \nabla_\phi Q_{\phi_t}^{N^*}(y)\\
    &= \omega_t(N^*)^{2\beta-1}\int_{\Omega} F^{N^*}\big(\mathcal{L}^{U_{\theta_t}^N}Q^{N^*}_{\phi_t}(x)\big)\paren{\nabla_\phi Q_{\phi_t}^{N^*}(x)\cdot \nabla_\phi {Q_{\phi_t}^{N^*}(y)}} d\mu(x).
    \end{split}
\end{align}
We show that as $N, N^* \to \infty$, the dynamics approach
\begin{small}
\begin{align} \label{limit_ODE}
\begin{split}
    \frac{d Q_t}{dt} &= \omega_t\mathcal{B}\mathcal{L}^{U_t} Q_t, \\
    \frac{d U_t}{dt} &= -\alpha_t\mathcal{A} \partial_u H(U_t, Q_t), \\
    (Q_0, U_0) &= (\overline{g}, 0),
\end{split}
\end{align}
\end{small}
where $\A$ is the integral operator defined by
\begin{align}\label{def_op_A}
    \A f(y) := \int_\Omega f(x)A(x,y) d\mu(x)
\end{align}
for $f \in \mathcal{D}(\A)$, with
\begin{align} \label{kernel_A_def}
\begin{split}
    A(x,y) :&=\mathbb{E}_{c,w,b}\Big[\nabla_{c,w,b}(c\sigma(w\cdot x+b))\cdot \nabla_{c,w,b}(c\sigma(w\cdot y+b))\Big] \\
    &= \E_{c,w,b} \brac{\sigma(w\cdot x + b)\sigma(w\cdot y + b) + c^2\sigma'(w\cdot x + b)\sigma'(w\cdot y + b) (x\cdot y + 1)},
\end{split}
\end{align}
and similarly $\mathcal{B}$ is the integral operator defined by
\begin{align}\label{kernel_B_op_def}
    \mathcal{B}f(y):=\int_\Omega f(x)B(x,y)d\mu(x)
\end{align}
for $f\in \mathcal{D}(\mathcal{B})$, with 
\begin{align} \label{kernel_B_def}
    B(x,y):=\eta(x)\eta(y)A(x,y).
\end{align}

We are able to show that the limit ODE flow \eqref{limit_ODE} admits a well-defined solution and that the training dynamics converge to it in a suitably strong sense.

\begin{lemma}[Existence and uniqueness of flow]\label{lemma_limit_ode}
    The candidate limit ODE \eqref{limit_ODE} admits a unique solution flow $(Q_t,U_t)_{t\geq 0} \in C([0,\infty); C^2(\overline\Omega) \times C^0(\overline\Omega))$.
\end{lemma}
\begin{proof}
See Section \ref{convergence_proof}.
\end{proof}

\begin{theorem}[Wide limit of the actor-critic pair] \label{thm_limit_ode}
    For any $t \geq 0$, the critic $Q_{\phi_t}^{N^*}$ in \eqref{457} and actor $U_{\theta_t}^N$ in \eqref{456} converge to their limits $Q_t$ and $U_t$, given by \eqref{limit_ODE}, as $N, N^* \to \infty$, in the sense that
    \begin{align}\label{eq_limit_ode_convergence}
        \lim_{N,N^* \to \infty}\mathbb{E}\Big[\|Q_{\phi_t}^{N^*} - Q_t\|_{\mathcal{H}^2} + \|U_{\theta_t}^N-U_t\|_{L^2}\Big]=0.
    \end{align}
\end{theorem}
\begin{proof}
    See Section \ref{convergence_proof}.
\end{proof}

The above theorem demonstrates that the uncertainties typically associated with neural networks (non-convexities in the loss landscape introduced by the network architecture leading to troublesome local minima, random parameter initializations, etc.)~disappear in the infinite-width limit, with instead the dynamics following the explicit infinite-dimensional ODE \eqref{limit_ODE}. We may therefore directly analyze this ODE to understand the convergence properties of Algorithm \ref{algo_ac}, which we undertake in the following subsection.

\subsection{Analysis of the limit training dynamics}
We now move to analyzing the infinite-width dynamics given in \eqref{limit_ODE}. The following verification-type result indicates that, subject to the convexity-like assumption below on the Hamiltonian, the only stationary point (i.e., fixed point) of the limit ODE system \eqref{limit_ODE} is the pairing $(V,u^*)$ of the true value function and an optimal control.

\begin{assumption}[First-order sufficiency for optimization of the Hamiltonian] \label{convex_hamiltonian}
    For all $x \in \Omega$ and functions $Q \in C^2(\Omega)$, the Hamiltonian
    \begin{equation*}
        H(a,Q(x)) = \sum_{i}b_i(x,a)\partial_i Q(x)+\frac{1}{2}\sum_{i,j}{\Phi\Phi^\intercal}_{ij}(x,a)\partial^2_{ij} Q(x) + c(x,a)
    \end{equation*}
    considered as a map $A \to \R$ has the property that
    \begin{equation*}
        \partial_u H(U,Q)(x) := \partial_u H(U(x), Q(x)) = 0 \implies U(x) \text{ minimizes } H(\cdot, Q(x)).
    \end{equation*}
\end{assumption}

\begin{remark}
    The above assumption is satisfied, for example, if for all $x \in \Omega$ and $Q \in C^2$, there exists a strictly increasing function $\kappa$ such that $\kappa(H(\cdot, Q(x))) : A \to \R$ is strictly convex, or $H(\cdot, Q(x)) : A \to \R$ satisfies a Polyak--{\L}ojasiewicz condition.
\end{remark}

\begin{theorem}[Verification]\label{fixed_point_verification}
Suppose there exist functions $\Tilde Q \in C^2(\Omega) \cap C^0(\overline{\Omega})$ and $\Tilde U \in C^0(\overline \Omega)$ such that
\begin{align}\label{FixedPointCondition}\begin{split}
    \mathcal{B}\mathcal{L}^{\Tilde{U}} \Tilde{Q} &= 0, \\
    \A \partial_u H(\Tilde{U}, \Tilde{Q}) &= 0, \\
    \Tilde{Q}\rvert_{\partial\Omega} &= g. 
    \end{split}
\end{align}
Then $\Tilde{Q}$ is the true value function of the control problem
\begin{equation*}
    \Tilde{Q}(x) = V(x) = \inf_{u \in \mathcal{U}} V^u(x), \quad \text{ for all } x \in \overline\Omega,
\end{equation*} and $\Tilde{U}$ minimizes the Hamiltonian
\begin{equation*}
    \Tilde{U}(x) = \arg\min_{a \in A} \curlbrac{\sum_{i}b_i(x,a)\partial_i V(x)+\frac{1}{2}\sum_{i,j}{\Phi\Phi^\intercal}_{ij}(x,a)\partial^2_{ij} V(x) + c(x,a)}, \quad \text{for all } x \in \Omega,
\end{equation*}
and is therefore an optimal control.

In particular, any fixed point $(\tilde{Q}, \tilde{U}) \in C^2(\overline{\Omega}) \times C^0(\overline{\Omega})$ of the ODE flow \eqref{limit_ODE} when the learning rates $\alpha,\omega$ are positive is a solution of the stochastic control problem.
\end{theorem}

\begin{proof}
    It is well known that $\A$, the limiting neural tangent kernel, is positive definite as a self-adjoint $L^2 \to L^2$ operator. Further, \cite[Lemma 35]{JMLR:v24:22-1075} shows that the same is true for $\mathcal{B}$. Thus, $\g^{\tilde U}\tilde Q = 0$ $\mu$-a.e. and $\partial_u H(\tilde{U},\tilde{Q}) = 0$ $\mu$-a.e. By the regularity of the operators and of $\tilde{Q}$ and $\tilde{U}$, these equalities must be true pointwise everywhere.
    
   Since $\g^{\tilde U}\tilde Q = 0$ and $\Tilde{Q}\rvert_{\partial\Omega} = g$, it follows by the Feynman--Kac theorem that $\Tilde{Q} = V^{\tilde U}$, the value function of the process when controlled by $\tilde U$.
    
    We now aim to show that $\tilde U$ is optimal. By Assumption \ref{convex_hamiltonian}, that $\partial_u H(\tilde U, \tilde Q) = 0$ implies $\tilde U$ is a pointwise minimizer of $H(\cdot, \Tilde{Q})(x)$. That is to say, for each $x \in \Omega$,
    \begin{align*}
        0 &= \sum_{i}b_i(x,\Tilde{U}(x))\partial_i 
        \Tilde{Q}(x)+\frac{1}{2}\sum_{i,j}{\Phi\Phi^\intercal}_{ij}(x,\Tilde{U}(x))\partial^2_{ij} \Tilde{Q}(x) + c(x,\Tilde{U}(x)) - \gamma \Tilde{Q}(x) \\
        &= \inf_{a \in A} \curlbrac{\sum_{i}b_i(x,a)\partial_i 
        \Tilde{Q}(x)+\frac{1}{2}\sum_{i,j}{\Phi\Phi^\intercal}_{ij}(x, a) \partial^2_{ij} \Tilde{Q}(x) + c(x,a)} - \gamma \Tilde{Q}(x).
    \end{align*}
    This, combined with the fact that $\Tilde{Q}\rvert_{\partial\Omega} = g$, implies by the verification theorem (see, e.g.~\cite{pham2009continuous}) that $\Tilde{Q} = V$, the true value function. Of course, the control $\Tilde{U}$ achieves $V^{\Tilde{U}} = \Tilde{Q} = V$ and so is optimal.
\end{proof}

Assuming that the learning rates $\alpha, \omega$ are positive, this result shows that any fixed point of the limiting flow is a solution of the stochastic control problem. Consequently, if one can show that the infinite-width limit ODE flow \eqref{limit_ODE} converges to a fixed point, then that result combined with Theorems \ref{thm_limit_ode} and \ref{fixed_point_verification} would constitute a proof of the convergence of Algorithm \ref{algo_ac}. More precisely, if one establishes that the limiting flow $(Q_t,U_t)_{t\geq 0}$ converges in $(C^2(\Omega) \cap C^0(\overline{\Omega}))\times C^0(\overline{\Omega})$, that is, in supremum norm, and the learning rates $\alpha,\omega$ are not integrable with respect to time, then it is straightforward to show that this limit of $(Q_t,U_t)_{t\geq 0}$ satisfies (\ref{FixedPointCondition}), and is thus a solution of the stochastic control problem. 

Unfortunately, though in practice numerical convergence to a fixed point reliably occurs, proving it remains an open problem. Few theoretical tools are available for infinite-dimensional ODEs -- especially ones as nonlinear as \eqref{limit_ODE}. We suspect that a proof of convergence to a fixed point in an appropriate topology is possible under some additional regularity conditions and the assumption that $\{\g^u\}_{u\in\mathcal{U}}$ is uniformly strongly negatively monotone in $L^2$.

\section{Numerical Analysis} \label{numerics}
We begin this section with a short proposition that reduces the computational cost of calculating the PDE operator $\g^{U_{\theta_t}^N} Q_{\phi_t}^{N^*}$ needed in \eqref{critic_gradient}. This is adapted from \cite[Prop. 2.2]{cheridito2025} and \cite[Prop.~1]{duarte2024}.
\begin{proposition}
    For some $Q \in C^2(\Omega)$, $x \in \Omega$, and $a \in A$, consider the function
    \begin{equation}
        \psi(h; Q, x, a) := \sum_{i=1}^{d'} Q\paren{x + \frac{h}{\sqrt 2}\Phi_{\star,i}(x,a) + \frac{h^2}{2{d'}} b(x,a)}.
    \end{equation}
    Then
    \begin{align*}
        \psi''(0; Q, x, a) &= \sum_{i=1}^d b_i(x,a) \partial_iQ(x) + \frac{1}{2}\sum_{i,j=1}^d (\Phi\Phi^\intercal)_{ij}(x,a) \partial_{ij} Q(x) \\
        &= \inprod{b(x,a)}{\nabla Q(x)} + \frac{1}{2}\mathrm{Tr}\paren{\Phi\Phi^\intercal(x,a) \mathrm{Hess}Q(x)}.
    \end{align*}
\end{proposition}
It follows that for an actor $U_{\theta_t}^N : \Omega \to A$, we have
\begin{equation*}
    \g^{U_{\theta_t}^N}Q_{\phi_t}^{N^*}(x) = \psi''(0; Q,x,U_\theta^N(x)) + c(x, U_\theta^N(x)) - \gamma Q_\phi^{N^*}(x)
\end{equation*}
at each sampled point $x \in \Omega$. Hence, we have reduced the cost of computing the PDE operator (and thus the critic update gradient \eqref{critic_gradient}) by alleviating the need to calculate the $d\times d$ matrix of second derivatives $\mathrm{Hess} Q_{\phi_t}^{N^*}$, instead using only the second derivative of $\psi$, which is the sum of $d'$ univariate functions. When $d$ is large, the resulting computational savings can be very significant. Note also that taking the derivative in $u$ of $\g^{U_{\theta_t}^N}Q_{\phi_t}^{N^*}(x)$ gives an inexpensive way to calculate $\partial_uH(U_{\theta_t}^N, Q_{\phi_t}^{N^*})$ for the actor update gradient \eqref{actor_gradient}.

\subsection{Linear-Quadratic Regulator}
We now numerically evaluate the performance of Algorithm \ref{algo_ac} for solving a variant of the Linear-Quadratic Regulator (LQR) stochastic control problem. Specifically, we investigate the HJB equation
\begin{align}\label{lqr}
\begin{split}
    \inf_{a \in \R^d}\curlbrac{\sum_{i=1}^d\brac{\partial_i^2 V(x) (1 + \epsilon x_ia_i)^2 + \xi\partial_iV(x)a_i} + q\nor{a}^2 + \tilde{f}(x)} - \gamma V(x) = 0, \,\,\,\, x &\in \Omega \\
    V(x) = kR^2, \,\,\,\, x &\in \partial\Omega 
\end{split}
\end{align}
where $\Omega = B(0,R) \subset \R^d$, $A = \R^d$, $p,q,\xi, R$ are positive constants, $\epsilon$ is a constant of either sign, $$k = \frac{\sqrt{q^2\gamma^2 + 4pq\xi^2} - \gamma q}{2\xi^2},$$ and $$\tilde{f}(x) = \gamma k \nor{x}^2 + \sum_{i=1}^d \frac{k^2(\xi + 2\epsilon)^2x_i^2}{q+2k\epsilon^2x_i^2} - 2kd.$$ This HJB equation arises from the controlled SDE
\begin{equation}\label{lqr_sde}
    dX_t = \xi u_t dt + \Phi(X_t,u_t) dW_t
\end{equation}
where $\Phi(x,a)$ is a diagonal matrix with $i$-th element $\sqrt{2}(1+\epsilon x_ia_i)$, and the cost functional to minimize is
\begin{equation*}
    V^u(x) = \E\brac{\int_0^\tau e^{-\gamma t}(q\nor{u_t}^2+\tilde{f}(X_t))dt + kR^2e^{-\gamma\tau} \Big\rvert X_0 = x}.
\end{equation*}
The choice of $k$ is such that the true value function is simply $$V(x) = k\nor{x}^2,$$ and the corresponding optimal control is
\begin{equation*}
    u_i^*(x) = \frac{-x_ik(\xi+2\epsilon)}{q + 2k\epsilon^2x_i^2}.
\end{equation*}
We have reproduced this problem exactly from \cite{zhou2021actor}, allowing for a direct comparison against their numerical method. That the given $V$ and $u^*$ are indeed the value function and optimal control can be verified by treating the pair as an ansatz. Note that the HJB equation \eqref{lqr} is fully nonlinear for any $\epsilon \neq 0$, making it a somewhat challenging example problem. We tackle it using a PyTorch implementation of Algorithm \ref{algo_ac}. The code is publicly available,\footnote{\url{https://github.com/JacksonHebner/Actor-Critic-Method-for-HJB-Code}} compatible with CUDA, and relatively computationally light.

The parameters are set as $p = q = \xi = \gamma = R = 1$ and $\epsilon = -1$, the same as in \cite{zhou2021actor}. We deviate slightly from the literal description of Algorithm \ref{algo_ac} in that we fix $\psi^N = \mathrm{Id}_{\R}$ and use Adam (rather than stochastic gradient descent) as our optimizer. The learning schedulers $\alpha_t, \omega_t$ are such that the actor trains for 200 steps, after which the critic trains for 100, and then the process repeats a total of 50 times. The learning rates are $\frac{0.01}{1+\mathfrak{n}}$, where $\mathfrak{n}$ is the number of times this process has repeated. Each individual training step involves computing the integrals in \eqref{theta_gradient} and \eqref{phi_gradient} via Monte Carlo with 3,000 randomly sampled points.

We first demonstrate the algorithm for the control problem in 10, 20, 50, 100, and 200 dimensions when sampling training points uniformly in $B(0,R)$, i.e.~$\mu = \mathrm{Leb}\rvert_{B(0,R)}$. The actor and critic neural networks each have 512 neurons in their single hidden layer. We use two metrics to evaluate training performance against the known analytic solution. The first is mean square error, which is simply
\begin{align*}
    \textrm{MSE}_K(Q, V) &= \frac{1}{K}\sum_{j=1}^K(Q(x_j) - V(x_j))^2, \\
    \textrm{MSE}_K(U, u^*) &= \frac{1}{K}\sum_{j=1}^K \nor{U(x_j) - u^*(x_j)}^2,
\end{align*}
for $K  = 10^6$ points in $B(0,R)$ sampled from $\mu$ independent of training. The second is relative error, which we define as
\begin{align*}
    \textrm{RE}_K(Q, V) &= \frac{\sum_{j=1}^K(Q(x_j) - V(x_j))^2}{\sum_{j=1}^KV(x_j)^2}, \\
    \textrm{RE}_K(U, u^*) &= \frac{\sum_{j=1}^K\nor{U(x_j) - u^*(x_j)}^2}{\sum_{j=1}^K\nor{u^*(x_j)}^2}
\end{align*}
for the same $K$ points.

\begin{table}[]
\centering
\begin{tabular}{@{}llllll@{}}
\toprule
Dimension                & 10 & 20 & 50 & 100 & 200 \\ \midrule
Critic mean square error & $2.37 \times 10^{-7}$   & $2.41\times 10^{-7}$   & $1.26 \times 10^{-7}$ & $4.67 \times 10^{-8}$ & $1.38 \times 10^{-8}$ \\
Critic relative error    & $8.69 \times 10^{-7}$ & $7.57\times 10^{-7}$   & $3.56 \times 10^{-7}$ & $1.27 \times 10^{-7}$ & $3.70 \times 10^{-8}$  \\
Actor mean square error  & $1.28 \times 10^{-4}$ & $9.71 \times 10^{-4}$   &  $1.79 \times 10^{-3}$ & $2.76 \times 10^{-3}$ & $5.97 \times 10^{-3}$ \\
Actor relative error     & $6.12 \times 10^{-4}$   &  $3.63\times 10^{-3}$  &  $5.53 \times 10^{-3}$ & $7.89 \times 10^{-3}$ & $1.64 \times 10^{-2}$ \\ \bottomrule
\end{tabular}
\caption{The performance of Algorithm \ref{algo_ac} on the LQR problem in various dimensions. The hyperparameters are constant across dimension.} \label{data_full_accuracy}
\end{table}

\begin{table}[]
\centering
\begin{tabular}{@{}lllll@{}}
\toprule
Dimension, Method              & 10d, this paper & 10d, \cite{zhou2021actor} & 20d, this paper & 20d, \cite{zhou2021actor} \\ \midrule
Critic mean square error & $2.37 \times 10^{-7}$   & Not provided   & $2.41 \times 10^{-7}$ & Not provided \\
Critic relative error    & $8.69 \times 10^{-7}$ & $2.25\times 10^{-4}$   & $7.57 \times 10^{-7}$ & $3.84 \times 10^{-4}$ \\
Actor mean square error  & $1.28 \times 10^{-4}$ & Not provided   &  $9.71 \times 10^{-4}$ & Not provided \\
Actor relative error     & $6.12 \times 10^{-4}$ & $2.45\times 10^{-3}$  &  $3.63 \times 10^{-3}$ & $2.76 \times 10^{-3}$ \\ \bottomrule
\end{tabular}
\caption{A comparison of the method in this paper against the method in \cite{zhou2021actor} in 10 and 20 dimensions.} \label{data_Duke_comparison}
\end{table}
The overall training performance, shown in Table \ref{data_full_accuracy}, is remarkable, especially for the critic. The actor is of course at a natural disadvantage when trying to learn a $d$-dimensional function compared to the critic which learns only one dimension. Note that the relative and mean square error metrics are both likely affected by the fact that the ``average" point in the unit ball is increasingly far away from the origin in high-dimensional space. The critic tends to perform better near the boundary (where the critic value is fixed to obey a known condition), so this may explain the otherwise counterintuitive drop in critic mean square error as dimension increases. The actor has no such special help near the boundary. Further, the actor is much more likely to be inherently limited in high dimensions by an irreducible neural network representation error (representing an $\mathbb{R}^{200} \to \mathbb{R}^{200}$ function by a network with only 512 hidden neurons is difficult). Nonetheless, we keep the neural network hyperparameters constant across Table \ref{data_full_accuracy} so as to give the most straightforward set of results.

Table \ref{data_Duke_comparison} compares the performance of Algorithm \ref{algo_ac} against that of the method in \cite{zhou2021actor} whence we have reproduced this example problem. Algorithm \ref{algo_ac} gives more accurate solutions. This is especially impressive because \cite{zhou2021actor} uses deeper neural networks (three hidden layers, each of 200 neurons) compared to ours (one hidden layer of 512 neurons), resulting in many more parameters to train, and trains these parameters longer ($\sim 3\times$ more training epochs).

\begin{remark}[Pathwise sampling]
We add that Algorithm \ref{algo_ac} can be easily modified to sample training points from the controlled SDE \eqref{lqr_sde} rather than from a fixed measure $\mu$. This technique would, of course, violate the assumptions that lead to the limiting ODE dynamics \eqref{limit_ODE}, but could still be of practical use when solving high-dimensional problems. Specifically, if we are only interested in solving the stochastic control problem when starting from a specific region $\Omega_{\mathrm{start}} \subsetneq \Omega$ of the domain, we could fix a number of controlled SDE paths that start in $\Omega_{\mathrm{start}}$ and are controlled by the current actor. In the case of our LQR problem, this would give a collection of $K$ particles $\{X_t^k\}_{k=1}^K$, each with dynamics
\begin{equation*}
        dX^k_t = \xi U_\theta^N(X^k_t) dt + \Phi(X_t,U_\theta^N(X^k_t)) dW_t, \quad X_0^k \in \Omega_{\mathrm{start}}.
\end{equation*}
These $K$ particles can then be used as points for estimating $L_\mathrm{critic}(\phi; \theta)$ and $L_\mathrm{actor}(\theta;\phi)$ by numerical integration. Whenever a particle exits the domain, we may restart it again at a point in $\Omega_{\mathrm{start}}.$

In many stochastic control problems, large areas of the domain $\Omega$ are almost never encountered by the controlled SDE when the control is reasonably close to optimal and the path is initiated far away. This SDE sampling technique avoids spending computation time training the actor and critic in those regions. However, by implementing this computational optimization, we limit the ability of the neural network to learn robustly and add potential pitfalls. For example, if early in training the critic outputs erroneously large values in a certain region of the domain, then the actor will learn to discourage the particles from visiting there, and thus the algorithm is unlikely to learn that the high critic values were erroneous in the first place.
\end{remark}

\subsection{Constructed control problems}

The above section has shown that Algorithm \ref{algo_ac} is effective at solving linear-quadratic regulator problems. While we have chosen a difficult LQR example (one where the actor is high-dimensional and the HJB equation is fully nonlinear), we remark that the linear-quadratic framework is still overall an easy and conveniently structured family of stochastic control problems. Thus, to more rigorously test the algorithm, we construct here several control problems at various grades of difficulty with closed-form solutions. More specifically, we fix a target value function and optimal control and reverse engineer an HJB equation with the desired behavior that is solved by them. This allows us to test situations in which some difficult properties are present. For example, we can construct the Hamiltonian so that it has many pointwise local minima that are not global minima (that is to say, violates Assumption \ref{convex_hamiltonian}).

Recall the HJB equation under consideration is
\begin{align}\label{hjb-again}
    \begin{split}
    \inf_{a \in A} \Big \{ \sum_{i}b_i(x,a)\partial_i V(x)+\frac{1}{2}\sum_{i,j}{\Phi\Phi^\intercal}_{ij}(x,a)\partial^2_{ij} V(x) + c(x,a) \Big\}- \gamma V(x)=0,\,\,\,\, &x  \in \Omega,\\
    V(x)=g(x),\,\,\,\, &x \in \partial\Omega.
    \end{split}
\end{align}
Suppose we select a true value function $V \in C^2(\Omega) \cap C^0(\overline{\Omega})$, an optimal feedback control $u^* : \Omega \to A$, a drift $b : \Omega \times A \to \mathbb{R}^d$, and a diffusion $\Phi: \Omega \times A \to \mathbb{R}^{d\times d'}$. We then see that the cost function $c : \Omega \times A \to \mathbb{R}$ given by
\begin{align}
    c(x,a) := \zeta(x,a) + \gamma V(x) - \sum_i b_i(x,a)\partial_iV(x) - \frac{1}{2}\sum_{i,j}\Phi\Phi^\intercal_{ij}(x,a) \partial_{ij}^2V(x),
\end{align}
where $\zeta: \Omega \times A \to [0,+\infty)$ is such that $\zeta(x,a) = 0$ if and only if $a = u^*(x)$, solves \eqref{hjb-again}. Then, by the verification theorem, $V(x)$ is the value of the optimally-controlled diffusion\footnote{To be rigorous, we must check that the technical assumptions of the verification theorem are satisfied. This is almost always the case, though, since $\Omega$ is compact. For example, the value function $V$ is of quadratic growth (and in fact bounded) due to being continuous on a compact set. Similarly, $X^{u^*}_t$ has a unique strong solution for a very wide range of $u^*, b, \Phi$, in part due to the compactness of $\Omega$. To show that $\lim_{t\to\infty} e^{-\gamma t}\E[V(X_t^{u^*})] = 0$, it is sufficient for $\gamma > 0$.}
\begin{equation}
    dX^{u^*}_t = b(X_t, u^*(X_t^{u^*}))dt + \Phi(X_t, u^*(X_t^{u^*}))dW_t, \quad X^{u^*}_0 = x.
\end{equation}
That is to say,
\begin{align*}
    V(x) &= \inf_{u \in \mathcal{U}}\E \brac{\int_0^{\tau_u} c(X_s^u, u_s) e^{-\gamma s} ds + g(X_{\tau_u}^u)e^{-\gamma\tau_u} \Big\rvert X_0^u = x} \\
    &= \E \brac{\int_0^{\tau_{u^*}} c(X_s^{u^*}, u^*_s) e^{-\gamma s} ds + g(X_{\tau_{u^*}}^{u^*})e^{-\gamma\tau_u^*} \Big\rvert X_0^{u^*} = x}.
\end{align*}

This is a very general framework that naturally accommodates every stochastic control problem in the form of \eqref{hjb-again} which admits a classical solution $V \in C^2(\Omega) \cap C^0(\overline \Omega)$ and has a unique optimal control $u^*$.\footnote{We remark that this can be trivially extended to finite time horizon problems by replacing the $\gamma V(x)$ term in the definition of $c(x,a)$ with $-\frac{\partial}{\partial t}V(t,x)$.} For example, one can verify that the linear-quadratic regulator problem considered previously can be recreated by fixing the same functions $V, \Phi, b$ and setting $$\zeta(x,a)=\sum_{i=1}^d\big(q+2k\epsilon^2 x_i^2\big)\big(a_i-u_i^*(x)\big)^2,$$
which is indeed non-negative and has the property that $\zeta(x,a) = 0$ if and only if $a = u^*(x).$

For all of the problems below, the boundary condition $g : \partial \Omega \to \R$ is constant and so we select $\overline{g} : \overline{\Omega} \to \R$ to be constant as well. For ball domains $\Omega = B(0,R)$, we set $\eta(x) = R^2 - \nor{x}^2$, and for square domains $\Omega = [-R,R]^d$, we set $\eta(x) = \prod_{i=1}^d (R^2-x_i^2).$ The reference measure is always $\mu = \mathrm{Leb}\rvert_{\Omega}$. Learning occurs for 4000 steps, where actor and critic learning alternate after each step and are at rates
\begin{equation*}
    \alpha_t = \frac{0.0025}{(10 + t)^{0.6}} \quad \text{and} \quad \omega_t = \frac{0.0025}{(10+t)^{0.8}}.
\end{equation*}
Like before, each neural network has a single hidden layer with 512 neurons.

\subsubsection{Problem 1 (easy test case)}

We consider the setup where the domain is $\Omega = B(0,1) \subset \R^{10}$, the action space is $A = \R^{10},$ the dimension of the Brownian motion is $d' = 10$, the discounting rate is $\gamma = 1$, and
\begin{align*}
    V(x) &= \exp\paren{-\nor{x}^2},\\
    u^*(x) &= x, \\
    b_i(x,a) &= a_i x_i, \\
    \Phi(x,a) &= \mathrm{Id_{10\times 10}},\\
    \zeta(x,a) &= \nor{x - a}_{L^1} + \nor{x - a}_{L^2}^2.
\end{align*}
Algorithm \ref{algo_ac} learns well under these circumstances. Within 2000 training epochs, the mean square critic error stabilizes at about $7.9 \times 10^{-5}$ and the mean square actor error at about $4.5 \times 10^{-4}.$ These are likely very close to the irreducible neural network representation errors for this problem (again, we are using a single-layer network with 512 hidden neurons).

This ability to learn scales well to higher dimensions. Table \ref{easy_table} shows the algorithm's performance in 10, 20, 50, and 100 dimensions. The actor and critic loss curves in higher dimensions look almost identical to those in Figure \ref{easy}, with rapid and stable actor convergence and noisy critic convergence that is initially non-monotonic.

\begin{figure}
\centering
\begin{subfigure}[b]{0.49\textwidth}\centering
   \includegraphics[width = \linewidth]{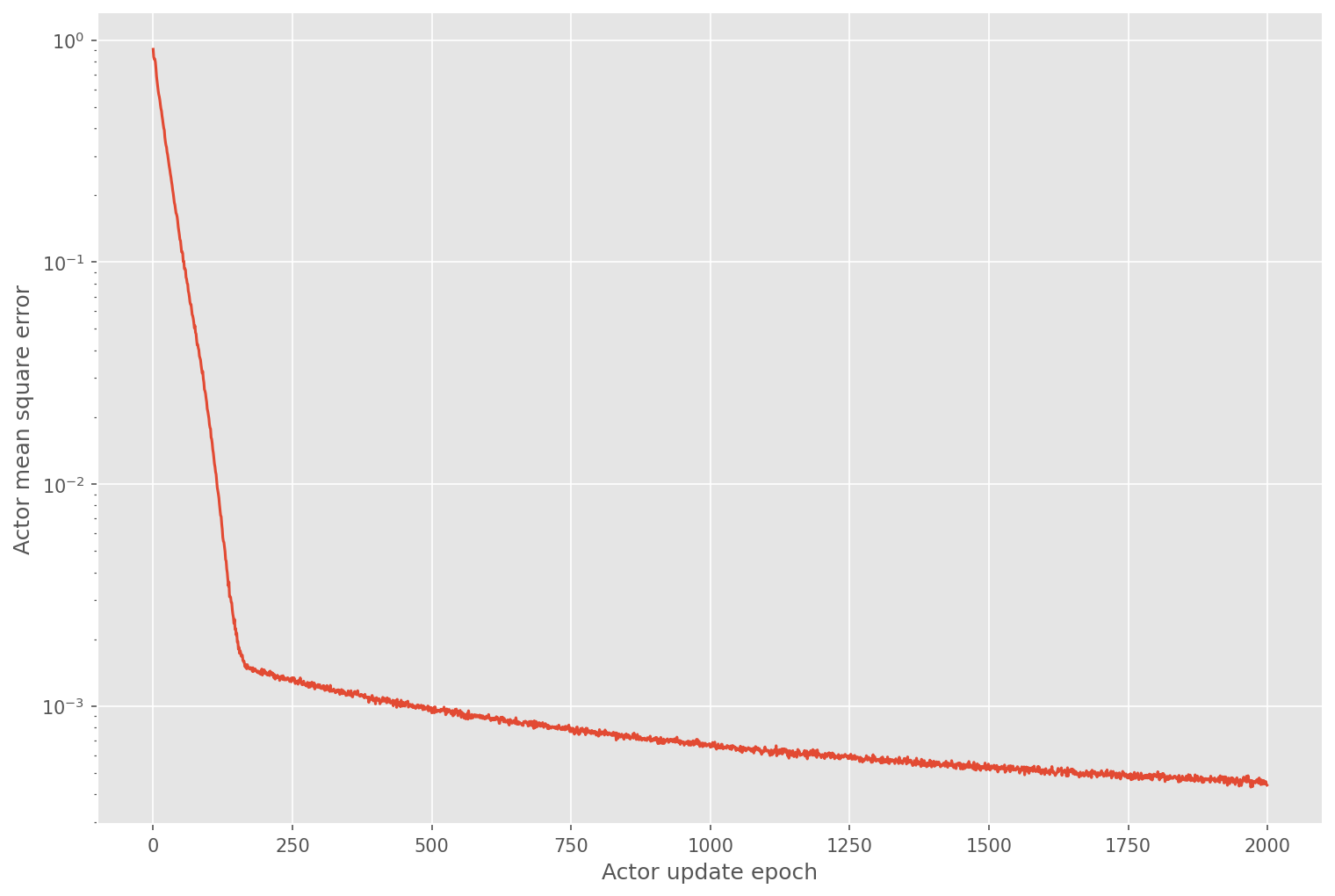}
   \caption{Actor mean square error}
\end{subfigure}%
\begin{subfigure}[b]{0.49\textwidth}\centering
   \includegraphics[width=\linewidth]{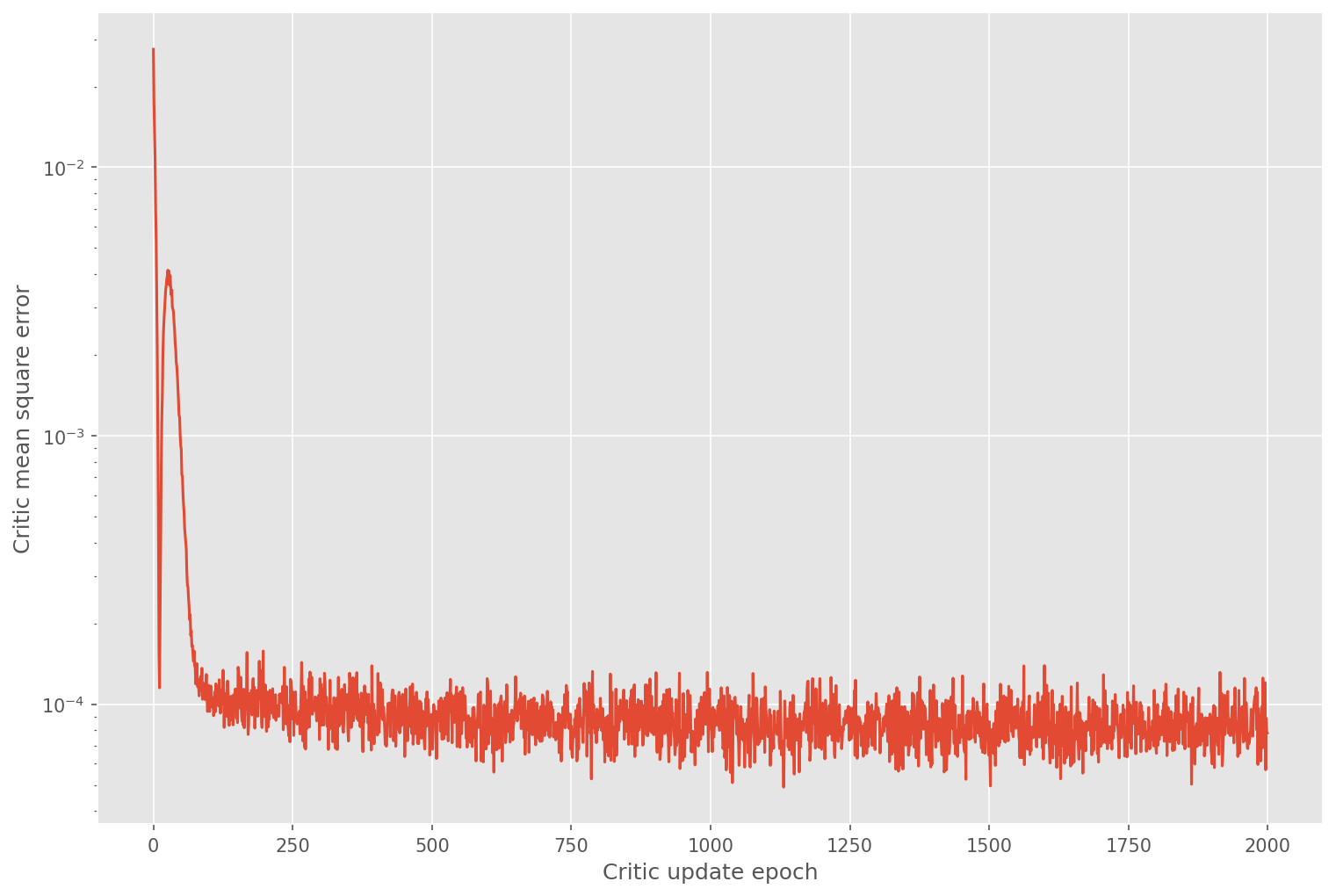}
   \caption{Critic mean square error}
\end{subfigure}
\caption{The actor and critic approach their true values  quickly in Problem 1.}
\label{easy}
\end{figure}

\begin{table}[]
\centering
\begin{tabular}{@{}lllll@{}}
\toprule
Dimension                & 10 & 20 & 50 & 100 \\ \midrule
Critic mean square error & $7.9\times 10^{-5}$  & $9.6 \times 10^{-6}$  & $1.4 \times 10^{-6}$  & $7.3 \times 10^{-7}$ \\
Critic relative error  & $4.1 \times 10^{-4}$  & $5.8 \times 10^{-5}$  & $9.6\times 10^{-6}$  & $5.2 \times 10^{-6}$ \\
Actor mean square error  & $4.5 \times 10^{-4}$ & $1.6 \times 10^{-3} $& $5.8 \times 10^{-3}$ & $1.4 \times 10^{-2}$ \\
Actor relative error  & $5.4 \times 10^{-4}$  & $1.8 \times 10^{-3}$  & $6.0 \times 10^{-3}$  & $1.4 \times 10^{-2}$
\\ \bottomrule
\end{tabular}
\caption{Critic and actor mean square errors for Problem 1 in various dimensions}
\label{easy_table}
\end{table}

\subsubsection{Problem 2A (non-convex Hamiltonian)}
We consider the setup where the domain is $\Omega = B(0,1)\subset \R^{10}$, the action space is $A = \R$, the dimension of the Brownian motion is $d' = 1$, the discounting rate is $\gamma = 1$, and
\begin{align*}
    V(x) &= \sin\paren{\nor{x}^2},\\
    u^*(x) &= \log\sum_{i}\exp(x_i), \\
    b_i(x,a) &= a\sin(x_i), \\
    \Phi_i(x,a) &= 1 + a^2 + x_i^2, \\
    \zeta(x,a) &= \log\cosh(a-u^*(x)).
\end{align*}

The main difficulty in this problem comes from the fact that the pointwise Hamiltonian does not satisfy Assumption \ref{convex_hamiltonian}. Notice that there exist critic functions $Q_\phi^{N^*} \in C^2(\Omega)$ and points $x \in \Omega$ such that
\begin{align*}
    H(\cdot, Q_\phi^{N^*})(x) = \inprod{b(x,\cdot)}{\nabla Q_\phi^{N^*}(x) - \nabla V(x)} + \frac{1}{2}\mathrm{Tr}(\Phi\Phi^\intercal(x,\cdot)\mathrm{Hess}(Q_\phi^{N^*}(x)-V(x)))+ \zeta(x, \cdot)
\end{align*}
is not quasi-convex as a map $A\to\R$ at each $x\in \Omega$, and in particular is not necessarily minimized if the first derivative is 0. Thus, Theorem \ref{fixed_point_verification} may not hold, and there exists the possibility of ``false" fixed points in the training landscape. This means that the actor-critic pair may converge to a solution far away from the true optimal control-value function pair and get stuck there, never learning the true solution to the control problem.

In fact, this is exactly what happens. The critic converges quickly, and typically to something close to the true value $V$, but the actor gets stuck at a local minimum far away from the true optimal control. In our runs, the critic mean square error usually stabilizes at around $5.1 \times 10^{-4}$ and the actor mean square error at around $3.25$. If we replace $\zeta$ with $$\zeta^*(x,a) = 100\log\cosh(a-u^*(x)),$$ then the fixed point is closer to the true solution (with an actor MSE of $0.131$, RE $2.39 \times 10^{-2}$, and critic MSE of $3.08\times 10^{-4}$, RE $5.65 \times 10^{-4}$), but is still clearly distinct from the true optimal control. Further, as can be seen in Figure \ref{100zeta}, the convergence to this fixed point is heavily non-monotonic. Even being extremely close to the true solution (as the actor-critic pair is for Problem 2A at about epoch 55) is evidently no guarantee that the actor-critic pair will not eventually converge towards a false solution much farther away if Assumption \ref{convex_hamiltonian} is violated. We remark that the oscillations in Figure \ref{100zeta} appear in many hyperparameter regimes and without regard to the choice of stochastic gradient descent or Adam as an optimizer. For ease of understanding, we also include in Figure \ref{100zeta} the loss functions
\begin{align*}
    L_\mathrm{actor}(\theta; \phi) &= \int_\Omega H(U_\theta^N, Q_\phi^{N^*})
    (x)d\mu(x), \\
    L_\mathrm{critic}(\phi; \theta) &= \int_\Omega \modu{\mathcal{L}^{U_\theta^N} Q_\phi^{N^*}(x)}^2 d\mu(x).
\end{align*}

\begin{figure}
\centering
\begin{subfigure}{0.49\textwidth}
   \includegraphics[width=1\linewidth]{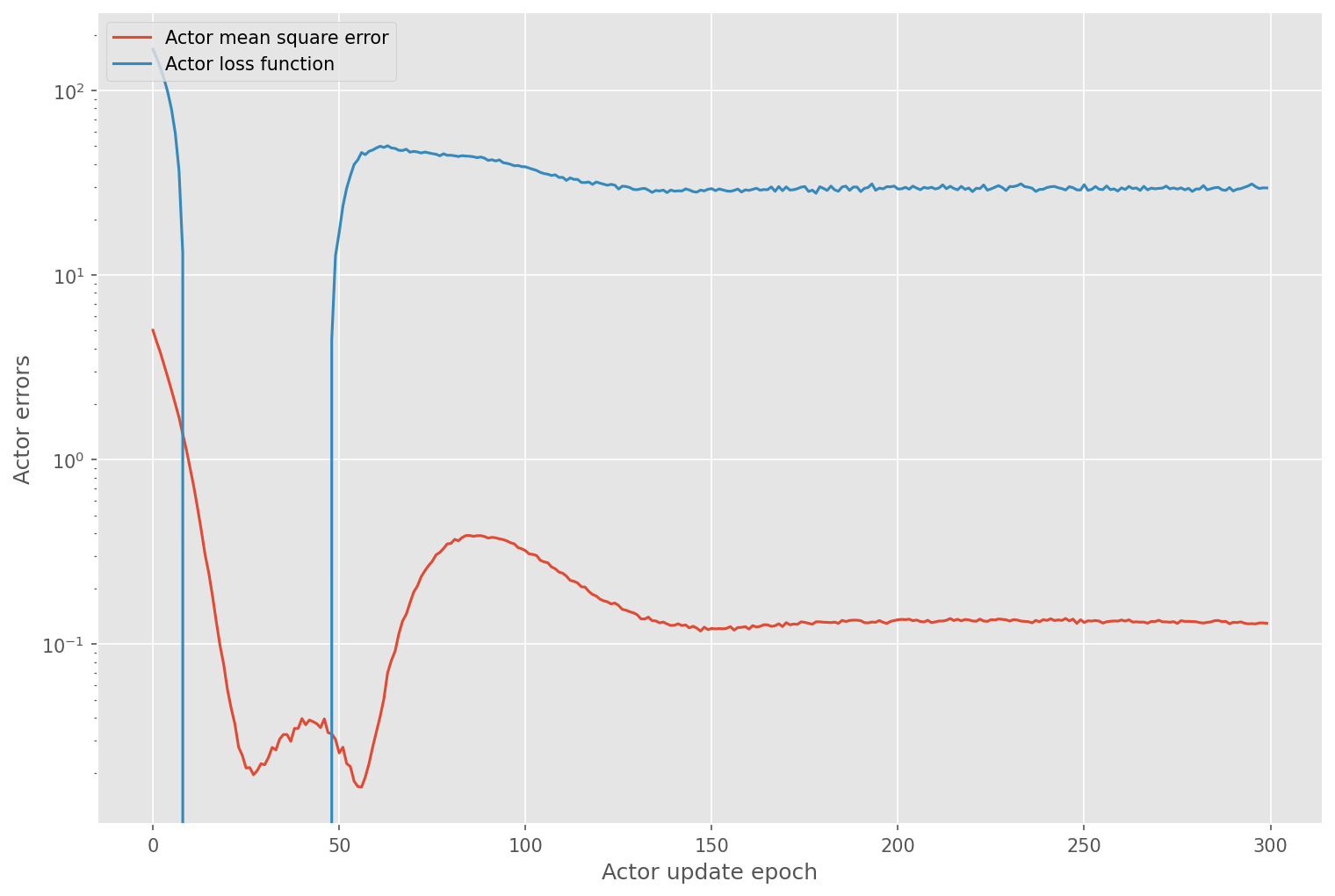}
   \caption{The actor comes close to the optimal control before being drawn to and caught upon a stable local minimum}
\end{subfigure}
\begin{subfigure}{0.49\textwidth}
   \includegraphics[width=1\linewidth]{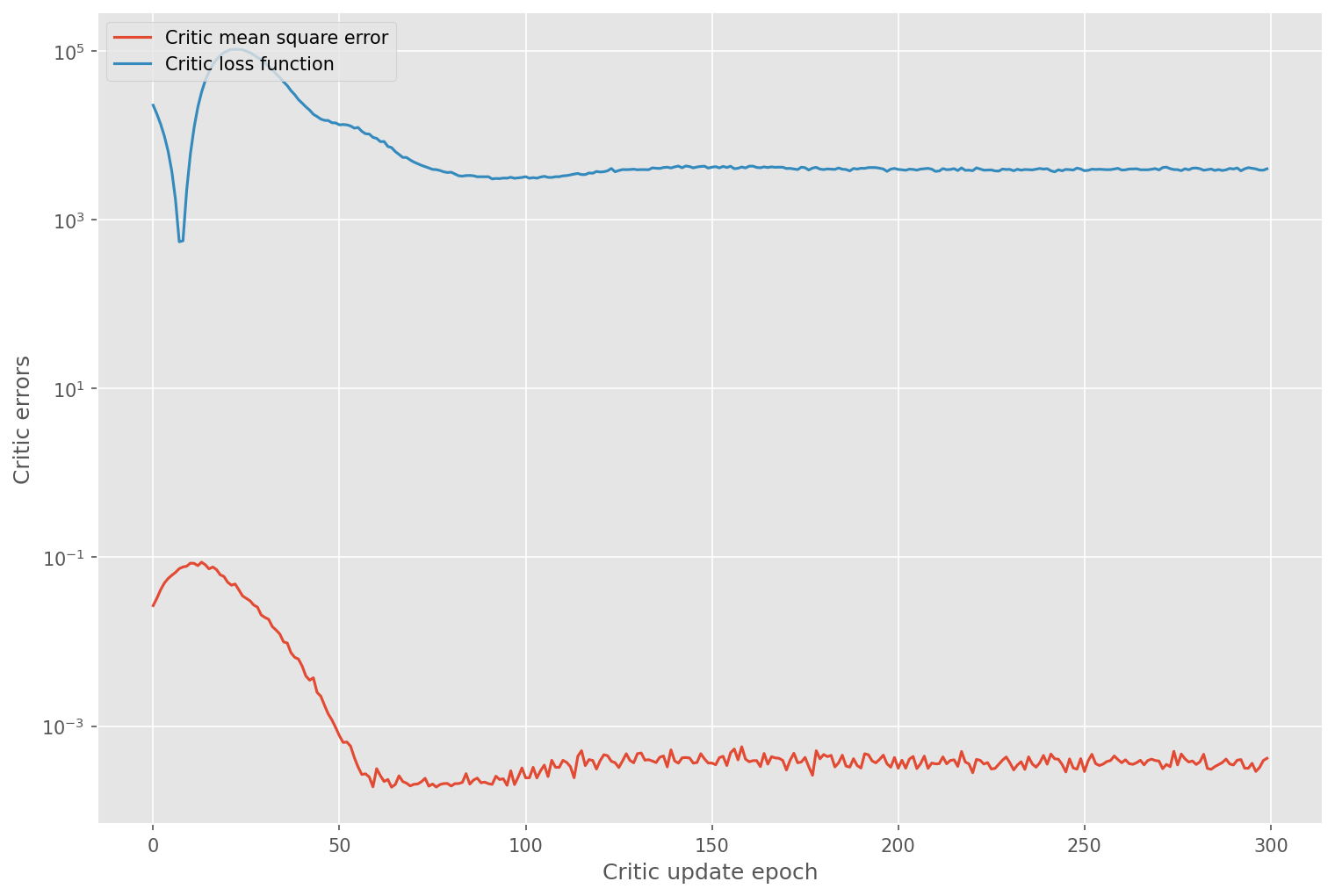}
   \caption{The critic accuracy also oscillates, though less than the actor}
\end{subfigure}

\caption{Problem 2A when $\zeta$ is replaced with $\zeta^*(x,a) = 100\log \cosh(a - u^*(x))$ and $\Omega = B(0,1)$.}
\label{100zeta}
\end{figure}

\subsubsection{Problem 2B (convex Hamiltonian)}
This is the same problem as above, but with $\Phi_i(x,a) = 1 + x_i^2$. That is to say, the domain is $\Omega = B(0,1)\subset \R^{10}$, the action space is $A = \R$, the dimension of the Brownian motion is $d' = 1$, the discounting rate is $\gamma = 1$, and
\begin{align*}
    V(x) &= \sin\paren{\nor{x}^2},\\
    u^*(x) &= \log\sum_{i}\exp(x_i), \\
    b_i(x,a) &= a\sin(x_i), \\
    \Phi_i(x,a) &= 1 + x_i^2, \\
    \zeta(x,a) &= \log\cosh(a-u^*(x)).
\end{align*}
With this change, Algorithm \ref{algo_ac} learns very well. The critic mean square error stabilizes at around $1.13 \times 10^{-4}$ (relative error $2.07 \times 10^{-4}$) and the actor mean square error at around $2.29 \times 10^{-4}$ (relative error $4.17 \times 10^{-5}$). The convergence to this solution is rapid and stable. The reason for this drastic increase in performance is the newfound convexity of the Hamiltonian in the control. Notice that since $b(x,\cdot)$ is affine and $\Phi\Phi^\intercal(x,\cdot)$ is constant for each $x\in\Omega$, the function
\begin{align*}
    H(\cdot, Q_\phi^{N^*})(x) = \inprod{b(x,\cdot)}{\nabla Q_\phi^{N^*}(x) - \nabla V(x)} + \frac{1}{2}\mathrm{Tr}(\Phi\Phi^\intercal(x,\cdot)\mathrm{Hess}(Q_\phi^{N^*}(x)-V(x)))+ \zeta(x, \cdot)
\end{align*}
is convex as a map $A\to\R$ at each $x\in \Omega$ for any critic function $Q_\phi^{N^*} \in C^2(\Omega).$ Thus, Assumption \ref{convex_hamiltonian} and the resulting Theorem \ref{fixed_point_verification} hold. There are no ``false" fixed points in the training landscape, and thus no possibility of the actor and critic converging towards a poor solution as in Problem 2A. In Figure \ref{problem_2B}, we see a steady, nearly monotonic convergence to the sole fixed point in the training landscape (which is, of course, the true solution $(V, u^*)$).

\begin{figure}
\centering
\begin{subfigure}{0.49\textwidth}
   \includegraphics[width=1\linewidth]{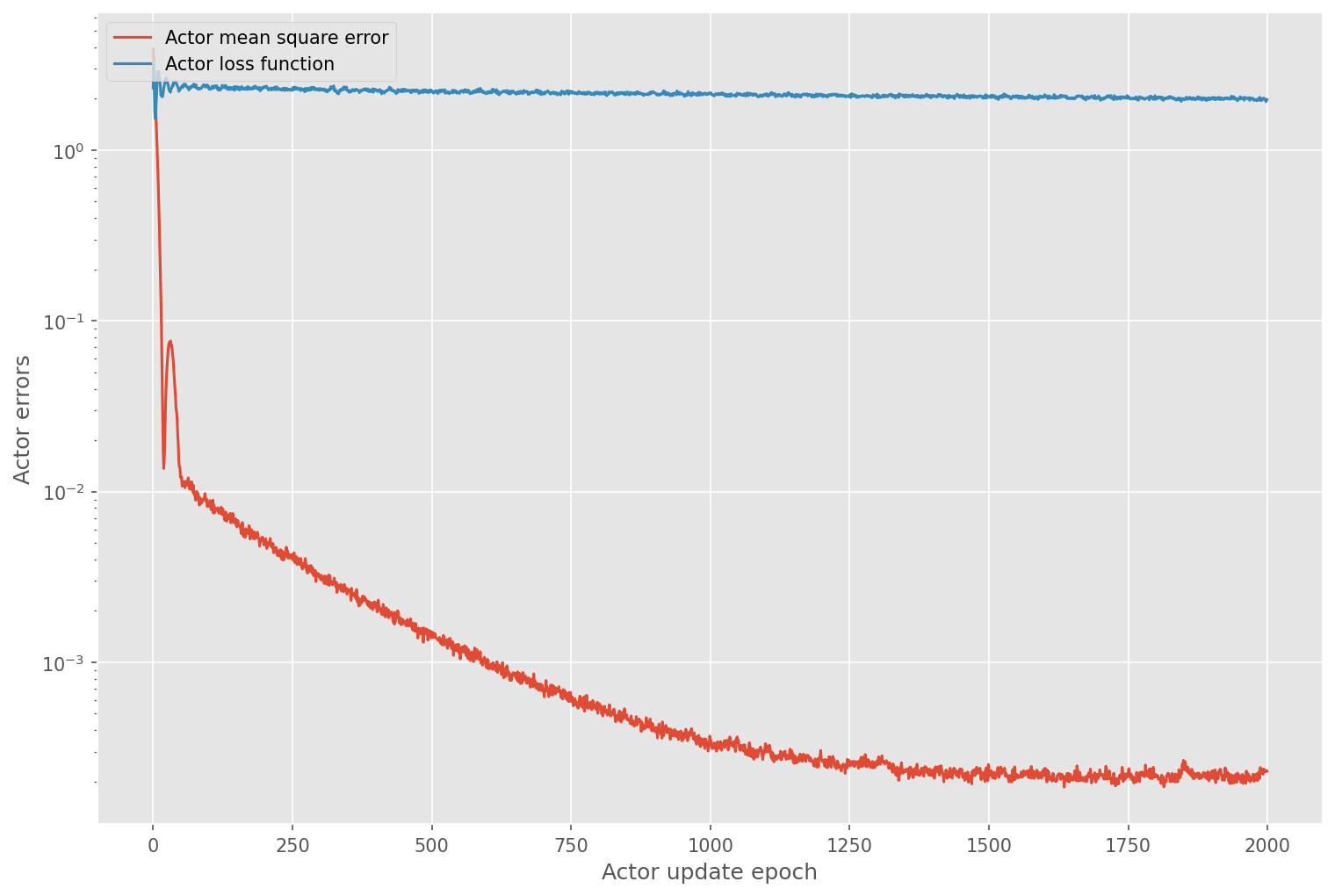}
   \caption{Actor mean square error}
\end{subfigure}
\begin{subfigure}{0.49\textwidth}
   \includegraphics[width=1\linewidth]{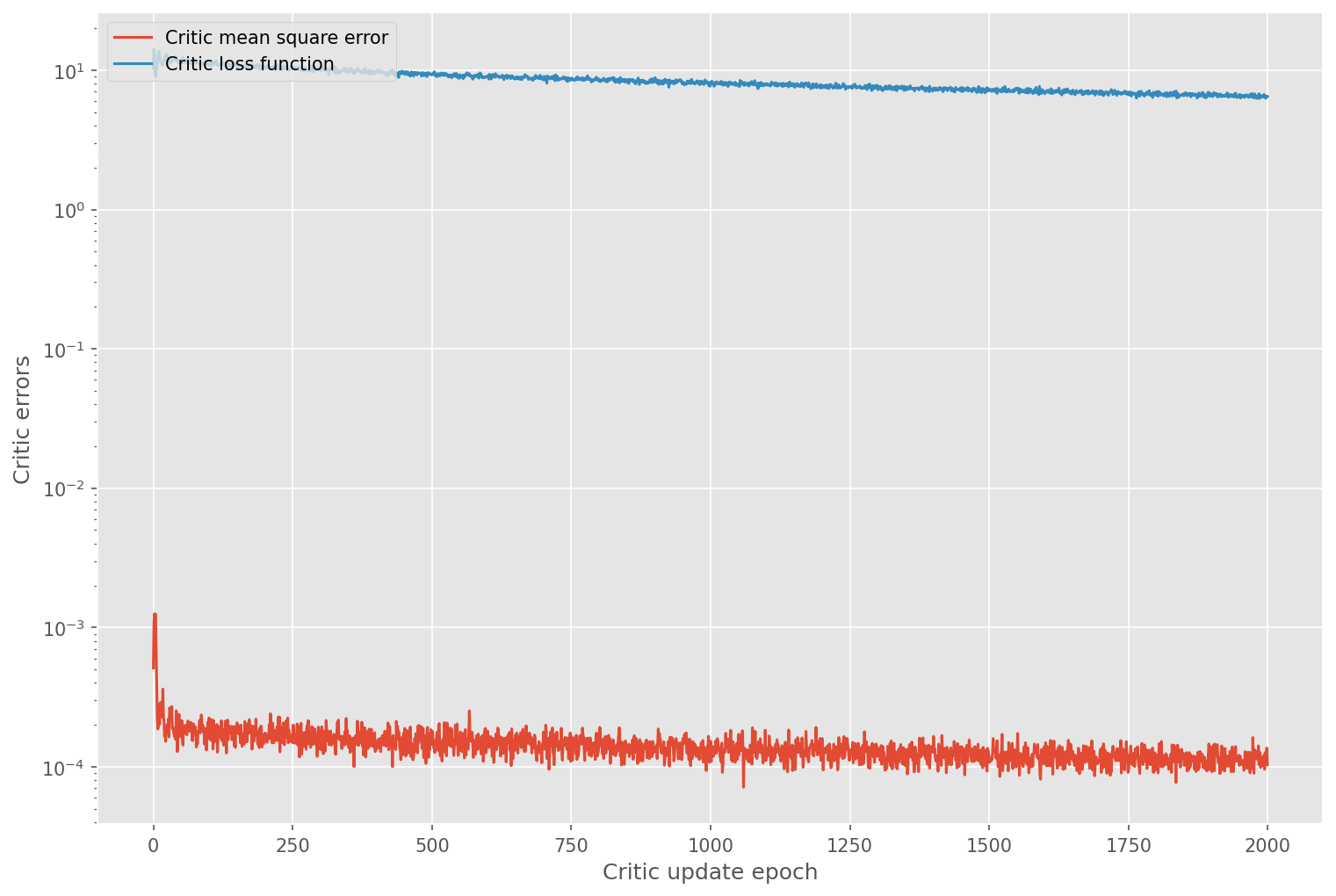}
   \caption{Critic mean square error}
\end{subfigure}
\caption{Convergence is monotonic for Problem 2B}
\label{problem_2B}
\end{figure}

\subsubsection{Problem 3 (high sensitivity of the drift coefficient to the control)}
We consider the setup where the domain is $\Omega = B(0,1)\subset \R^{10}$, the action space is $A = [-1000,1000]^3$, the dimension of the Brownian motion is $d' = 10$, the discounting rate is $\gamma = 1$, and
\begin{align*}
    V(x) &= 2 \nor{x}^4 - \nor{x}^2,\\
    u^*(x) &= (\tanh(x_1), \sinh(x_2), \cosh(x_3)), \\
    b_i(x,a) &= x_i(e^{a_1}+e^{a_2}+e^{a_3}) + e^{-x_i}, \\
    \Phi(x,a) &= \mathrm{Id}_{10\times10}, \\
    \zeta(x,a) &= \nor{a - u^*(x)}_{L^1}.
\end{align*}
Heuristically, we set $A = [-1000,1000]^3$ rather than $A = \R^3$ to ensure that the hypotheses of the verification theorem are satisfied while changing the problem as little as possible. To implement this, we apply a simple clamp function element-wise to the actor neural network output. We expect \textit{a priori} some difficulty due to the extreme sensitivity of the drift coefficient $b_i(x,a)$ to the action $a$.

Indeed, training behavior is highly inconsistent. With our standard initialization, the critic network will usually converge to a function which gives a mean square error on the order of $10^{-2}$. However, while the actor network will sometimes converge to a function which gives a mean square error on the order of $10^{-1}$, most of the time its outputs blow up to the boundary of $A = [-1000,1000]^3$ -- presumably, they would go to infinity if we were to set $A = \R^3$. This blowup corresponds with the estimated integral-Hamiltonian $L_{\textrm{actor}}(\theta; \phi)$ going to (seemingly) negative infinity and the critic loss function $L_{\mathrm{critic}}(\phi;\theta)$ being dragged to (seemingly) positive infinity. If we reduce the learning rates, this fate is often avoided, but the learned solutions are nevertheless not impressive. Representative error curves of the actor-critic blowup under standard learning rates and convergence under reduced learning rates are provided in the first two regimes of Figure \ref{ReLU_P3}.

A theoretical explanation for the exploding actor values is readily available. Recall that the actor loss term is the integral of the estimated Hamiltonian, i.e.
\begin{align*}
    L_\mathrm{actor}(\theta; \phi) &= \int_\Omega H(U_\theta^N, Q_\phi^{N^*})(x)d\mu(x) \\
    &= \int_\Omega \paren{\sum_{i}b_i(x,U_\theta^N(x))\partial_i Q_\phi^{N^*}(x)+\frac{1}{2}\sum_{i,j}{\Phi\Phi^\intercal}_{ij}(x,U_\theta^N(x))\partial^2_{ij} Q_\phi^{N^*} + c(x,U_\theta^N(x))} d\mu(x).
\end{align*}
The motivation for minimizing this term is that the true value function $V$ and optimal control $u^*$ satisfy
\begin{align*}
    H(u^*, V)(x) = \inf_{u \in \mathcal{U}} H(u,V)(x) = \gamma V(x)
\end{align*}
for all $x \in \Omega$, and so it is reasonable to expect lower Hamiltonian values to correspond with better controls. However, it is possible that $Q_\phi^{N^*}$ and $V$ differ in such a way that the above Hamiltonian has a minimum less than $\gamma V(x)$, or even no finite minimum, for many or all $x \in \Omega$. For example, in the above problem, the Hamiltonian reduces to
\begin{align*}
    H(U_\theta^N, Q_\phi^{N^*})(x) = \inprod{b(x,U_\theta^{N}(x))}{\nabla Q_\phi^{N^*}(x) - \nabla V(x)} + \frac{1}{2}\mathrm{Tr}(\Phi\Phi^\intercal\mathrm{Hess}(Q_\phi^{N^*}(x)-V(x)))+ \zeta(x, U_\theta^N(x)).
\end{align*}
In the specific case of Problem 3, one can show that the actor optimizer's task is equivalent to minimizing
$$ \int_{\Omega} \paren{e^{(U_\theta^N)_1(x)}+e^{(U_\theta^N)_2(x)}+e^{(U_\theta^N)_3(x)}}\sum_{i=1}^{10} x_i \partial_i (Q_\phi^{N^*}-V)(x) + \zeta(x,U_\theta^N(x)) d\mu(x).$$
Naturally, then, the actor's optimizer might tend towards increasing one or more of the coordinates of $U_\theta^N(x)$ at any $x \in \Omega$ such that $$\sum_{i=1}^{10}x_i \partial_i(Q_{\phi}^{N^*}-V)(x) < 0.$$ The only factor encouraging $U_{\theta}^N(x)$ to stay close to $u^*(x)$ is the $\zeta(x, U_{\theta}^N(x))$ penalty term. But this penalty grows only linearly in each coordinate of the actor, whereas $e^{(U_\theta^N)_1(x)}+e^{(U_\theta^N)_2(x)}+e^{(U_\theta^N)_3(x)}$ is exponential. It is thus possible for the actor to be encouraged to wander off to infinity (or the boundary of the action space) for at least some $x \in \Omega$.

One solution to this numerical issue is to redefine the actor loss function as
\begin{align}\label{modified_loss}
    L_{\mathrm{actor}}(\theta; \phi) = \int_{\Omega} \max\paren{H(U_\theta^N,Q_\phi^{N^*})(x) - \gamma Q_\phi^{N^*}(x), \delta} d\mu(x)
\end{align}
for some $\delta \leq 0$. In essence, this trick makes the actor agnostic to minimizing the Hamiltonian too far beyond what we should expect it to be if the critic were roughly accurate. In the above problem, if we choose $\delta = -10$, then Algorithm \ref{algo_ac} trains well, obtaining a critic mean square error of around $4.91 \times 10^{-3}$ and an actor mean square error of around $2.63 \times 10^{-2}.$

For a comparison of the typical results under the standard hyperparameters, the reduced learning rate hyperparameters, and the modified Hamiltonian, see Table \ref{P3_table} and Figure \ref{ReLU_P3}.

\begin{table}[]
\centering
\begin{tabular}{@{}llll@{}}
\toprule
Regime                & Standard & Reduced learning rates & Modified actor loss \\ \midrule
Critic mean square error & $1.43 \times 10^{-2}$  & $0.247$  & $4.91 \times 10^{-3}$ \\
Critic relative error  & $3.30 \times 10^{-2}$  & $0.566$  & $1.12\times 10^{-2}$  \\
Actor mean square error  & $5.73 \times 10^3$ & $1.23 $& $2.63 \times 10^{-2}$ \\
Actor relative error  & $4.57 \times 10^3$  & $0.983$  & $2.10 \times 10^{-2}$
\\ \bottomrule
\end{tabular}
\caption{Critic and actor mean square errors for Problem 3 in three regimes}
\label{P3_table}
\end{table}

\begin{figure}
\centering
\begin{subfigure}{0.49\textwidth}
   \includegraphics[width=1\linewidth]{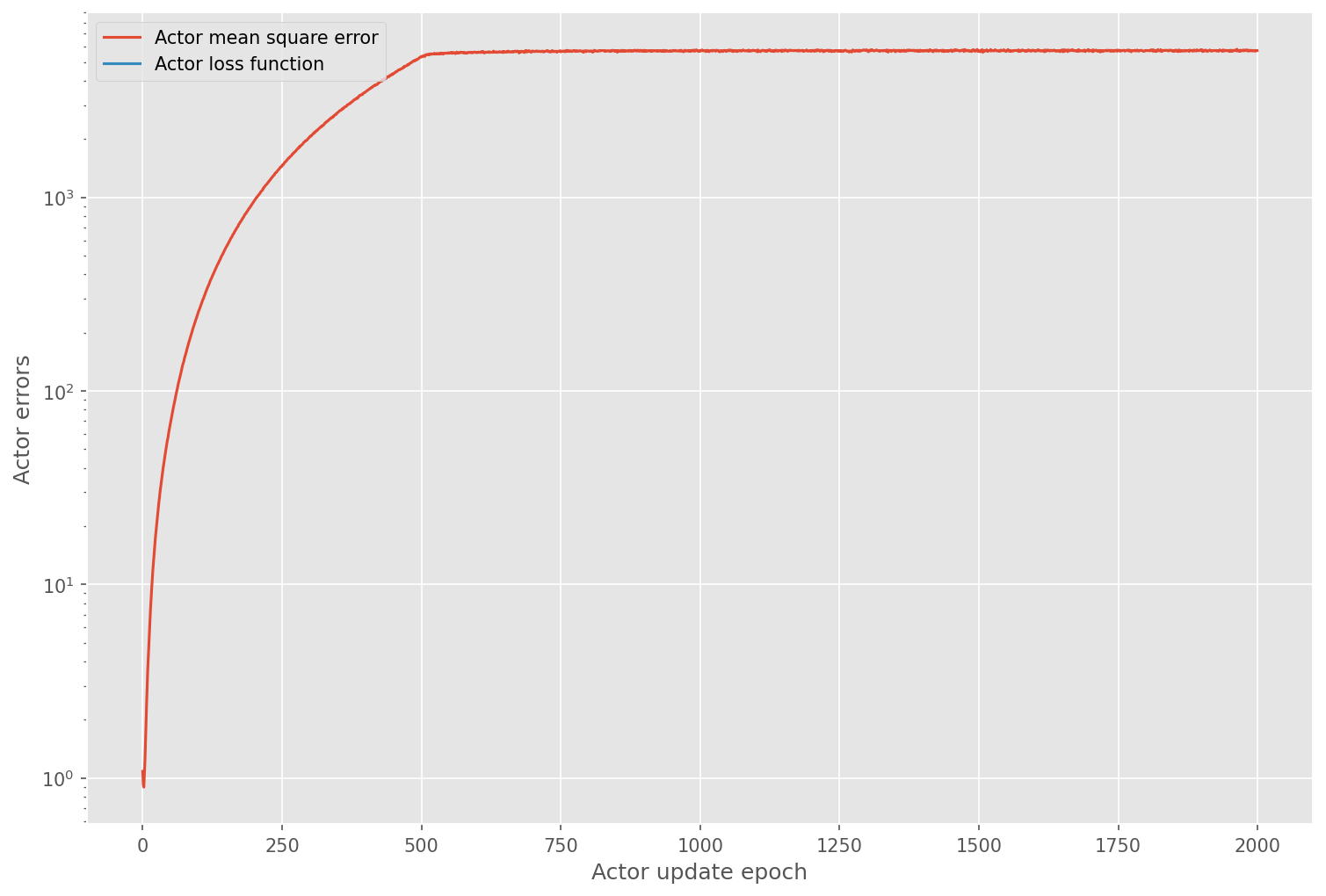}
   \caption{Actor mean square error in the standard regime}
\end{subfigure}
\begin{subfigure}{0.49\textwidth}
   \includegraphics[width=1\linewidth]{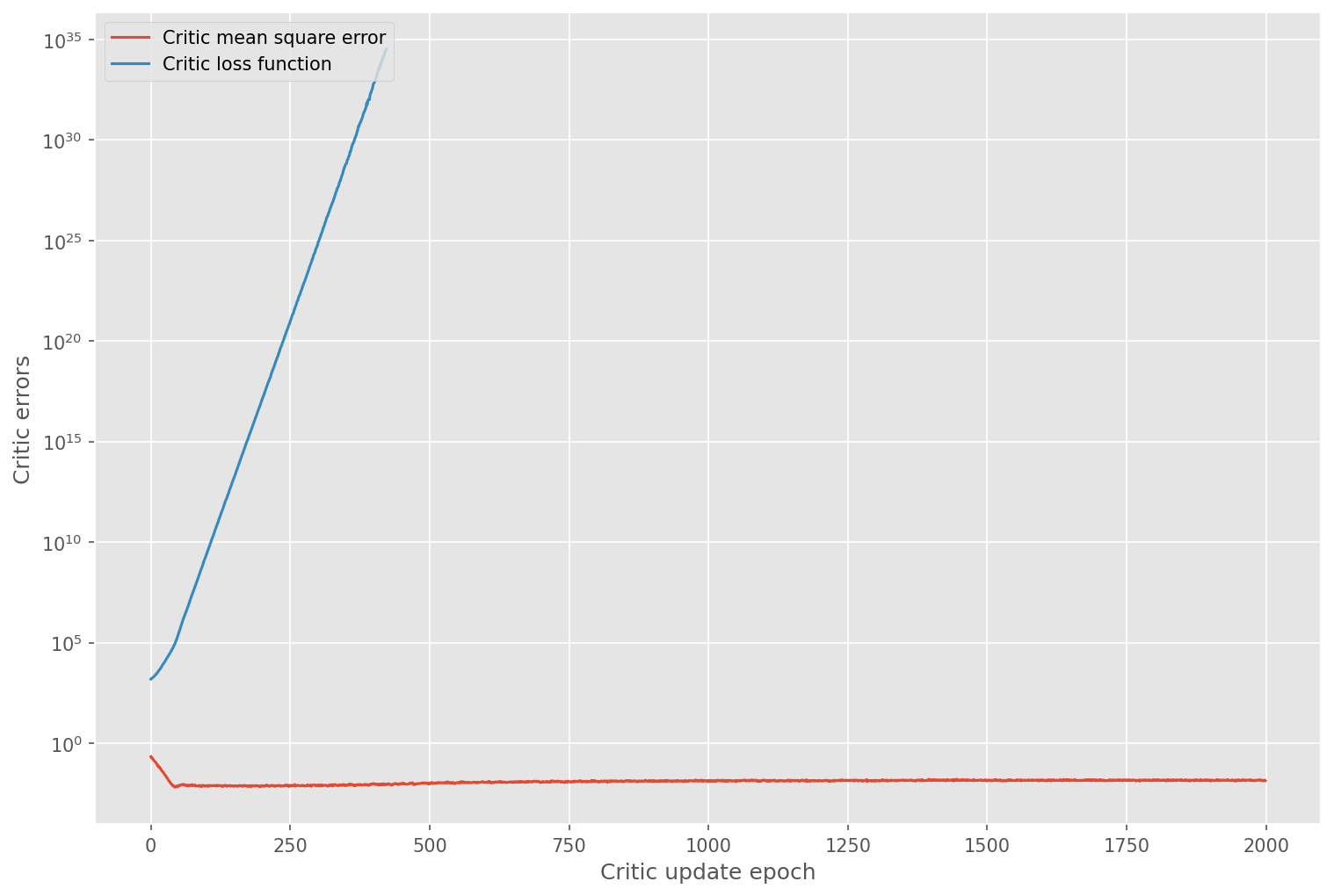}
   \caption{Critic mean square error in the standard regime}
\end{subfigure}

\begin{subfigure}{0.49\textwidth}
   \includegraphics[width=1\linewidth]{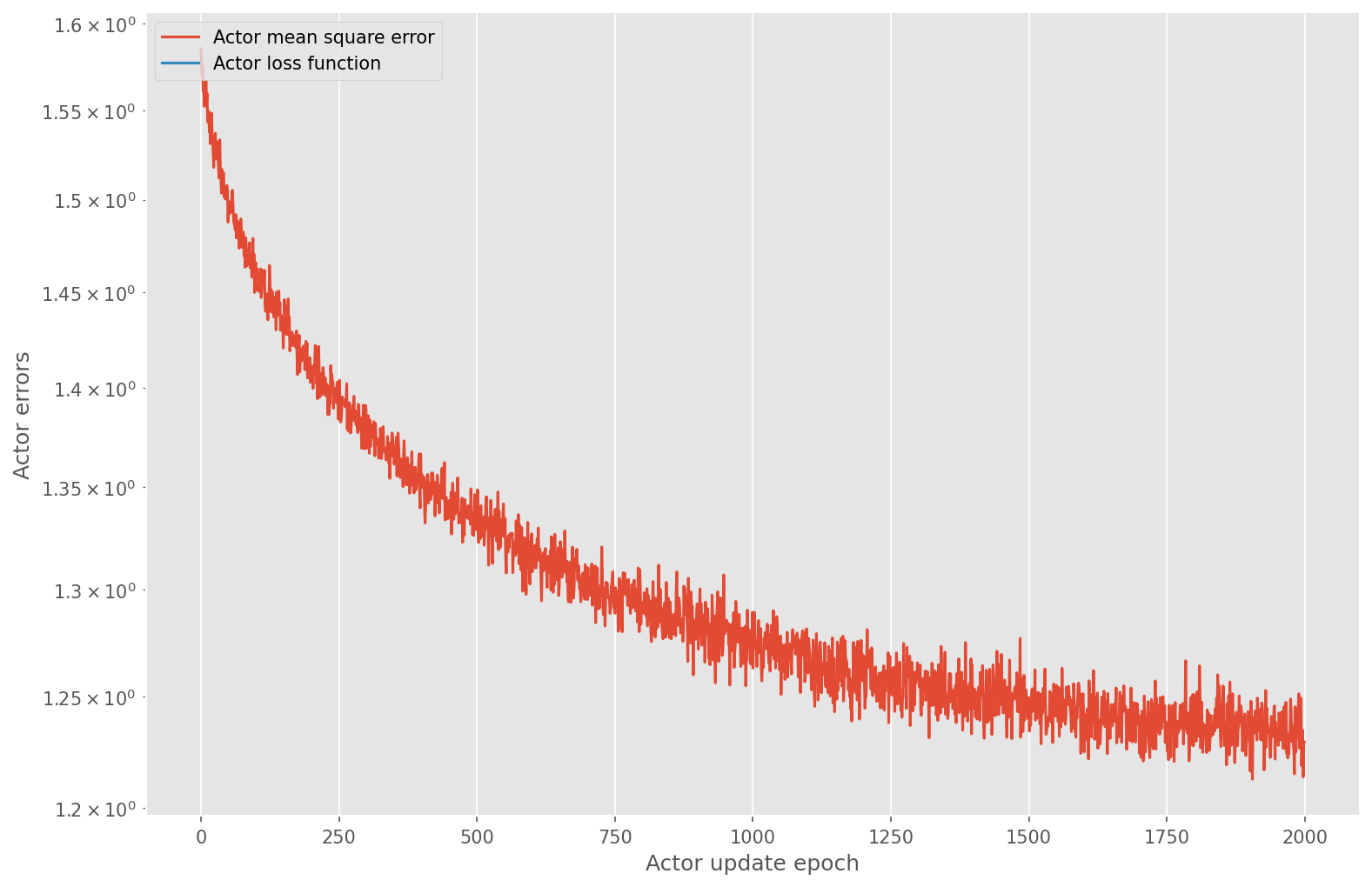}
   \caption{Actor mean square error with reduced learning rates}
\end{subfigure}
\begin{subfigure}{0.49\textwidth}
   \includegraphics[width=1\linewidth]{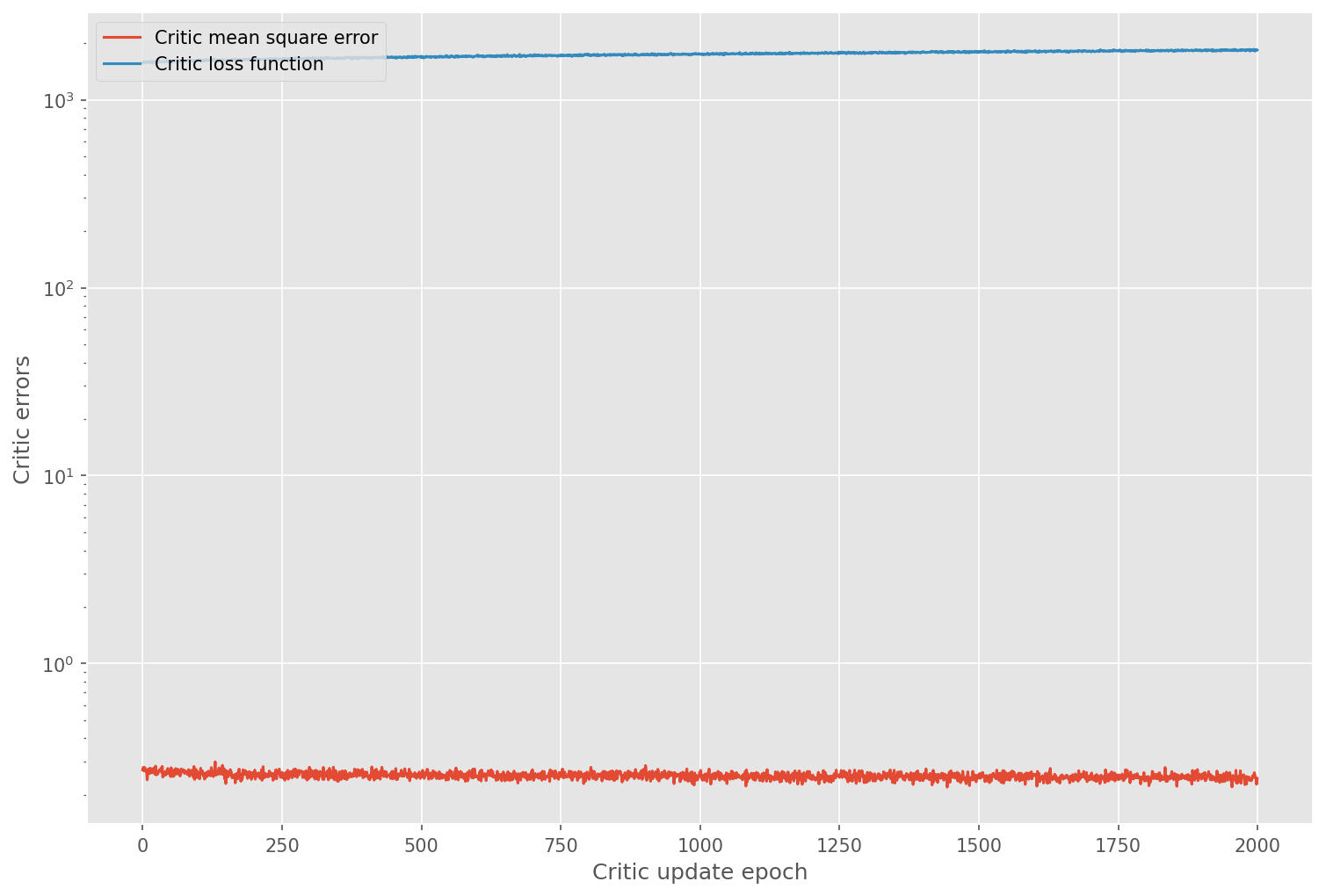}
   \caption{Critic mean square error with reduced learning rates}
\end{subfigure}

\begin{subfigure}{0.49\textwidth}
   \includegraphics[width=1\linewidth]{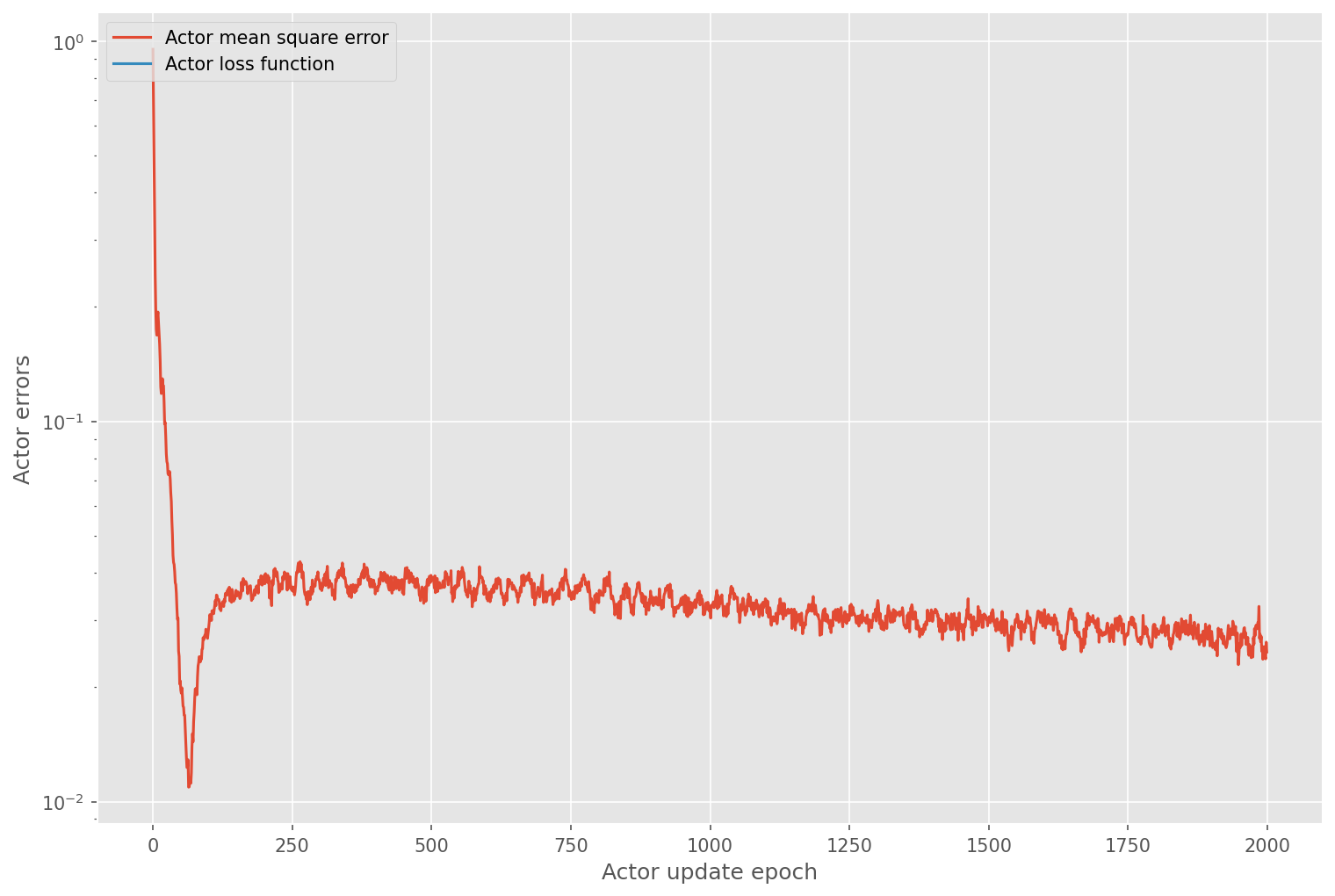}
   \caption{Actor mean square error with $L_{\mathrm{actor}}(\theta;\phi)$ defined as in \eqref{modified_loss}}
\end{subfigure}
\begin{subfigure}{0.49\textwidth}
   \includegraphics[width=1\linewidth]{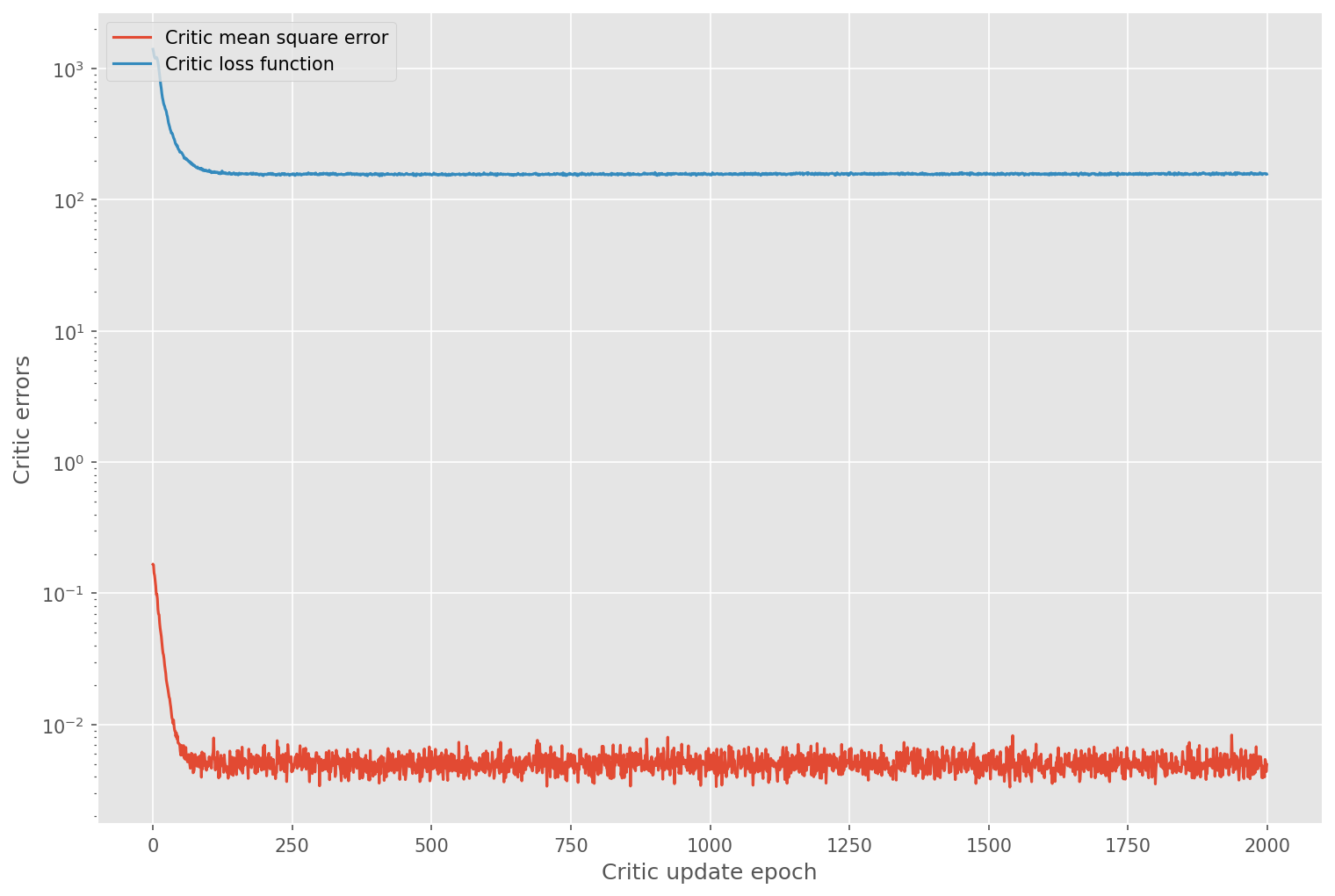}
   \caption{Critic mean square error with $L_{\mathrm{actor}}(\theta;\phi)$ defined as in \eqref{modified_loss}}
\end{subfigure}

\caption{Performance on Problem 3 in three distinct regimes}
\label{ReLU_P3}
\end{figure}

\subsubsection{Problem 4}
We consider the setup where the domain is $\Omega = [-1,1]^{10}$, the action space is $A = \R^{10}$, the dimension of the Brownian motion is $d' = 10$, the discounting rate is $\gamma = 1$, and
\begin{align*}
    V(x) &= 1 + \prod_{i=1}^{10}\paren{1-\sin\paren{\frac{\pi x_i^2}{2}}},\\
    u^*_i(x) &= x_i \paren{1 + \prod_{j=1}^{10}x_j}, \\
    b_i(x,a) &= a_i x_i + \nor{x}^2, \\
    \Phi(x,a) & = \mathrm{Id}_{10\times 10} \paren{1 + \frac{\nor{a}^2}{10}},\\
    \zeta(x,a) &= \nor{a - u^*(x)}^2.
\end{align*}

Algorithm \ref{algo_ac} learns well in this case. The actor quickly converges, giving a mean square error of about $2.11 \times 10^{-2}$ and relative error of about $6.35 \times 10^{-3}$. The critic convergence, while noisy, is even better -- a mean square error of about $4.60 \times 10^{-5}$ and relative error of about $4.57 \times 10^{-5}.$

This is especially impressive given the complicated forms of the value function $V$ and optimal control $u^*$, the fact that the diffusion coefficient $\Phi$ is controlled, and the modest choice of $\zeta$, which is the term responsible for making the optimal control optimal.

\begin{figure}
\centering
\begin{subfigure}{0.49\textwidth}
   \includegraphics[width=1\linewidth]{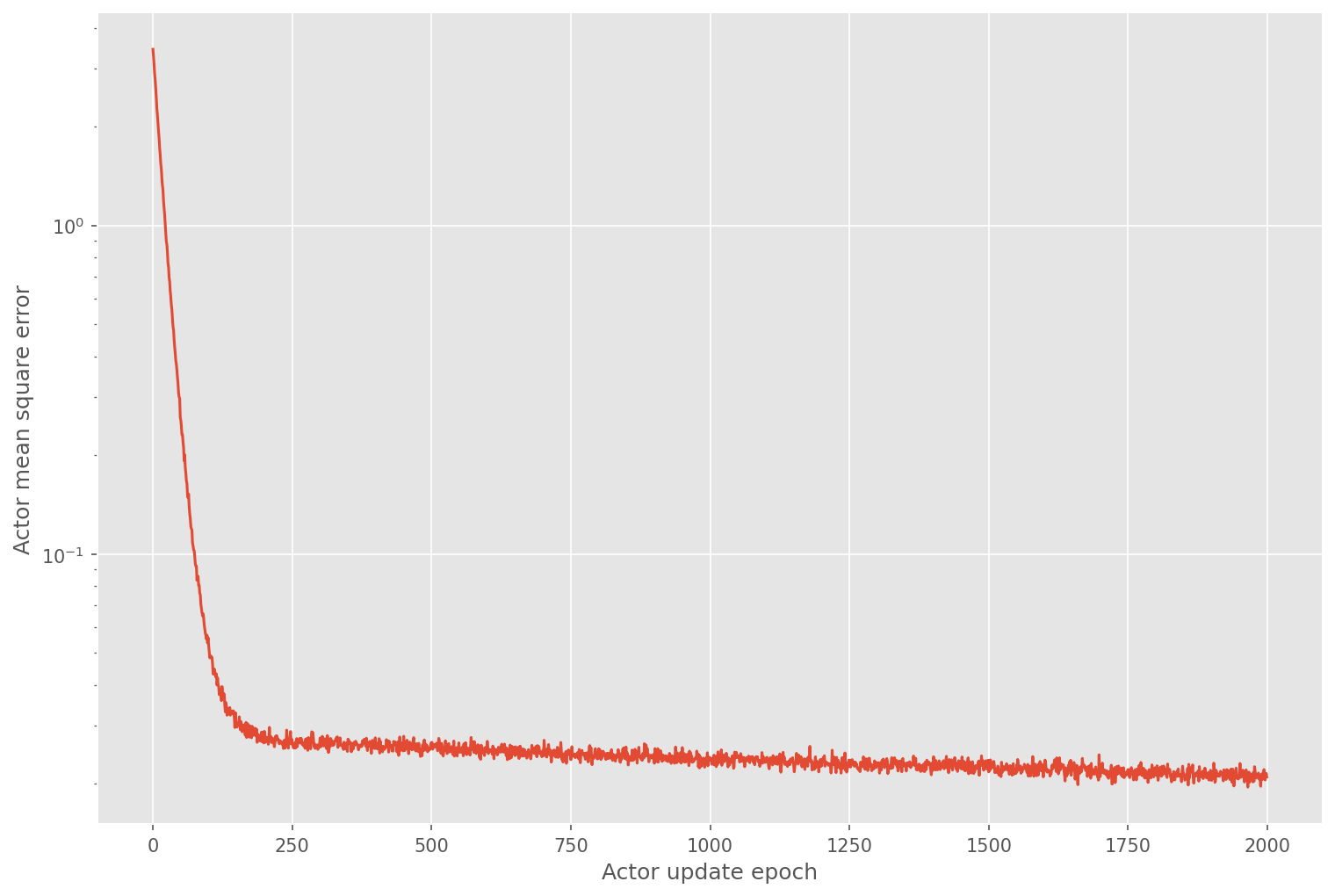}
   \caption{Actor mean square error}
\end{subfigure}
\begin{subfigure}{0.49\textwidth}
   \includegraphics[width=1\linewidth]{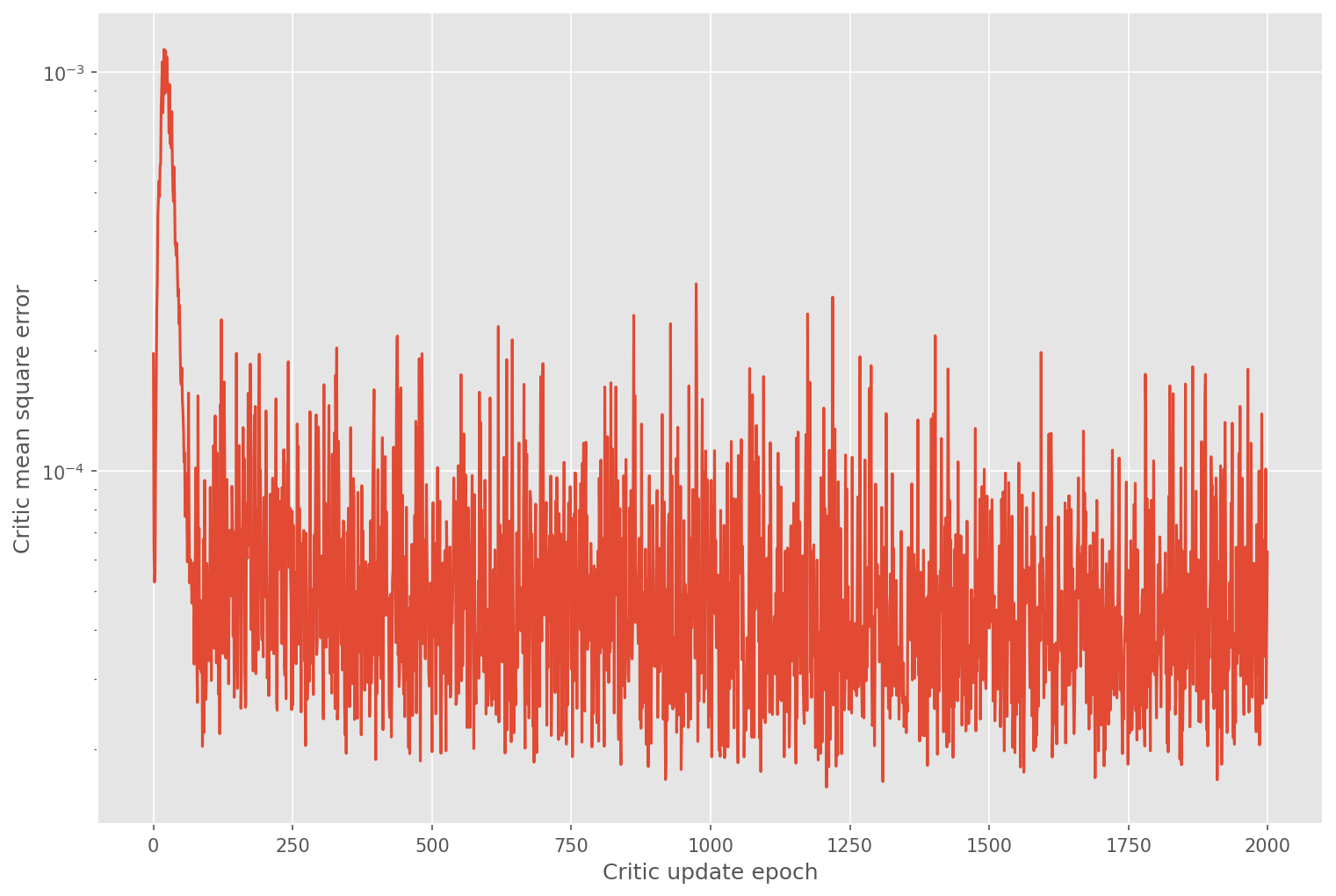}
   \caption{Critic mean square error}
\end{subfigure}
\caption{Convergence is rapid for Problem 4}
\label{problem_4}
\end{figure}

\subsubsection{Problem 5}
We consider the setup where the domain is $\Omega = [-1,1]^{10}$, the action space is $A = \R$, the dimension of the Brownian motion is $d' = 10$, the discounting rate is $\gamma = 1$, and
\begin{align*}
    V(x) &= 1 + \nor{x}^2\prod_{i=1}^{10} \sin(\pi x_i), \\
    u^*(x) &= x_1\sin(\pi x_3) + x_2, \\
    b_i(x,a) &= x_i a, \\
    \Phi(x,a) &= \textrm{Id}_{10\times10}, \\
    \zeta(x,a) &= (a-u^*(x))^2.
\end{align*}
Convergence is extremely noisy, as can be seen in Figure \ref{problem_5}. Nonetheless, Algorithm \ref{algo_ac} learns decently in this scenario. The actor mean square error converges to about $1.1 \times 10^{-1}$ and relative error to about $0.23$. The critic mean square error and relative error both converge to around $8.0 \times 10^{-3}$.

\begin{figure}
\centering
\begin{subfigure}{0.49\textwidth}
   \includegraphics[width=1\linewidth]{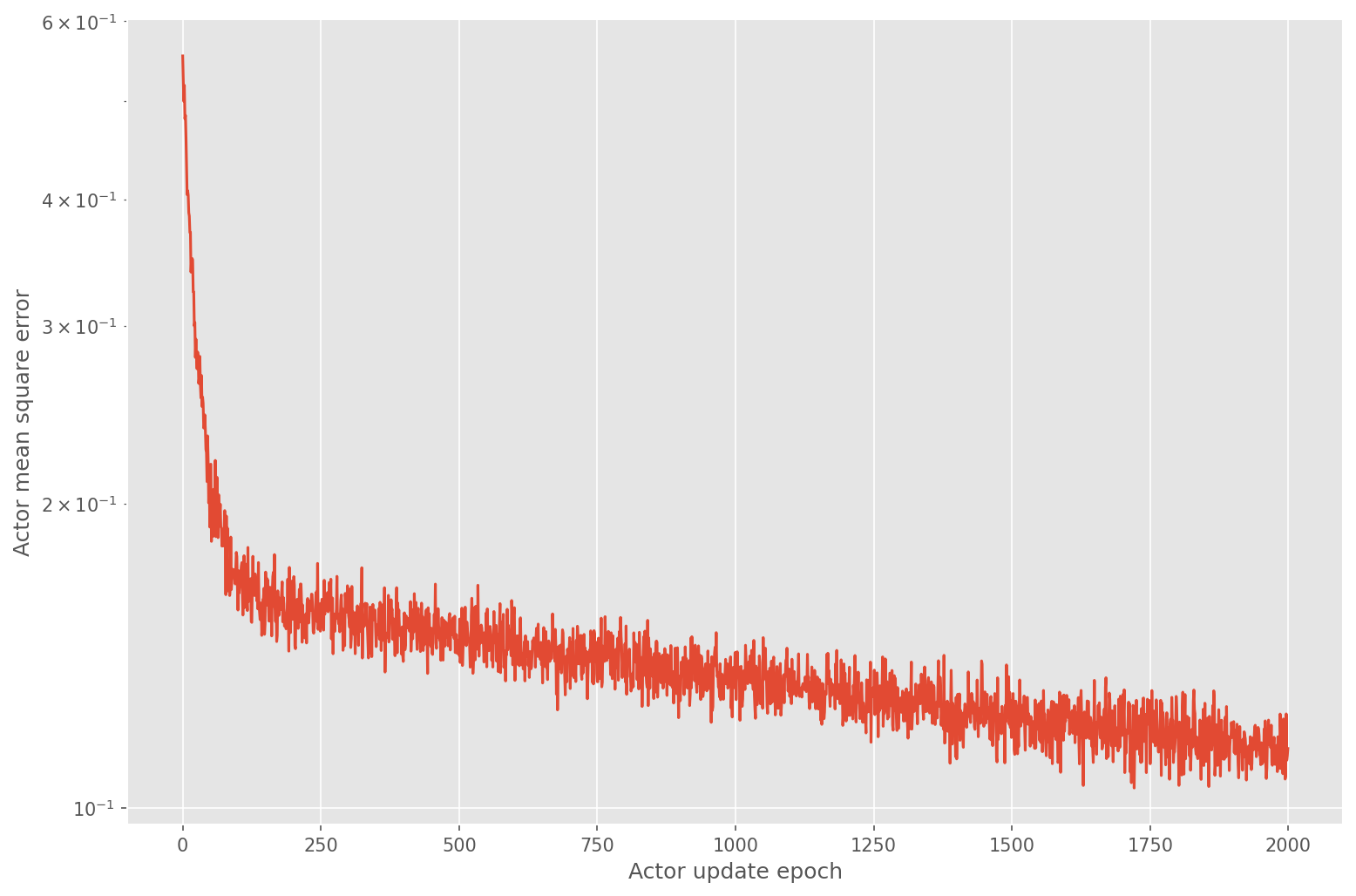}
   \caption{Actor mean square error}
\end{subfigure}
\begin{subfigure}{0.49\textwidth}
   \includegraphics[width=1\linewidth]{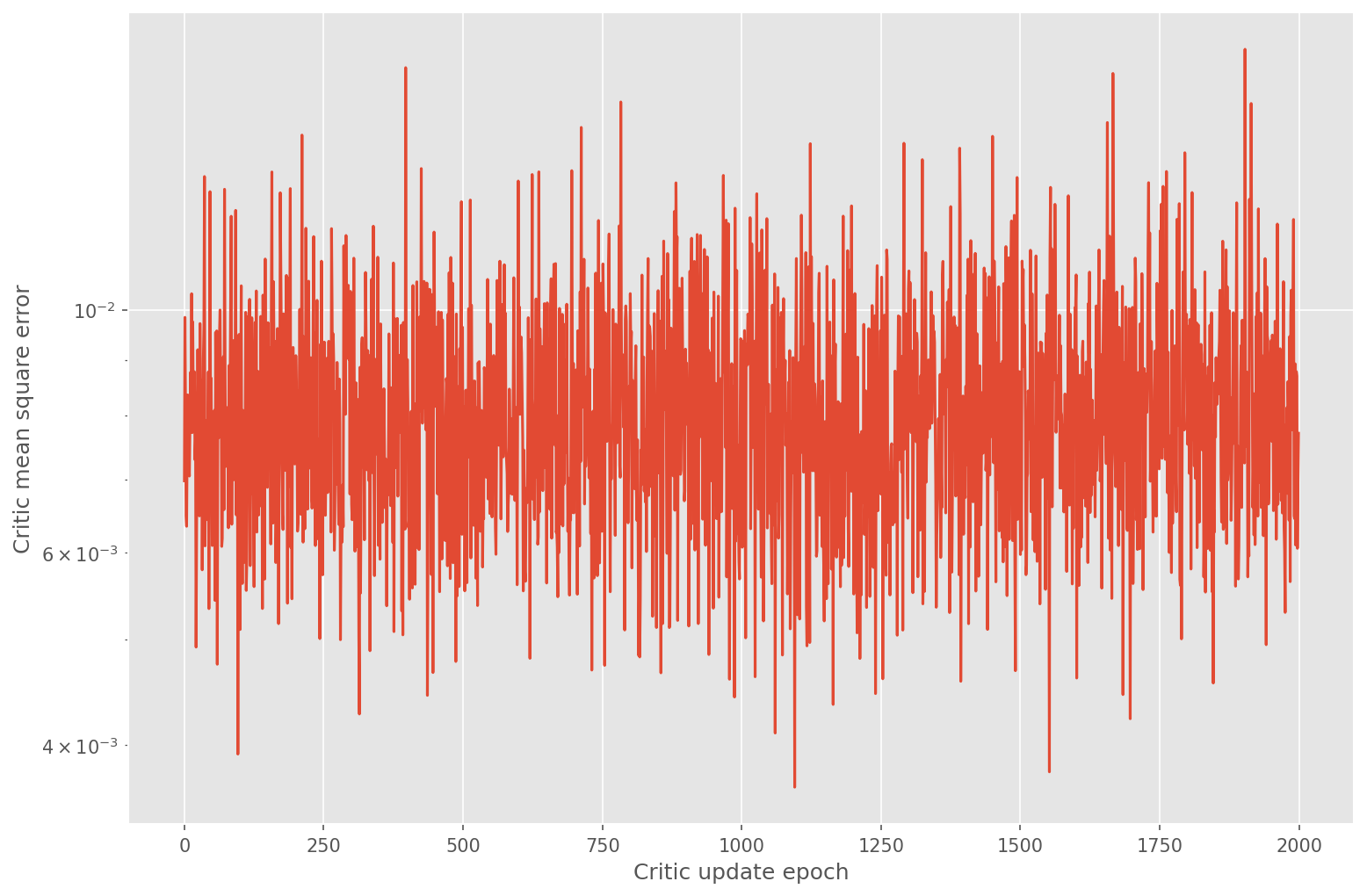}
   \caption{Critic mean square error}
\end{subfigure}
\caption{Convergence is noisy for Problem 5}
\label{problem_5}
\end{figure}

\subsection{Merton problems}
We now demonstrate using Algorithm \ref{algo_ac} to solve $n$-asset Merton problems. Here, the market contains $n$ risky assets $S_t = (S_{t,i})_{i=1}^{n}$ that evolve according to a (correlated) geometric Brownian motion and a deterministic asset $B_t$ representing risk-free interest:
\begin{align}
    dS_t/S_t &= \mu dt + \sigma dW_t, & S_0 &= (S_{0,i})_{i=1}^n \\
    dB_t &= rB_tdt,  & B_0 &= b_0.
\end{align}
Here, $\sigma \in M^{n\times n}$ is a constant matrix such that the covariance matrix $\sigma\sigma^\intercal \in S_{>0}^{n\times n}$ is positive definite and $\mu \in \R^n$ is a constant vector. We will use the constant absolute risk aversion (CARA) utility function
\begin{equation*}
    U_p(x) = -e^{-px}
\end{equation*}
with risk aversion parameter $p > 0$.

\subsubsection{Finite-horizon terminal wealth optimization}\label{subsec:merton_simple}
Let the portfolio $\pi = (\pi_{t,i})_{i=1}^n \in \R^n$ represent dollar investments in the $n$ risky assets. The dynamics of a self-financing portfolio $X_t$ then follow
\begin{equation}
    dX_t= (rX_t + \inprod{\pi_t}{\mu - r\mathbf{1}_n})dt + \inprod{\pi_t}{\sigma dW_t}, \quad X_0 = x,
\end{equation}
where $\pi \in \mathcal{U}$ serves as our feedback control to maximize the value functional
\begin{align*}
    V(t,x) &= \sup_{\pi \in \mathcal{U}} V^\pi(t,x) \\
    &= \sup_{\pi \in \mathcal{U}} \E\brac{U_p(X_T^\pi) \mid X_t^\pi = x}
\end{align*}
This gives rise to the HJB equation
\begin{align}\label{eq_merton_hjb}\begin{split}
    0 &= V_t(t,x) + \sup_{\pi \in \R^n} \curlbrac{(rx + \inprod{\pi}{\mu - r\mathbf{1}_n})V_x(t,x) + \frac{1}{2}\pi^\intercal\sigma\sigma^\intercal\pi V_{xx}(t,x)} \\
    V(T,x) &= U_p(x). \end{split}
\end{align}
The analytic solution is\footnote{The solution given here can be verified by direct substitution. For more general cases, see e.g.~\cite[Ch.~3, Thm.~6.3]{karatzas1998}.}
\begin{align}\label{merton_solution}\begin{split}
    V(t,x) &= -\exp\paren{-pe^{r(T-t)}x- \frac{T-t}{2} (\mu - r\mathbf{1}_n)^\intercal(\sigma\sigma^\intercal)^{-1} (\mu - r\mathbf{1}_n)} \\
    u^*(t,x) = \pi^*(t,x) &= \frac{1}{p}e^{-r(T-t)}(\sigma\sigma^\intercal)^{-1}(\mu - r\mathbf{1}_n).\end{split}
\end{align}
While the PDE domain closure $\overline\Omega = [0,T] \times \R$ is not compact (and we do not have \textit{a priori} the value $V$ at the boundary $t = 0$), we can still construct a critic neural network in the spirit of Assumptions \ref{assumption_eta}, \ref{assumption_g} and \eqref{q_def} with
\begin{align*}
    \overline{g}(t,x) &= U_p(x) \\
    \eta(t,x) &= T-t. 
\end{align*}
Notice that using the above functions in \eqref{q_def} enforces the boundary condition in \eqref{eq_merton_hjb} while still allowing the critic flexibility to learn any $C^{1,2}$ function on the interior. To train the actor and critic networks and evaluate their accuracy, we sample points from the reference probability measure
\begin{align*}
    \mu &= \mathrm{Unif}([0,T]) \times \rho
\end{align*}
where we select $\rho = \mathcal{N}(1, 0.5^2)$, the normal distribution with mean $1$ and standard deviation $0.5$, so that $\mu$ is absolutely continuous with respect to Lebesgue measure and has full support over $\Omega$. The actor and critic networks each have 512 hidden neurons, like with the previous examples. We use an Adam optimizer for 10,000 steps with the learning rates
\begin{equation*}
    \alpha_t = \frac{0.01}{(10+t)^{0.8}} \quad \text{and} \quad \omega_t =  \frac{0.01}{(10+t)^{0.6}},
\end{equation*}
alternating between actor and critic training after each training step.

\paragraph{Synthetic parameters} We begin with results for an easy setup with synthetic parameters. Table \ref{merton_simple_data} contains the results for $n = 10, 20, 50, 100, 200$ when
\begin{equation}\label{merton_simple}
    \mu = 0.1 \times \mathbf{1}_n, \quad \sigma = 0.2 \times \mathrm{Id}_{n\times n}, \quad r = 0.05, \quad T = \frac{1}{12}, \quad p = 1.
\end{equation}
That is to say, the market consists of $n$ independent stocks and the agent is investing to maximize CARA utility in one month's time.

Algorithm \ref{algo_ac} learns the optimal control and PDE solution well, as shown in Table \ref{merton_simple_data}.

\begin{table}[]
\centering
\begin{tabular}{@{}llllll@{}}
\toprule
Dimension $(n)$                & 10 & 20 & 50 & 100 & 200 \\ \midrule
Critic mean square error & $1.07 \times 10^{-8}$   & $2.29\times 10^{-8}$   & $5.11 \times 10^{-7}$ & $6.70 \times 10^{-6}$ & $9.79 \times 10^{-5}$ \\
Critic relative error    & $5.25 \times 10^{-8}$ & $1.16\times 10^{-7}$   & $2.79 \times 10^{-6}$ & $4.12 \times 10^{-5}$ & $7.55 \times 10^{-4}$  \\
Actor mean square error  & $6.73 \times 10^{-5}$ & $1.52 \times 10^{-4}$   &  $3.99 \times 10^{-3}$ & $2.73 \times 10^{-2}$ & $0.342$ \\
Actor relative error     & $4.32 \times 10^{-6}$   &  $4.89\times 10^{-6}$  &  $5.13 \times 10^{-5}$ & $1.75 \times 10^{-4}$ & $1.10 \times 10^{-3}$ \\ \bottomrule
\end{tabular}
\caption{Algorithm \ref{algo_ac} on the classical Merton problem with synthetic parameters \eqref{merton_simple}} \label{merton_simple_data}
\end{table}

\paragraph{Calibrated parameters} We now fix $n = 100$ and use 5 years of real data to estimate the mean $\mu$ and covariance $\sigma\sigma^\intercal$ of the 100 stocks in the S\&P 100 as of February 9, 2026\footnote{More precisely, we estimate the mean and covariance of the daily log-returns. The parameters are then annualized and $\mu$ is calculated from the log-return mean and covariance.}. This is a more difficult problem than before due to the poor conditioning of the matrix $\sigma\sigma^\intercal$ compared to the synthetic parameters \eqref{merton_simple}. Of course, this data is affected by selection and survivorship bias as well as other complicating factors present in the model, which limit how well we can expect these values to generalize to future data. We introduce this problem more as an example of the algorithm's performance with ill-conditioned covariance matrices than as a \textit{per se} robust financial model.

Algorithm \ref{algo_ac} performs well even in this context, with Table \ref{merton_real_data} showing low relative error for both the critic and actor. 

\begin{table}[]
\centering
\begin{tabular}{@{}lll@{}}
\toprule
Parameters                & Synthetic ($n = 100$) & Calibrated ($n = 100$) \\ \midrule
Critic mean square error & $6.70 \times 10^{-6}$   & $2.10\times 10^{-4}$ \\
Critic relative error    & $4.12 \times 10^{-5}$ & $1.73\times 10^{-3}$ \\
Actor mean square error  & $2.73 \times 10^{-2}$ & $0.470$ \\
Actor relative error     & $1.75 \times 10^{-4}$   &  $1.23\times 10^{-3}$ \\ \bottomrule
\end{tabular}
\caption{Algorithm \ref{algo_ac} on the classical Merton problem with synthetic parameters \eqref{merton_simple} and parameters calibrated to 5 years of data for S\&P 100 stocks} \label{merton_real_data}
\end{table}

\paragraph{Logarithmic utility}
We remark that if we instead employ the logarithmic utility $U_{\log} (x) = \log x$, then the PDE solution and optimal control are
\begin{align*}
    V(t,x) &= \log x + \paren{r + \frac{1}{2}(\mu - r\mathbf{1}_n)^\intercal(\sigma\sigma^\intercal)^{-1} (\mu - r\mathbf{1}_n)}(T-t) \\
    u^*(t,x) = \pi^*(t,x) &= (\sigma\sigma^\intercal)^{-1}(\mu - r\mathbf{1}_n)x.
\end{align*}
In this case, one must be careful when choosing the training measure $\mu$ to use in Algorithm \ref{algo_ac}. In particular, the asymptote at $x = 0$ means that we must choose a measure with low density in that region. For example, with $n = 50$ and parameters as in \eqref{merton_simple}, the training measure
\begin{equation*}
    \mu = \mathrm{Unif}([0,T]) \times \mathrm{Gamma}(\alpha = 2, \lambda =1)
\end{equation*}
gives critic relative error of $3.28\times 10^{-5}$ and actor relative error of $0.123$. However, if we select
\begin{equation*}
    \mu = \mathrm{Unif}([0,T]) \times \mathrm{Gamma}(\alpha = 1, \lambda =1),
\end{equation*}
then the critic relative error is $0.409$ and the actor relative error is $0.984$.

\paragraph{Monte Carlo benchmarking}

We have demonstrated that Algorithm \ref{algo_ac} learns the optimal control accurately in the norm of the action space $A = \R^n$. However, one might suspect that actor error metrics are unreliable because small differences in the control may, in certain circumstances, lead to substantial changes in the value functional. To test this, we run 1,000,000 Monte Carlo simulations of the wealth process of an agent who follows the analytic optimal control $\pi^*(t,x)$ and an agent who follows the learned control in the same market and then evaluate the difference in mean CARA utility at time $T = 1/12$. More precisely, we sample 2,000 random initial wealth values from $\rho = \mathcal{N}(1, 0.5^2)$, simulate 500 market paths for each initial wealth sample, calculate the corresponding final utilities obtained by following each control, and then take averages. That is to say, we report an estimate of
\begin{align*}
    M_{\mathrm{CARA}}(\pi) &= \int_{-\infty}^\infty V^\pi(x)d\rho(x)\\
    &= \int_{-\infty}^\infty \E [U_p(X^{x,\pi}_T) \mid X_0^x = x]d\rho(x).
\end{align*}
As another metric, we report the \textit{mean certainty equivalent}, i.e.
\begin{align}
    M_{\mathrm{CE}}(\pi) &= \int_{-\infty}^\infty \mathrm{CE}^\pi(x)d\rho(x) \\
    &= \int_{-\infty}^\infty U_p^{-1}(\E [U_p(X^{x,\pi}_T) \mid X_0^x = x])d\rho(x)
\end{align}
The integrand $\mathrm{CE}^\pi(x)$ represents the minimum deterministic (certain) amount of wealth an investor should accept in lieu of the random final wealth $X_T^{x,\pi}$. That is to say, receiving wealth $\mathrm{CE}^\pi(x)$ at time $t = T$ is equivalent in mean utility to investing wealth $x$ according to the trading strategy $\pi$ and reaping $X_T^{x,\pi}$.

Table \ref{cara_simulated_paths} shows extremely close alignment in both the mean CARA utility and mean certainty equivalent between the simulated wealth process under the theoretically optimal control and the learned control. In the case of the synthetic parameters, the learned control is so accurate as not differ from the optimal control to within five significant figures of mean certainty equivalent.

\begin{table}[]
\centering
\begin{tabular}{@{}lll@{}}
\toprule
Merton model                & Synthetic ($n = 100$) & Calibrated ($n = 100$) \\ \midrule
$M_{\mathrm{CARA}}$ (learned control) & $-0.3186$ & $-0.2256$ \\
$M_{\mathrm{CE}}$ (learned control) & $1.2666$ & $1.6199$ \\
$M_{\mathrm{CARA}}$ (optimal control) & $-0.3186$ & $-0.2250$ \\
$M_{\mathrm{CE}}$ (optimal control) & $1.2666$ & $1.6219$ \\
\bottomrule
\end{tabular}
\caption{Mean CARA utility and certainty equivalent with learned control and analytic optimal control in the classical Merton problem} \label{cara_simulated_paths}
\end{table}

\subsubsection{Finite-horizon terminal wealth optimization without short selling}
We wish to emphasize the versatility of Algorithm \ref{algo_ac} by making modifications to the Merton problem. For example, the classical Merton problem relies on the control taking values in $A = \R^n$, i.e.~the agent having access to infinite margin for short selling. This assumption is unrealistic, especially for smaller, non-institutional investors. One might therefore be interested in solving the problem with controls constrained to $A = [0,+\infty)^n$. In such a case, the agent can still borrow infinite capital at an interest rate $r$, but is unable to short sell stocks. At the level of code, this change can be easily implemented by parameterizing the portfolio as the element-wise exponential of the actor, i.e.~$\pi(t,x) = \exp(U^N_{\theta}(t,x))$, or as the clipped actor, i.e.~$\pi(t,x) = \mathrm{ReLU}(U^N_{\theta}(t,x))$. The latter of the two is the approach we take.

For our synthetic parameters \eqref{merton_simple}, forbidding short selling changes little because the optimal control is non-negative in each stock anyway. However, this is not the case for the parameters calibrated to real data, and so we use those for testing.

There is no simple closed-form solution for the HJB equation and optimal control. Thus, to measure performance, we compare mean terminal CARA utility and certainty equivalent between the learned control and several reference controls using the same benchmark technique as in the previous section. The reference controls are:
\begin{itemize}
    \item The clipped control $\pi_{\mathrm{clip}}(t,x) = \mathrm{ReLU}(\pi^*(t,x))$, where $\pi^*(t,x)$, given in \eqref{merton_solution}, is the optimal solution when short selling is allowed. This corresponds to naively implementing the optimal control in the classical case except borrowing at interest to cover money that would otherwise have been obtained by shorting stocks.
    \item The zero control $\pi_{\mathrm{zero}} \equiv \mathbf{0}_n$, which corresponds to leaving all wealth in the riskless asset.
    \item The uniform control $\pi_{\mathrm{unif}}(t,x) = \frac{x}{n}\mathbf{1}_n$, which corresponds to investing all available wealth equally in the $n$ risky assets without borrowing.
    \item The constrained optimization control $\pi_{\mathrm{\lambda}}$, which comes from solving the optimization problem
    \begin{equation*}
        \pi_\lambda \equiv \arg \max_{\pi \in [0,+\infty)^n} \curlbrac{\inprod{\pi}{\mu - r\mathbf{1}_n} - \lambda \pi^\intercal\sigma\sigma^\intercal\pi}.
    \end{equation*}
    for some $\lambda > 0$ using gradient descent. This estimate is inspired by the observation that $V_x(t,x) > 0$ and $V_{xx}(t,x) < 0$ for all $(t,x) \in \overline\Omega$, and so the above optimization bears resemblance to
    \begin{equation*}
        \sup_{\pi \in [0,+\infty)^n} \curlbrac{\inprod{\pi}{\mu - r\mathbf{1}_n}V_x(t,x) + \frac{1}{2}\pi^\intercal\sigma\sigma^\intercal\pi V_{xx}(t,x)}.
    \end{equation*}
    Substituting the solution \eqref{merton_solution} to the unconstrained problem as an approximation, we see that we should expect $\lambda \approx -\frac{V_{xx}(t,x)}{2V_x(t,x)} \approx \frac{p}{2} = \frac{1}{2}$ to give the optimal control (though it is not quite constant in time).
\end{itemize}

Table \ref{cara_no_short_simulated_paths} shows that the learned control performs very well, although is still marginally inferior to the constrained optimization control $\pi_\lambda$ when $\lambda = 0.5$, at which parameter it is very near to the theoretical optimal control. We note that the extremely low mean CARA value $-47.24$ and certainty equivalent $-3.0403$ are not typos, but rather result from aggressive borrowing at interest when naively adapting the optimal control from when short selling was available. The implication is that an investor with constant risk aversion would be willing to lose a substantial deterministic amount of wealth rather than risk investing according to $\pi_{\mathrm{clip}}$ for even a month.

\begin{table}[]
\centering
\begin{tabular}{@{}llllllll@{}}
\toprule
Control                & Learned control & $\pi_{\mathrm{clip}}$ & $\pi_{\mathrm{zero}}$ & $\pi_{\mathrm{unif}}$ & $\pi_{\lambda =0.25}$ & $\pi_{\lambda =0.5}$ & $\pi_{\lambda =1}$ \\ \midrule
$M_{\mathrm{CARA}}$ & $-0.3732$ & $-47.24$ & $-0.4212$ & $-0.4188$ & $-0.4225$ & $-0.3642$ & $-0.3775$\\
$M_{\mathrm{CE}}$ & $1.1105$ & $-3.0403$ & $0.9894$ & $0.9966$ & $0.9881$ & $1.1349$ & $1.0987$\\
\bottomrule
\end{tabular}
\caption{Mean CARA utility and certainty equivalent for several controls with no short selling under parameters calibrated to real data} \label{cara_no_short_simulated_paths}
\end{table}

\subsubsection{Finite-horizon terminal wealth optimization with different interest rates}

We now consider another variant of the Merton problem where the agent earns a lower interest rate $r_{\mathrm{lend}}$ in the risk-free asset than the $r_{\mathrm{borrow}}$ he suffers when borrowing. To model this, the SDE dynamics change to
\begin{equation}
    dX_t= (\inprod{\pi_t}{\mu} + r_{\mathrm{lend}}(X_t - \inprod{\pi_t}{\mathbf{1}_n})_+ - r_{\mathrm{borrow}}(X_t - \inprod{\pi_t}{\mathbf{1}_n})_-)dt + \inprod{\pi_t}{\sigma dW_t}, \quad X_0 = x,
\end{equation}
and thus the HJB equation is
\begin{align}\label{eq_merton_hjb_diff_rates}\begin{split}
    -V_t(t,x) &= \sup_{\pi \in \R^n} \curlbrac{\paren{\inprod{\pi_t}{\mu} + r_{\mathrm{lend}}(x - \inprod{\pi_t}{\mathbf{1}_n})_+ - r_{\mathrm{borrow}}(x - \inprod{\pi_t}{\mathbf{1}_n})_-}V_x(t,x) + \frac{1}{2}\pi^\intercal\sigma\sigma^\intercal\pi V_{xx}(t,x)} \\
    V(T,x) &= U_p(x). \end{split}
\end{align}

Like with the previous example, this problem admits no general closed-form solution. We thus benchmark performance on the calibrated parameters using the Monte Carlo technique described before. The reference controls are $\pi^*(t,x)$ from \eqref{merton_solution}, representing the optimal control when the lending and borrowing rates are the same, $\pi_\mathrm{unif}$, representing investing all wealth equally across the risky assets, and $\pi_\text{zero}$, representing investing all wealth in the riskless asset.

Table \ref{cara_different_rates_simulated_paths} shows that Algorithm \ref{algo_ac} outperforms all three of these approaches. This implies that Algorithm \ref{algo_ac} is sensitive enough to detect the $r_{\mathrm{borrow}} - r_{\mathrm{lend}} = 0.05$ gap and utilize it in learning a control.

\begin{table}[]
\centering
\begin{tabular}{@{}lllll@{}}
\toprule
Control                & Learned control & $\pi^*(t,x)$ in \eqref{merton_solution} & $\pi_{\mathrm{unif}}$ & $\pi_{\text{zero}}$ \\ \midrule
$M_{\mathrm{CARA}}$ & $-0.2351$ & $-0.2370$ & $-0.4188$ & $-0.4212$ \\
$M_{\mathrm{CE}}$ & $1.5759$ & $1.5679$ & $0.9966$ & $0.9894$ \\
\bottomrule
\end{tabular}
\caption{Mean CARA utility and certainty equivalent for several controls with $r_\text{lend} = 0.05$ and $r_{\text{borrow}} = 0.10$ under calibrated parameters. Notice that the learned control beats naively applying $\pi^*(t,x)$ from \eqref{merton_solution}, which is here calculated using $r = r_{\text{lend}} = 0.05.$ } \label{cara_different_rates_simulated_paths}
\end{table}

\subsubsection{Finite-horizon terminal wealth optimization with drawdown loss}

We now consider a path-dependent variant of the calibrated Merton problem where the agent is penalized for spending time below a recent high-water mark. To keep the problem Markovian, we augment the state by a smooth high-water proxy $Y_t$ which catches up to wealth only when wealth exceeds it:
\begin{equation*}
    dY_t = \kappa (X_t-Y_t)_+dt, \quad Y_0 = x, \quad \kappa > 0.
\end{equation*}
The instantaneous drawdown proxy is then $D_t = (Y_t-X_t)_+$. We take the objective to be
\begin{align}\label{eq_drawdown_objective}
    V(t,x,y) = \sup_{\pi\in\mathcal{U}} \E\brac{U_p\paren{X_T^\pi - \lambda\int_t^T D_s^\pi ds} \mid X_t=x, Y_t=y},
\end{align}
where $\lambda > 0$ scales the drawdown penalty. The corresponding HJB equation is
\begin{align}\label{eq_merton_hjb_drawdown}\begin{split}
    0 &= V_t(t,x,y) + \sup_{\pi \in \R^n} \curlbrac{(rx + \inprod{\pi}{\mu-r\mathbf{1}_n})V_x(t,x,y) + \kappa(x-y)_+V_y(t,x,y) + \frac{1}{2}\pi^\intercal\sigma\sigma^\intercal\pi V_{xx}(t,x,y)} \\
    &\qquad\qquad + p\lambda(y-x)_+V(t,x,y), \\
    V(T,x,y) &= U_p(x).
\end{split}\end{align}
We use the calibrated S\&P 100 parameters described above with $n=100$, together with $\lambda=\kappa=10$. The sampling measure is
$$\mu = \mathrm{Unif}([0,T]) \times \mathcal{N}(1, 0.5^2) \times \mathrm{Gamma}(\alpha = 1, \lambda =1),$$
where the second component represents the wealth coordinate $x$ and the third represents the high-water proxy $y$.

In addition to the learned control, we compare against the classical Merton control $\pi^*$ from \eqref{merton_solution}, the zero control $\pi_{\mathrm{zero}}$, the uniform control $\pi_{\mathrm{unif}}$, and a simple drawdown-throttled Merton control
\begin{equation*}
    \pi_{\mathrm{dd}}(t,x,y) = \exp\paren{-5(y-x)_+}\pi^*(t,x).
\end{equation*}
For each control we report the mean drawdown-adjusted CARA utility, the corresponding mean drawdown-adjusted certainty equivalent, the mean integrated drawdown $\E[\int_0^T D_tdt]$, and the terminal-wealth certainty equivalent which ignores the drawdown penalty (i.e.~certainty equivalent if $\lambda = 0$ in \eqref{eq_drawdown_objective}).

\begin{table}[]
\centering
\begin{tabular}{@{}llllll@{}}
\toprule
Control & Learned control & $\pi^*(t,x)$ in \eqref{merton_solution} & $\pi_{\mathrm{zero}}$ & $\pi_{\mathrm{unif}}$ & $\pi_{\mathrm{dd}}$ \\ \midrule
Mean drawdown-adjusted utility & $-0.3009$ & $-0.3084$ & $-0.4212$ & $-0.4216$ & $-0.3174$ \\
Mean drawdown-adjusted CE & $1.3380$ & $1.3100$ & $0.9894$ & $0.9881$ & $1.2726$ \\
Mean integrated drawdown & $3.820\times 10^{-3}$ & $7.849\times 10^{-3}$ & 0 & $7.996\times 10^{-4}$ & $8.036\times 10^{-3}$ \\
Mean terminal wealth CE & $1.4482$ & $1.6133$ & $0.9894$ & $0.9966$ & $1.4178$ \\
\bottomrule
\end{tabular}
\caption{Performance of several controls in the drawdown-loss Merton problem with calibrated parameters and $\lambda=\kappa=10$.} \label{cara_drawdown_simulated_paths}
\end{table}

Table \ref{cara_drawdown_simulated_paths} shows that the learned control responds to the drawdown penalty while retaining much of the terminal-wealth performance of the classical Merton policy. Relative to $\pi^*$, it cuts mean integrated drawdown by about 51\% (from $7.85\times10^{-3}$ to $3.82\times10^{-3}$) and improves the drawdown-adjusted certainty equivalent from $1.3100$ to $1.3380$ at the cost of a lower terminal-wealth certainty equivalent ($1.4482$ versus $1.6133$). The zero and uniform controls keep drawdowns small only by forgoing risky investment, and their drawdown-adjusted certainty equivalents remain close to 1, the mean of $\rho = \mathcal{N}(1, 0.5^2)$. The simple drawdown-throttled Merton improves on the zero and uniform controls, but is beaten on the penalized objective by both the unthrottled Merton control and the learned control.

This example thus illustrates that Algorithm \ref{algo_ac} can be effectively applied to a path-dependent objective given some smoothing.

\subsubsection{Rolling realized-return backtest on S\&P 100 stocks}

As our final Merton example, we test how the learned controls behave when the calibrated parameters are estimated from a trailing five-year window and then applied to the following year's realized stock returns. We take 10 real US stocks.\footnote{Namely, ACN, BIIB, COF, DD, GILD, HD, NEE, PM, TXN, and USB. These were randomly selected from the 94 stocks in the S\&P 100 on January 1, 2020 for which uninterrupted daily close data is available 2015--2025.}

For each test year $y \in \{2020,2021,\dots,2025\}$, we estimate $\mu$ and $\sigma\sigma^\intercal$ from the daily adjusted log-returns of the five previous calendar years (years $y-5$ through $y-1$) and train a fresh actor-critic pair for the classical CARA Merton problem with $T=1$ and $p=1$. The risk-free rate for the control deployed in year $y$ is the midpoint of the Federal Funds target range on January 1 of that year. We then evaluate the learned control and the analytic optimal control from \eqref{merton_solution} on the realized adjusted returns of year $y$, starting from random initial wealth values sampled from $\rho = \mathcal{N}(1, 0.5^2)$. Each row of the table is thus an independent Merton problem calibrated to a distinct five-year window, sharing only the stock tickers, implementation, and hyperparameters. Conditioned on a realized historical return path and an initial wealth, the terminal wealth is deterministic, so $M_{\mathrm{CE}}$ collapses towards representing $$M_{\mathrm{CE}}(\pi) = M_{\mathrm{TW}}(\pi) = \int_{-\infty}^\infty X_T^{x,\pi}d\rho(x),$$ the mean (deterministic) terminal wealth integrated over starting wealths.

\begin{table}[H]
\centering
\small
\begin{tabular}{@{}llllllll@{}}
\toprule
Fit window & Test year & Rate $r$ & Actor RE & Learned $M_{\mathrm{CARA}}$ & Learned $M_{\mathrm{TW}}$ & Analytic $M_{\mathrm{CARA}}$ & Analytic $M_{\mathrm{TW}}$ \\ \midrule
2015--2019 & 2020 & $0.01625$ & $1.76\times10^{-2}$ & $-1.803\times10^{-2}$ & $4.140$ & $-5.741\times10^{-3}$ & $5.285$ \\
2016--2020 & 2021 & $0.00125$ & $4.09\times10^{-3}$ & $-9.267\times10^{-2}$ & $2.552$ & $-8.172\times10^{-2}$ & $2.628$ \\
2017--2021 & 2022 & $0.00125$ & $1.19\times10^{-2}$ & $-2.967$ & $-1.033$ & $-4.248$ & $-1.325$ \\
2018--2022 & 2023 & $0.04375$ & $3.36\times10^{-4}$ & $-0.5563$ & $0.7151$ & $-0.5565$ & $0.7150$ \\
2019--2023 & 2024 & $0.05375$ & $4.08\times10^{-4}$ & $-0.1887$ & $1.805$ & $-0.1864$ & $1.815$ \\
2020--2024 & 2025 & $0.04375$ & $4.44\times10^{-4}$ & $-0.1876$ & $1.812$ & $-0.1837$ & $1.826$ \\ \midrule
& Average & & & $-0.6684$ & $1.665$ & $-0.8770$ & $1.824$ \\
\bottomrule
\end{tabular}
\caption{Out-of-sample realized return performance for rolling Merton controls fit to trailing five-year windows.} \label{cara_rolling_realized_sp100}
\end{table}

The learned control closely tracks the analytic optimal control in both actor relative error and in empirical returns for the test year. What little deviation there is in the learned control from the analytic optimal control tends towards caution, as the learned control performs better in bear years (e.g.~2022) and worse in bull years (e.g.~2020) due to its decreased leverage. The root cause is that the performance of the analytic optimal control is not robust to errors in parameter estimation or model dynamics. Thus, error in learning said control, especially in the direction of risk-aversion or regularization, is not necessarily a bad outcome from a practical perspective.

\subsection{Short-term optimal liquidation}

We now analyze a short-term optimal liquidation model which is in essence a relaxed variant of that in \cite{Schied28092010}. Here, $X \in [0,+\infty)$ represents the cash an agent has, $Q \in [0,+\infty)^n$ represents the number of shares he owns in $n$ different assets, $S \in [0,+\infty)^n$ represents the prices of those assets, and $u \in [0,+\infty)^n$ represents the agent's rate of selling them. The underlying dynamics are
\begin{align*}
    dS_t &= (\mu-\Gamma u_t) dt + \sigma dW_t\\
    dX_t &= \inprod{u_t}{S_t} dt - \inprod{u_t}{\Lambda u_t}dt\\
    dQ_t &= -u_tdt,
\end{align*}
where $\sigma,\Gamma,\Lambda \in S^{n\times n}_{\geq0}$ are positive semi-definite (representing volatility, permanent price impact, and temporary price impact, respectively) and $\mu \in \R^n$. The agent's objective is to maximize the $p$-CARA final utility
\begin{equation*}
    V^\pi(t,x,q,s) = \E\brac{U_p(X_T^\pi + \inprod{Q_T^\pi}{S_T} - \inprod{Q_T^\pi}{\Psi Q_T^\pi}) \mid X_t^\pi = x, Q_t^\pi = q, S_t = s}.
\end{equation*}
Here, the positive semi-definite matrix $\Psi \in S_{\geq0}^{n\times n}$ determines a (potentially large) penalty the agent incurs for retaining any assets at time $T$.

The dimensionality of the system can be reduced by defining the net-worth process
\begin{equation*}
    Y_t = X_t + \inprod{Q_t}{S_t}.
\end{equation*}
Then using that $Q_t$ is monotone and therefore of finite variation, we have the dynamics
\begin{align*}
    dY_t &= dX_t + \inprod{Q_t}{dS_t} + \inprod{S_t}{dQ_t} \\
    &= \inprod{u_t}{S_t} dt - \inprod{u_t}{\Lambda u_t}dt + \inprod{Q_t}{\mu - \Gamma u_t}dt + \inprod{Q_t}{\sigma dW_t} + \inprod{S_t}{-u_t}dt \\
    &= (-\inprod{u_t}{\Lambda u_t} + \inprod{Q_t}{\mu - \Gamma u_t})dt + \inprod{Q_t}{\sigma dW_t},
\end{align*}
and the objective can be written as
\begin{equation}
    J(t,y,q) = \E[U_p(Y_T - \inprod{Q_T}{\Psi Q_T}) \mid Y_t = y, Q_t = q].
\end{equation}
We optimize over the set of feedback controls that involve selling slower than linearly given the remaining amount of shares. That is to say, each $u \in \mathcal{U}_{\mathrm{SL}}$ obeys the bounds
\begin{equation*}
    0 \leq u_i(t,y,q ) \leq\frac{q_i}{T-t}, \quad i \in \{1,2, \dots, n\}.
\end{equation*}
The corresponding HJB equation on $\overline \Omega = [0,T] \times [0,+\infty) \times [0,+\infty)^n$ is then (omitting arguments for brevity)
\begin{align}\begin{split}
    0 &= V_t(t,y,q) + \sup_{u\in\mathcal{U}_{\mathrm{SL}}}\curlbrac{(-\inprod{u}{\Lambda u} +\inprod{q}{\mu - \Gamma u})V_y - \inprod{u}{\nabla_q V} + \frac{1}{2}\mathrm{Tr}\brac{\sigma q(\sigma q )^\intercal} V_{yy}} \\
    V(T,y,q) &= U_p(y - \inprod{q}{\Psi q}). \end{split}
\end{align}
The appropriate interpolation and auxiliary functions are
\begin{align*}
    \overline g(t,y,q) &= U_p(y - \inprod{q}{\Psi q}) \\
    \eta(t,y,q) &= T-t.
\end{align*}

\paragraph{Results} For want of an analytic solution, we use a Monte Carlo technique akin to that of the previous section to benchmark the learned control\footnote{In particular, we sample 2,000 initial points for our Monte Carlo evaluations and simulate 500 paths from each initial point, for 1,000,000 simulated paths total. We evaluate against the results of Algorithm \ref{algo_ac} when trained with the same hyperparameters as the previous example (excepting the sampling measure $\mu$).}. In particular, we compare it against the linear selling control
\begin{equation*}
    (u_\mathrm{lin}(t,y,q))_i = \frac{q_i}{T-t}, \quad i \in \{1,2, \dots, n\}
\end{equation*}
and the Almgren-Chriss style \cite{almgren2001optimal} control\footnote{This control involves faster-than-linear selling and therefore does not belong to $\mathcal{U}_{\mathrm{SL}}$. It is nonetheless a valid control to compare against, though.}
\begin{equation*}
    (u_\mathrm{AC}(t,y,q))_i = \Gamma_{ii} \coth(\Gamma_{ii}({T-t}))q_i, \quad i \in \{1,2, \dots, n\}.
\end{equation*}
With the market parameters
\begin{equation*}
    n = 10, \quad T = 1, \quad p = 0.05
\end{equation*}
and dynamics
\begin{equation*}
    \mu = 0.15 \times \mathbf{1}_n, \quad \sigma = 0.2 \times \mathrm{Id}_{n\times n}, \quad \Gamma = 0.05 \times \mathrm{Id}_{n\times n}, \quad \Lambda = 0.05 \times \mathrm{Id}_{n\times n}, \quad \Psi = 0.1 \times \mathrm{Id}_{n\times n}
\end{equation*}
and training measure
$$\mu = \mathrm{Unif}([0,T]) \times \mathrm{Gamma}(\alpha = 10, \lambda = 1) \times (\otimes_{i=1}^n\mathrm{Gamma}(\alpha = 1, \lambda = 1)),$$
we obtain the following results:
\begin{table}[H]
\centering
\begin{tabular}{@{}llll@{}}
\toprule
Control                & Learned control & $u_{\mathrm{lin}}$ & $u_{\mathrm{AC}}$ \\ \midrule
$M_{\mathrm{CARA}}$ & $-0.6188$ & $-0.6395$ & $-0.6395$\\
$M_{\mathrm{CE}}$ & $9.8409$ & $9.1914$ & $9.1913$\\
\bottomrule
\end{tabular}
\caption{Mean CARA utility and certainty equivalent for the short-term optimal liquidation problem} \label{cara_liquidation}
\end{table}
This shows that Algorithm \ref{algo_ac} is able to meaningfully learn a good control in the above rather complicated and somewhat high-dimensional optimal liquidation problem.

\section{Proofs of Results} \label{convergence_proof}
In this final section, we assume without loss of generality that $N=N^*$ and denote $Q_t^N := Q_{\phi_t}^N$ and $U_t^N := U_{\theta_t}^N$. We also fix an arbitrary training time bound $T>0$. The constants $C>0$ may vary from line to line (even within the same chain of inequalities) and depend on $T$, but are always independent of $N$. For simplicity, the proof is done with a one-dimensional action space (i.e. $U_t^N: \overline{\Omega} \to \R$), but essentially nothing other than notation changes if $U_t^N$ maps to $\R^k$ for $k \geq 2$.

As a reminder, we are operating under Assumptions \ref{assumption_bounded_second_derivative}, \ref{assumption_eta}, \ref{assumption_g}, and \ref{assumption_neural_net}. There is a fixed neural network scaling parameter $\beta\in \paren{\frac{1}{2},1}$ and a fixed family of smooth truncation functions $\{\psi^N\}_{N\geq1}$ (cf. Definition \ref{definition_smooth_truncation}) with parameter $\delta \in \paren{0, \frac{1-\beta}{4}}$.

We begin by presenting the following lemmas before starting the main proof.

\begin{proposition}[Lipschitz behaviors]\label{proposition_lipschitz}
There exists a constant $C>0$ such that
\begin{enumerate}
    \item For any $x \in \Omega$, any admissible feedback control $u \in \mathcal{U}$, and any functions $f_1, f_2 \in \mathcal{D}(\mathcal{L}^{u})$
    \begin{align}
      \modu{\mathcal{L}^{u}f_1(x)-\mathcal{L}^{u}f_2(x)}\leq C \sum_{|\alpha|\leq 2}\modu{D_\alpha f_1(x)-D_\alpha f_2(x)}
      \end{align}
      
      \item For any $x \in \Omega$, any admissible feedback controls $u_1, u_2 \in \mathcal{U}$, and any $f \in \mathcal{D}(\mathcal{L}^{u_1})\cap \mathcal{D}(\mathcal{L}^{u_2})$
      \begin{align}
        \modu{\g^{u_1}f(x)-\g^{u_2}f(x)} \leq C\modu{u_1(x)-u_2(x)} \big( 1+\sum_{|\alpha|\leq 2}\modu{D_\alpha f(x)} \big)
      \end{align}

      \item For any $x \in \Omega$, any $u\in \mathcal{U},$ and any $f_1,f_2 \in C^2,$
      \begin{align}
          \modu{\partial_u H(u, f_1)(x) - \partial_u H(u,f_2)(x)} \leq C\sum_{\modu{\alpha} \leq 2}\modu{D_\alpha f_1(x) - D_\alpha f_2(x)}
      \end{align}
      
      \item For any $x \in \Omega$, any $u_1, u_2 \in \mathcal{U}$, and any $f \in C^2,$
      \begin{align}
          \modu{\partial_u H(u_1, f)(x)-\partial_u H(u_2, f)(x)} \leq C \modu{u_1(x) - u_2(x)}\big( 1+\sum_{|\alpha|\leq 2}\modu{D_\alpha f(x)} \big)
      \end{align}
\end{enumerate}
\end{proposition}
\begin{proof}
Each of these properties ultimately derives from Assumption \ref{assumption_bounded_second_derivative} -- namely that there exists a constant that bounds $\{b_i, \Phi \Phi^\intercal_{ij}, c\}_{i,j}$ and their relevant first-order and second-order derivatives.
\begin{enumerate}
    \item Using the bound on the coefficients directly,
    \begin{align*}
        \modu{\mathcal{L}^{u}f_1(x)-\mathcal{L}^{u}f_2(x)} &\leq \modu{\sum_{i}b_i(x,u(x))\partial_i (f_1-f_2)(x)+\frac{1}{2}\sum_{i,j}{\Phi\Phi^\intercal}_{ij}(x,u(x))\partial^2_{ij} (f_1-f_2)(x) - \gamma(f_1-f_2)(x)}\\ 
        &\leq C\sum_{|\alpha|\leq 2}\modu{D_\alpha (f_1-f_2)(x)}.
    \end{align*}
    
    \item By the mean value theorem and the bound on the first derivatives in $u$,
    \begin{align*}
        \modu{\g^{u_1}f(x)-\g^{u_2}f(x)} &\leq \bigg|\sum_{i}(b_i(x,u_1(x) )-b_i(x,u_2(x)))\partial_i f(x) \\
        &\quad+\frac{1}{2}\sum_{i,j}({\Phi\Phi^\intercal}_{ij}(x,u_1(x)) - {\Phi\Phi^\intercal}_{ij}(x,u_2(x)))\partial^2_{ij} f(x) + c(x,u_1(x)) - c(x,u_2(x))\bigg| \\
        &\leq C\modu{u_1(x)-u_2(x)} \big( 1+\sum_{|\alpha|\leq 2}\modu{D_\alpha f(x)} \big).
    \end{align*}
    
    \item Using the bound on the first derivatives in $u$,
    \begin{align*}
        \modu{\partial_u H(u, f_1)(x) - \partial_u H(u,f_2)(x)} &\leq \modu{\sum_{i}\partial_u b_i(x, u(x))\partial_i (f_1-f_2)(x) + \frac{1}{2}\sum_{i,j} \partial_u (\Phi\Phi^\intercal)_{ij}(x,u(x)) \partial_{ij}(f_1-f_2)(x)} \\
        &\leq C\sum_{\modu{\alpha} \leq 2}\modu{D_\alpha (f_1 - f_2)(x)}.
    \end{align*}

    \item By the mean value theorem and the bound on the second derivatives in $u$,
    \begin{align*}
        \modu{\partial_u H(u_1, f)(x)-\partial_u H(u_2, f)(x)} &\leq \bigg| \sum_i (\partial_u b_i(x,u_1(x))-\partial_ub_i(x,u_2(x)))\partial_i f(x) \\
        &\quad + \frac{1}{2}\sum_{i,j} (\partial_u (\Phi\Phi^\intercal)_{ij}(x,u_1(x)) - \partial_u(\Phi\Phi^\intercal)_{ij}(x,u_2(x)))\partial_{ij}f(x) + \partial_u c(x,u_1(x)) - \partial_u c(x,u_2(x)) \bigg|\\
        &\leq C \modu{u_1(x) - u_2(x)}\big( 1+\sum_{|\alpha|\leq 2}\modu{D_\alpha f(x)} \big).
    \end{align*}
\end{enumerate}
\end{proof}

\begin{lemma}[Parameter growth bounds] \label{lemma_param_growth}
    Following the dynamics of \eqref{phi_gradient} and \eqref{theta_gradient}, there exists a constant $C > 0$ such that, for any $0\leq t\leq T$, the critic network parameters satisfy \begin{align}
    \modu{c_t^i - c_0^i} \leq CN^{\delta + \beta - 1},\quad \modu{w_t^{i,j}-w_0^{i,j}} \leq CN^{\delta+\beta-1},\quad \modu{b_t^i - b_0^i} \leq CN^{\delta + \beta - 1},
\end{align}
and the actor network parameters satisfy
\begin{align}
    \modu{h_t^i-h_0^i}\leq CN^{\delta + \beta - 1},\quad \modu{v_t^{i,j} -v_0^{i,j}}\leq CN^{\delta + \beta - 1},\quad \modu{z_t^i-z_0^i}\leq CN^{\delta + \beta - 1}.
\end{align}
In particular, because of Assumption \ref{assumption_neural_net},
\begin{align}
    |c_t^i|\leq C,\quad \mathbb{E}|w_t^{i,j}|\leq C,\quad \mathbb{E}|b_t^i|\leq C,
\end{align}
and
\begin{align}
    |h_t^i|\leq C,\quad \mathbb{E}|v_t^{i,j}|\leq C,\quad \mathbb{E}|z_t^i|\leq C.
\end{align}
\end{lemma}
\begin{proof}
    \begin{enumerate}
    \item It follows directly from training dynamics \eqref{phi_gradient} and the definition of $Q_t^N$ that
    \begin{align*}
        \frac{dc^i_t}{dt} = \omega_t N^{\beta-1} \int_\Omega F^N\paren{\g^{U_t^N}Q_t^N(x)}\eta(x)\sigma(w_t^i \cdot x + b_t^i)d\mu(x).
    \end{align*}
    Since $\eta$, $\sigma$, and $\omega_t$ are bounded and $\modu{F^N(x)} \leq 2N^\delta$, this gives
    \begin{align*}
        \modu{\frac{dc_t^i}{dt}} \leq N^{\beta-1}N^{\delta}C,
    \end{align*}
    so
    \begin{align*}
        \modu{c_t^i-c_0^i} \leq tCN^{\delta+\beta-1} \leq CN^{\delta+\beta-1}.
    \end{align*}
    Thus
    \begin{align*}
        \modu{c_t^i} \leq \modu{c_0^i} + \modu{c_t^i - c_0^i} \leq \modu{c_0^i} + CN^{\delta + \beta - 1} \leq C,
    \end{align*}
    where we have used Assumption \ref{assumption_neural_net} to bound $\modu{c_0^i}$.

    Similarly, for any individual weight $w_t^{i,j}$ in a weight vector $w_t^i$, we have
    \begin{align*}
        \frac{dw_t^{i,j}}{dt} = \omega_t N^{\beta-1} \int_\Omega F^N\paren{\g^{U_t^N}Q_t^N(x)}\eta(x) c_t^i \sigma'(w_t^i \cdot x + b_t^i)x_jd\mu(x).
    \end{align*}
    By $\modu{c_t^i} \leq C$, the boundedness of $\eta$, $\sigma'$, and $\omega_t$, and the fact that $\modu{F^N(x)} \leq 2N^\delta$, this gives
    \begin{align*}
        \modu{w_t^{i,j} - w_0^{i,j}} \leq tCN^{\delta+\beta-1} \leq CN^{\delta+\beta-1}.
    \end{align*}
    Thus
    \begin{align*}
        \modu{w_t^{i,j}} \leq \modu{w_0^{i,j}} + \modu{w_t^{i,j} - w_0^{i,j}} \leq \modu{w_0^{i,j}} + CN^{\delta + \beta - 1}.
    \end{align*}
    Taking expectations shows $\E\modu{w_t^{i,j}} \leq C$.

    In the same way, for any bias $b_t^i$, we have
    \begin{align*}
        \frac{db_t^i}{dt} = \omega_t N^{\beta-1} \int_\Omega F^N\paren{\g^{U_t^N}Q_t^N(x)}\eta(x) c_t^i \sigma'(w_t^i \cdot x + b_t^i)d\mu(x).
    \end{align*}
    This gives
    \begin{align*}
        \modu{b_t^i - b_0^i} \leq tCN^{\delta+\beta-1} \leq CN^{\delta+\beta-1},
    \end{align*}
    and so
    \begin{align*}
        \modu{b_t^i} &\leq \modu{b_0^i} + \modu{b_t^i - b_0^i} \leq \modu{b_0^i} + CN^{\delta+\beta-1}.
    \end{align*}
    Taking expectations shows $\E\modu{b_t^i} \leq C$.
    
    \item It follows from the training dynamics \eqref{theta_gradient} that the actor network's parameter dynamics are given by
    \begin{align*}
        \frac{dh^i_t}{dt} &= -\alpha_t N^{\beta-1} \int_\Omega \psi^N(\partial_u H(U_t^N,Q_t^N))\sigma(v_t^i \cdot x + z_t^i)d\mu(x), \\
        \frac{dv^{i,j}_t}{dt} &= -\alpha_t N^{\beta-1} \int_\Omega \psi^N(\partial_u H(U_t^N,Q_t^N))h_t^i\sigma'(v_t^i \cdot x + z_t^i)x_jd\mu(x), \\
        \frac{dz^i_t}{dt} &= -\alpha_t N^{\beta-1} \int_\Omega \psi^N(\partial_u H(U_t^N,Q_t^N))h_t^i\sigma'(v_t^i \cdot x + z_t^i)d\mu(x).
    \end{align*}
    Using the boundedness of $\alpha_t$ and repeating the above reasoning gives the desired bounds.
    \end{enumerate}
\end{proof}
\begin{lemma}\label{lemma_kernel_regularity}
    There exists a constant $C > 0$ such that the kernels $A$ and $B$ defined in \eqref{kernel_A_def} and \eqref{kernel_B_def} obey
    \begin{align*}
        \modu{A(x,y)} \leq C, \qquad \sum_{\modu{\alpha} \leq 2} \modu{D_\alpha^y B(x,y)} \leq C,
    \end{align*}
    for all $(x,y) \in \overline\Omega\times\overline\Omega.$ In particular, the linear operators $\mathcal{A} : C^0(\overline \Omega) \to C^0(\overline \Omega)$ and $\mathcal{B} : C^0(\overline \Omega) \to  C^2(\overline \Omega)$ defined in \eqref{def_op_A}, \eqref{kernel_B_op_def} are bounded with operator norms $\nor{\A}, \nor{\mathcal B} \leq C$.
\end{lemma}
\begin{proof}
    That $\modu{A(x,y)} \leq C$ follows immediately from the bounds on $\sigma, \sigma'$, the boundedness of $\Omega$, and the assumption that the distribution of $c_0$ has compact support. For the bound on $B$, observe that for any multi-index $\alpha$ such that $\modu{\alpha} \leq 2,$
    \begin{align*}
        \modu{D_\alpha^y B(x,y)} &= \modu{D_\alpha^y \brac{\eta(x)\eta(y)A(x,y)}} \\
        &= \eta(x) \modu{D_\alpha^y \paren{\eta(y) \E_{c,w,b} \brac{\sigma(w\cdot x + b)\sigma(w\cdot y + b) + c^2\sigma'(w\cdot x + b)\sigma'(w\cdot y + b) (x\cdot y + 1)}}} \\
        &= \eta(x) \modu{\E_{c,w,b} \brac{D_\alpha^y \paren{\eta(y)\paren{\sigma(w\cdot x + b)\sigma(w\cdot y + b) + c^2\sigma'(w\cdot x + b)\sigma'(w\cdot y + b) (x\cdot y + 1)} }}} \\
        &\leq C,
    \end{align*}
    where we are able to interchange the derivative and the expectation and show that the expectation is bounded because of Assumption \ref{assumption_neural_net} and that $\eta \in C^3_b$ and $\sigma \in C^4_b.$ The boundedness of the operators $\A$, $\mathcal{B}$ then follows from differentiation under the integral and the fact that $\mu(\Omega) < \infty$.
\end{proof}

\begin{lemma}\label{QQ}
    There exists a constant $C>0$ such that, for any $y \in \overline\Omega$ and $t \in [0,T]$, we have 
    \begin{align}
        \sum_{\modu{\alpha} \leq 2} |D_\alpha Q_t(y)| \leq C.
    \end{align}
\end{lemma}
\begin{proof}
    From the limit ODE \eqref{limit_ODE}, we have
    \begin{align*}
        Q_t(y) - \overline{g}(y) = \int_0^t \omega_s\int_\Omega \mathcal{L}^{U_s}Q_s(x) B(x,y)d\mu(x)ds.
    \end{align*}
    Then by the boundedness of $\omega_t$ and $B$ and its derivatives, it follows that
    \begin{align*}
        \sum_{\modu{\alpha} \leq 2} |D_\alpha (Q_t(y)-\overline{g}(y))| &\leq \int_0^t \omega_s \int_\Omega |\mathcal{L}^{U_s}Q_s(x)| \modu{\sum_{\modu{\alpha} \leq 2} D_\alpha^y B(x,y)}d\mu(x)ds\\
        &\leq C\int_0^t \int_\Omega |\mathcal{L}^{U_s}Q_s(x)-\mathcal{L}^{U_s}\overline{g}(x)|+|\mathcal{L}^{U_s}\overline{g}(x)| d\mu(x)ds \\
        & \leq C \int_0^t \int_\Omega \sum_{\modu{\alpha} \leq 2} |D_\alpha Q_s(x) - D_\alpha \overline{g}(x)|d\mu(x)ds + C \int_0^t \int_\Omega \sum_{\modu{\alpha}\leq 2} \modu{D_\alpha \overline{g}(x)}d\mu(x)ds \\
        & \leq C \int_0^t \int_\Omega \sum_{\modu{\alpha} \leq 2} |D_\alpha Q_s(x) - D_\alpha \overline{g}(x)|d\mu(x)ds + C.
    \end{align*}
    In particular,
    \begin{align*}
        \max_{x \in \overline{\Omega} }\sum_{\modu{\alpha} \leq 2} \modu{D_\alpha (Q_t(x)-\overline{g}(x))} \leq C\mu(\Omega)\int_0^t \max_{x \in \overline{\Omega}} \sum_{\modu{\alpha} \leq 2} \modu{D_\alpha (Q_s(x)-\overline{g}(x))} ds + C.
    \end{align*}
    Thus, Gronwall's inequality shows that $$\max_{x \in \overline{\Omega}} \sum_{\modu{\alpha} \leq 2}\modu{D_\alpha Q_t(x) - D_\alpha \overline{g}(x)} \leq C$$ for $t\in[0,T]$. Since $\overline{g} \in C^2_b(\overline\Omega)$, this completes the proof.
\end{proof}
\begin{lemma}\label{l7}
Let
\begin{align}
    \Xi_0^N(y):= \modu{U_0^N(y)-U_0(y)} + \sum_{\modu{\alpha}\leq 2}|D_\alpha(Q_0^N-Q_0)(y)|=|U_0^N(y)|+ \sum_{|\alpha|\leq 2}\modu{D_\alpha Q_0^N(y) - D_\alpha \overline{g}(y)}.
\end{align}
Then
\begin{align}
    \lim_{N \to \infty} \E\brac{\int_\Omega \Xi_0^N(y)^2 d\mu(y)} = 0.
\end{align}
\end{lemma}
\begin{proof}
Notice that
\begin{align*}
    \E\brac{U_0^N(y)}=0, \quad \E\brac{D_\alpha Q_0^N(y) - D_\alpha \overline{g}(y)}=0.
\end{align*}
for arbitrary $\modu{\alpha} \leq 2$. Then
\begin{align*}
    \mathbb{E}[U_0^N(y)^2]=\text{Var}[U_0^N(y)]=N^{1-2\beta}\text{Var}_{h,v,z}[h\sigma(v\cdot y+z)] \leq CN^{1-2\beta}
\end{align*}
and
\begin{align*}
    \mathbb{E}[\paren{D_\alpha Q_0^N(y) - D_\alpha \overline{g}(y)}^2]=\text{Var}[D_\alpha (\eta(y)Z_0^N(y))]=N^{1-2\beta}\text{Var}_{c,w,b}[cD_\alpha(\eta(y)\sigma(w\cdot y+b))] \leq CN^{1-2\beta}.
\end{align*}
Therefore
\begin{align*}
    \E\brac{\int_\Omega \Xi_0^N(y)^2 d\mu(y)} &= \int_\Omega
    \mathbb{E}\brac{\Xi_0^N(y)^2} d\mu(y) \\
    &\leq C\int_\Omega \mathbb{E}\brac{U_0^N(y)^2} + \sum_{|\alpha|\leq 2}\mathbb{E}\brac{\paren{D_\alpha Q_0^N(y) - D_\alpha\overline{g}(y)}^2}d\mu(y) \\ &\leq C N^{1-2\beta} \to 0.
\end{align*}
\end{proof}

\begin{lemma} \label{l8}
    Let $A_s^N$ represent the neural tangent kernel of each finite-width network at time $s \in [0,T]$, i.e.
    \begin{align}
        A_s^N(x,y) = \frac{1}{N}\sum_{i=1}^N \nabla_{c^i,w^i,b^i}\brac{c^i_s \sigma(w^i_s\cdot x+b^i_s)}\cdot \nabla_{c^i,w^i,b^i}\brac{c^i_s \sigma(w^i_s\cdot y+b^i_s)},
    \end{align}
    and let $A = \lim_{N\to\infty}A_0^N$ be the infinite-width limit at time 0 (i.e. the kernel in \eqref{kernel_A_def}). Then
    \begin{align}
        \lim_{N \to \infty}\E\brac{\int_0^T \int_{\Omega^2} N^{2\delta}\modu{{A}_s^N(x,y)-A(x,y)}^2d\mu(x)d\mu(y)ds} = 0.
    \end{align}
\end{lemma}
\begin{proof}
It suffices to prove that
\begin{align*}
    \lim_{N \to \infty} \int_0^T \int_{\Omega^2}N^{2\delta}\E\brac{\modu{A_s^N(x,y)-A(x,y)}^2}d\mu(x)d\mu(y)ds = 0.
\end{align*}
First, we notice that
\begin{align*}
    &\modu{A_s^N(x,y)-A^N_0(x,y)} \\&= \bigg|\frac{1}{N}\sum_{i=1}^N \sigma(w^i_s\cdot x + b_s^i)\sigma(w_s^i \cdot y + b_s^i) - \sigma(w^i_0\cdot x + b_0^i)\sigma(w_0^i \cdot y + b_0^i) \\
    &\quad + (x\cdot y +1)(c_s^i)^2\sigma'(w_s^i\cdot x + b_s^i)\sigma'(w_s^i\cdot y + b_s^i) - (x\cdot y +1)(c_0^i)^2\sigma'(w_0^i\cdot x + b_0^i)\sigma'(w_0^i\cdot y + b_0^i)\bigg| \\
    &\leq \frac{1}{N}\sum_{i=1}^N \modu{\sigma(w^i_s\cdot x + b_s^i)\sigma(w_s^i \cdot y + b_s^i) - \sigma(w^i_0\cdot x + b_0^i)\sigma(w_0^i \cdot y + b_0^i)} \\
    &\quad+ \frac{1}{N}\sum_{i=1}^N\modu{(x\cdot y +1)(c_s^i)^2\paren{\sigma'(w_s^i\cdot x + b_s^i)\sigma'(w_s^i\cdot y + b_s^i) - \sigma'(w_0^i\cdot x + b_0^i)\sigma'(w_0^i\cdot y + b_0^i)}} \\
    &\quad + \frac{1}{N}\sum_{i=1}^N\modu{(x\cdot y +1)\sigma'(w_0^i\cdot x + b_0^i)\sigma'(w_0^i\cdot y + b_0^i)\paren{(c_s^i)^2 - (c_0^i)^2}}.
\end{align*}
We deal with each of these three sums independently. For the first, notice that for each $i \in \{1,\dots, N\}$, since $\sigma \in C_b^4$, the mean value theorem combined with Lemma \ref{lemma_param_growth} gives that
\begin{align*}
    \big|\sigma(w^i_s\cdot x + b_s^i)\sigma(w_s^i \cdot y + b_s^i) &- \sigma(w^i_0\cdot x + b_0^i)\sigma(w_0^i \cdot y + b_0^i) \big| \\
    &= \Big| \sigma'(w^*\!\cdot x + b^*)\,\sigma(w^*\!\cdot y + b^*)\paren{x\cdot(w_s^i - w_0^i) + (b_s^i - b_0^i)} \\
    &\qquad\ + \sigma(w^*\!\cdot x + b^*)\,\sigma'(w^*\!\cdot y + b^*)\paren{y\cdot(w_s^i - w_0^i) + (b_s^i - b_0^i)} \Big| \\
    &\leq C\paren{\nor{w_s^i - w_0^i}_{l^1} + \modu{b_s^i - b_0^i}} \\
    &\leq CN^{\delta + \beta - 1},
\end{align*}
where $(w^*,b^*)$ is a point on the line segment between $(w_0^i,b_0^i)$ and $(w_s^i,b_s^i)$. Thus
\begin{align*}
    \frac{1}{N}\sum_{i=1}^N \modu{\sigma(w^i_s\cdot x + b_s^i)\sigma(w_s^i \cdot y + b_s^i) - \sigma(w^i_0\cdot x + b_0^i)\sigma(w_0^i \cdot y + b_0^i)} \leq C N^{\delta + \beta - 1}.
\end{align*}
For the second sum, we apply the mean value theorem to $(w,b)\mapsto \sigma'(w\cdot x+b)\sigma'(w\cdot y+b)$, obtaining for each $i \in \{1,\dots,N\}$,
\begin{align*}
    \big|(x\cdot y +1)(c_s^i)^2&\big(\sigma'(w_s^i\cdot x + b_s^i)\sigma'(w_s^i\cdot y + b_s^i) - \sigma'(w_0^i\cdot x + b_0^i)\sigma'(w_0^i\cdot y + b_0^i)\big)\big| \\
    &= \big|(x\cdot y +1)(c_s^i)^2\big[\sigma''(w^*\!\cdot x + b^*)\sigma'(w^*\!\cdot y + b^*)\paren{x\cdot(w_s^i-w_0^i)+(b_s^i-b_0^i)} \\
    &\qquad + \sigma'(w^*\!\cdot x + b^*)\sigma''(w^*\!\cdot y + b^*)\paren{y\cdot(w_s^i-w_0^i)+(b_s^i-b_0^i)}\big]\big|\\
    &\leq C\paren{\nor{w_s^i - w_0^i}_{l^1} + \modu{b_s^i-b_0^i}}\\
    &\leq C N^{\delta+\beta-1},
\end{align*}
where $(w^*, b^*)$ is a point on the line segment between $(w_0^i, b_0^i)$ and $(w_s^i,b_s^i)$, and we have used that $\sigma', \sigma''$ are bounded together with Lemma \ref{lemma_param_growth}. Thus
\begin{align*}
    \frac{1}{N}\sum_{i=1}^N\modu{(x\cdot y +1)(c_s^i)^2\paren{\sigma'(w_s^i\cdot x + b_s^i)\sigma'(w_s^i\cdot y + b_s^i) - \sigma'(w_0^i\cdot x + b_0^i)\sigma'(w_0^i\cdot y + b_0^i)}} &\leq CN^{\delta+\beta-1}.
\end{align*}
Finally, using $\modu{(x\cdot y +1)\sigma'(w_0^i\cdot x + b_0^i)\sigma'(w_0^i\cdot y + b_0^i)} \le C$ uniformly in $i$, it is straightforward that
\begin{align*}
    \frac{1}{N}\sum_{i=1}^N\big|(x\cdot y +1)\sigma'(&w_0^i\cdot x + b_0^i)\sigma'(w_0^i\cdot y + b_0^i)\paren{(c_s^i)^2 - (c_0^i)^2}\big| \\
    &\leq \frac{C}{N}\sum_{i=1}^N \modu{(c_s^i)^2 - (c_0^i)^2} \\
    &\leq \frac{C}{N} \sum_{i=1}^N  \modu{c_s^i + c_0^i}\modu{c_s^i - c_0^i} \\
    &\leq \frac{C}{N}\sum_{i=1}^N N^{\delta + \beta - 1}
    \leq C N^{\delta + \beta - 1},
\end{align*}
where the last line uses $\modu{c_s^i + c_0^i}\le C$ and $\modu{c_s^i - c_0^i}\le CN^{\delta+\beta-1}$ from Lemma \ref{lemma_param_growth}.
Hence
\begin{align*}
    \modu{A_s^N(x,y)-A^N_0(x,y)} \leq CN^{\delta+\beta-1},
\end{align*}
and in particular
\begin{align*}
    \modu{A_s^N(x,y)-A^N_0(x,y)}^2 \leq CN^{2\delta+2\beta-2}.
\end{align*}
Next, since $A_0^N(x,y)$ is the mean of $N$ independent and identically distributed random variables with finite variance (namely, $\sigma(w^i_0\cdot x + b_0^i)\sigma(w_0^i \cdot y + b_0^i) + (x\cdot y +1)(c_0^i)^2\sigma'(w_0^i\cdot x + b_0^i)\sigma'(w_0^i\cdot y + b_0^i)$), we have
\begin{align*}
    \mathbb{E}\brac{\modu{{A}_0^N(x,y)-A(x,y)}^2}&=\mathbb{E}\brac{\paren{A_0^N(x,y)-\mathbb{E}[A_0^N(x,y)]}^2}=\text{Var}[{A}_0^N(x,y)]
    \leq \frac{C}{N}.
\end{align*}
Combining the inequalities above gives
\begin{align*}
    N^{2\delta} \mathbb{E}\brac{\modu{A_s^N(x,y)-A(x,y)}^2} &\leq 2N^{2\delta} \E\brac{(A_s^N(x,y) - A_0^N(x,y))^2 + (A_0^N(x,y) - A(x,y))^2}\\
    &\leq C N^{4\delta+2\beta-2}+ CN^{2\delta-1}.
    \end{align*}
Thus
\begin{align*}
    \int_0^T \int_{\Omega^2}N^{2\delta}\E\brac{\modu{A_s^N(x,y)-A(x,y)}^2}d\mu(x)d\mu(y)ds \leq C N^{4\delta+2\beta-2}+ CN^{2\delta-1} \to 0.
\end{align*}
\end{proof}

\begin{lemma} \label{l9}
Let $B_s^N(x,y) = \eta(x)\eta(y)A_s^N(x,y)$ and $B = \lim_{N\to\infty} B_0^N$ be the infinite-width limit at time 0 (i.e.~the kernel in \eqref{kernel_B_def}). Denote
\begin{align}
    S_s^N(x,y) := \sum_{|\alpha|\leq 2} \modu{D_\alpha^y B(x,y)-D_\alpha^y B^N_s(x,y)}.
\end{align}
Then
\begin{align}
    \lim_{N\to\infty}\mathbb{E}\brac{\int_0^T \int_{\Omega^2} N^{2\delta} S_s^N(x,y)^2 d\mu(x)d\mu(y)ds} = 0.      
\end{align}
\end{lemma}
\begin{proof}
    The proof is very similar to that of Lemma \ref{l8}. It suffices to show that for arbitrary multi-index $\alpha$ with $\modu{\alpha} \leq 2$, we have
    \begin{align*}
        \lim_{N \to \infty} \int_0^T \int_{\Omega^2}N^{2\delta}\mathbb{E}\brac{\modu{D_\alpha^y B(x,y)-D_\alpha^y B^N_s(x,y)}^2} d\mu(x)d\mu(y) ds = 0.
    \end{align*}
    Since $\eta\in C_b^3$ and $\sigma \in C_b^4$, it follows for $\modu{\alpha} \leq 2$ and any $i \in \{1,\dots,N\}$ that
    \begin{align*}
        D_\alpha^y \paren{\eta(x)\eta(y)\nabla_{c^i,w^i,b^i}\brac{c^i_s \sigma(w^i_s\cdot x+b^i_s)}\cdot \nabla_{c^i,w^i,b^i}\brac{c^i_s \sigma(w^i_s\cdot y+b^i_s)}} := R(c_s^i,w_s^i,b_s^i) \in C^1_b
    \end{align*} as a function of the parameters $(c_s^i,w_s^i,b_s^i)$ with a $C^1$-bound independent of $(x,y)\in \Omega^2$. Then by the mean value theorem (as in the proof of Lemma \ref{l8}),
    \begin{align*}
        \modu{R(c_s^i,w_s^i,b_s^i) - R(c_0^i,w_0^i,b_0^i)} &\leq \paren{\|c_s^i-c_0^i\|+ \|w_s^i-w_0^i\|+\|b_s^i-b_0^i\|} \sup_{c,w,b}\nor{\nabla_{c,w,b}R(c,w,b)} \\
        &\leq C\paren{\|c_s^i-c_0^i\|+ \|w_s^i-w_0^i\|+\|b_s^i-b_0^i\|} \\
        &\leq CN^{\delta + \beta - 1},
    \end{align*}
    where we use the bounds on the parameter increments provided by Lemma \ref{lemma_param_growth}. Thus
    \begin{align*}
        \modu{D_\alpha^y B_0^N(x,y)-D_\alpha^y B^N_s(x,y)}^2 &\leq \frac{1}{N} \sum_{i=1}^N \modu{R(c_s^i,w_s^i,b_s^i) - R(c_0^i,w_0^i,b_0^i)}^2\\
        &\leq \frac{C}{N}\sum_{i=1}^N N^{2\delta+2\beta-2} \\
        &\leq CN^{2\delta+2\beta-2},
    \end{align*}
    and so
    \begin{align*}
     \mathbb{E}\brac{\modu{D_\alpha^y B_0^N(x,y)-D_\alpha^y B^N_s(x,y)}^2} \leq C N^{2\delta+2\beta-2}.
    \end{align*}
    On the other hand, $D_\alpha^y B^N_0(x,y)$ is the mean of $N$ independent and identically distributed random variables with finite variance. Thus
    \begin{align*}
        \mathbb{E}\brac{\modu{D_\alpha^y B_0^N(x,y)-D_\alpha^y B(x,y)}^2}=\mathbb{E}\brac{\paren{D_\alpha^y B_0^N(x,y)-\mathbb{E}[D_\alpha^y B_0^N(x,y)]}^2} = \text{Var}[D_\alpha^y B_0^N(x,y)] \leq \frac{C}{N}.
    \end{align*}
    Combining the inequalities above gives
    \begin{align*}
        N^{2\delta} \mathbb{E}\brac{\modu{D_\alpha^y B(x,y)-D_\alpha^y B^N_s(x,y)}^2} &\leq 2N^{2\delta}\E\brac{(D_\alpha^yB_s^N(x,y) - D_\alpha^y B_0^N(x,y))^2 + (D_\alpha^yB_0^N(x,y) - D_\alpha^yB(x,y))^2}\\
        &\leq CN^{4\delta+2\beta-2}+ CN^{2\delta-1}.
    \end{align*}
    Thus
    \begin{align*}
        \int_0^T \int_{\Omega^2}N^{2\delta}\mathbb{E}\brac{\modu{D_\alpha^y B(x,y)-D_\alpha^y B^N_s(x,y)}^2} d\mu(x)d\mu(y) ds \leq CN^{4\delta+2\beta-2}+ CN^{2\delta-1} \to 0.
    \end{align*}
\end{proof}

\begin{lemma} \label{lemma_dominated_convergence}
    We have that
    \begin{align}
        \lim_{N \to \infty}\E\brac{\int_0^T \int_\Omega \modu{\g^{U_s}Q_s(x)}^2\mathbf{1}_{\{\modu{\g^{U_s}Q_s} \geq N^\delta\}}(x) + \modu{\partial_uH(U_s,Q_s)(x)}^2\mathbf{1}_{\{\modu{\partial_uH(U_s,Q_s)} \geq N^\delta\}}(x)d\mu(x)ds} = 0.
    \end{align}
\end{lemma}
\begin{proof}
    By Lemma \ref{QQ} and Assumption \ref{assumption_bounded_second_derivative}, it is clear that there exists a constant $C$ such that for each $s \in [0,T]$ and $x \in \Omega$,
    \begin{align*}
        \modu{\g^{U_s}Q_s(x)} \leq C \quad \text{ and } \quad \modu{\partial_uH(U_s,Q_s)(x)} \leq C.
    \end{align*}
    Thus, the dominated convergence theorem (on the product $\sigma$-algebra) gives the desired result.
\end{proof}

\begin{proof}[\textbf{Proof of Lemma \ref{lemma_limit_ode}}]
Equipped respectively with the $C^2$ norm and the supremum norm, the spaces $C^2(\overline\Omega)$ and $C^0(\overline\Omega)$ are Banach. We write the dynamics \eqref{limit_ODE} as $(\partial_t Q_t, \partial_t U_t) = f(t, Q_t, U_t)$ with
\begin{align*}
    f(t,Q,U) := \big(\omega_t\,\mathcal{B}\g^{U}Q,\ -\alpha_t\,\A\,\partial_u H(U,Q)\big),
\end{align*}
where the schedulers $\omega,\alpha$ are bounded and continuous. By Lemma \ref{lemma_kernel_regularity}, $\A: C^0(\overline\Omega)\to C^0(\overline\Omega)$ and $\mathcal{B} : C^0(\overline\Omega)\to C^2(\overline\Omega)$ are bounded linear operators, so $f(t,\cdot,\cdot) : C^2(\overline\Omega)\times C^0(\overline\Omega) \to C^2(\overline\Omega)\times C^0(\overline\Omega)$ is well-defined.

Now fix $R>0$ and let $Q_1,Q_2$ satisfy $\nor{Q_1}_{C^2},\nor{Q_2}_{C^2}\leq R$ and $U_1,U_2 \in C^0(\overline\Omega)$. Splitting $\g^{U_1}Q_1-\g^{U_2}Q_2 = (\g^{U_1}Q_1-\g^{U_1}Q_2)+(\g^{U_1}Q_2-\g^{U_2}Q_2)$ and applying Proposition \ref{proposition_lipschitz}, we get
\begin{align*}
    \nor{\g^{U_1}Q_1-\g^{U_2}Q_2}_{C^0} \leq C(1+R)\big(\nor{Q_1-Q_2}_{C^2}+\nor{U_1-U_2}_{C^0}\big).
\end{align*}
Similarly, Proposition \ref{proposition_lipschitz} implies
\begin{align*}
    \nor{\partial_u H(U_1,Q_1)-\partial_u H(U_2,Q_2)}_{C^0} \leq C(1+R)\big(\nor{Q_1-Q_2}_{C^2}+\nor{U_1-U_2}_{C^0}\big).
\end{align*}
Since $\mathcal{B},\A$ are bounded operators and $\omega, \alpha$ are bounded, $f(t,\cdot,\cdot)$ is Lipschitz on every bounded set of $C^2(\overline\Omega)\times C^0(\overline\Omega)$ with a constant independent of $t$. The continuity of $\omega,\alpha$ also gives continuity of $t\mapsto f(t,Q,U)$. The Banach space Picard--Lindel\"of theorem (cf.~\cite[Chapter~2.2]{kolokoltsov2019differential}) therefore yields a unique solution flow $(Q_t,U_t)$ on some maximal interval $[0,T_{\max})$.

It remains to prove global existence, i.e.~that $T_{\max} = \infty$. From Assumption \ref{assumption_bounded_second_derivative}, we see that $\nor{\g^{U}Q}_{C^0}\leq C(1+\nor{Q}_{C^2})$ and $\nor{\partial_u H(U,Q)}_{C^0}\leq C(1+\nor{Q}_{C^2})$. Thus, on $[0,T_{\max})$,
\begin{align*}
    \|Q_t\|_{C^2} &\leq \|\overline g\|_{C^2} + \int_0^t \|\omega_s \mathcal{B}\mathcal{L}^{U_s}Q_s\|_{C^2}\,ds \leq \|\overline g\|_{C^2} + C\int_0^t \big(1+\|Q_s\|_{C^2}\big)\,ds .
\end{align*}
Gronwall's inequality then gives $\nor{Q_t}_{C^2}\leq (1+\nor{\overline g}_{C^2})e^{Ct}$, a bound which is finite on every $[0,T]$ and independent of $U$. Consequently,
\begin{align*}
    \nor{U_t}_{C^0} \leq \int_0^t \alpha_s\nor{\A\,\partial_u H(U_s,Q_s)}_{C^0}\,ds \leq C\int_0^t \big(1+\nor{Q_s}_{C^2}\big)\,ds < \infty
\end{align*}
on every $[0,T]$. There is therefore no finite-time explosion, and so the flow $(Q_t,U_t)_{t \geq 0} \in C([0,\infty); C^2(\overline{\Omega}) \times C^0(\overline\Omega))$ exists and is unique on $[0,\infty)$.
\end{proof}

\begin{proof}[\textbf{Proof of Theorem \ref{thm_limit_ode}}]
    Integrating the dynamics of $Q^N$ and $Q$ with respect to time gives
    \begin{align*}
        Q_t^N(y) = Q_0^N(y)+ \int_0^t \omega_s \int_\Omega F^N\paren{\mathcal{L}^{U_s^N}Q_s^N(x)}B_s^N(x,y)d\mu(x) ds
    \end{align*}
    and
    \begin{align*}
        Q_t(y)=Q_0(y)+\int_0^t \omega_s \int_\Omega \mathcal{L}^{U_s}Q_s(x) B(x,y)d\mu(x)ds.
    \end{align*}
    Then for arbitrary multi-index $\alpha$ with $\modu{\alpha} \leq 2$,
    \begin{align*}
        D_\alpha(Q_t^N-Q_t)(y)&=D_\alpha (Q_0^N-Q_0)(y)+ \int_0^t \omega_s \int_\Omega F^N\paren{\mathcal{L}^{U_s^N}Q_s^N(x)}D_\alpha^y B_s^N(x,y) - \mathcal{L}^{U_s}Q_s(x) D_\alpha^y B(x,y)d\mu(x) ds.
    \end{align*}
    We may bound the integral above by four integrals that are individually easier to handle:
    \begin{align}
        \begin{split}
            |D_\alpha(Q_t^N-Q_t)(y)| \leq |D_\alpha (Q_0^N-Q_0)(y)| &+ C\int_0^t \int_\Omega |F^N(\mathcal{L}^{U_s^N}Q_s^N(x))||D_\alpha^y B_s^N(x,y)-D_\alpha^y B(x,y)|d\mu(x) ds \\
            &+ C\int_0^t \int_\Omega |F^N(\mathcal{L}^{U_s^N}Q_s^N(x))-F^N(\mathcal{L}^{U_s^N}Q_s(x))||D_\alpha^y B(x,y)|d\mu(x) ds\\
            &+ C\int_0^t \int_\Omega |F^N(\mathcal{L}^{U_s^N}Q_s(x))-F^N(\mathcal{L}^{U_s}Q_s(x))||D_\alpha^y B(x,y)|d\mu(x) ds\\
            &+ C\int_0^t \int_\Omega |F^N(\mathcal{L}^{U_s}Q_s(x))-\mathcal{L}^{U_s}Q_s(x)||D_\alpha^y B(x,y)|d\mu(x) ds.
        \end{split}
    \end{align}
    Using that $|F^N(x)|\leq 2N^\delta$ and $\modu{F^N(x) - x} \leq 2\modu{x}\mathbf{1}_{\{\modu{x} \geq N^\delta\}}$, each $F^N$ is Lipschitz with a uniform constant, Proposition \ref{proposition_lipschitz}, Lemma \ref{QQ}, and that there exists a constant $C$ such that $\sum_{\modu{\alpha} \leq 2}|D_\alpha^y B(x,y)|\leq C$, we have
    \begin{align} \label{Q_split}
    \begin{split}
        |D_\alpha(Q_t^N-Q_t)(y)|
        &\leq |D_\alpha (Q_0^N-Q_0)(y)| + C\int_0^t\int_\Omega N^\delta |D_\alpha^y B_s^N(x,y)-D_\alpha^y B(x,y)|d\mu(x) ds \\
        &\quad+C \int_0^t \int_\Omega \sum_{\modu{\alpha}\leq 2}\modu{D_\alpha \paren{Q_s^N - Q_s}(x)}d\mu(x) ds\\
        &\quad+C \int_0^t \int_\Omega |U_s^N(x)-U_s(x)|\paren{1+\sum_{\modu{\alpha}\leq 2}|D_\alpha Q_s(x)|}d\mu(x) ds\\
        &\quad+ C\int_0^t \int_\Omega |\mathcal{L}^{U_s}Q_s(x)|\mathbf{1}_{\{|\mathcal{L}^{U_s}Q_s|\geq N^{\delta}\}}(x) d\mu(x) ds\\
        &\leq |D_\alpha (Q_0^N-Q_0)(y)| + C\int_0^t\int_\Omega N^\delta |D_\alpha^y B_s^N(x,y)-D_\alpha^y B(x,y)|d\mu(x) ds \\
        &\quad+C \int_0^t \int_\Omega \sum_{|\alpha|\leq 2}|D_\alpha(Q_s^N-Q_s)(x)|+|U_s^N(x)-U_s(x)|d\mu(x) ds\\
        &\quad+ C\int_0^t \int_\Omega |\mathcal{L}^{U_s}Q_s(x)|\mathbf{1}_{\{|\mathcal{L}^{U_s}Q_s|\geq N^{\delta}\}}(x) d\mu(x) ds.
    \end{split}
    \end{align}
    Similarly, for the actor process $U_t^N$, we have
    \begin{align} \label{UN_growth}
        U_t^N(y)=U_0^N(y)-\int_0^t \alpha_s \int_\Omega \psi^N(\partial_u H(U_s^N,Q_s^N)(x))A^N_s(x,y)d\mu(x) ds,
    \end{align}
    and
    \begin{align} \label{U_growth}
        U_t(y)=U_0(y)-\int_0^t \alpha_s\int_\Omega \partial_u H(U_s,Q_s)(x){A}(x,y)d\mu(x) ds.
    \end{align}
    Subtracting \eqref{U_growth} from \eqref{UN_growth}, splitting the integral, and using that $\alpha_s \leq C$ gives
    \begin{align} \label{U_diff_growth}
        \begin{split}
            |U_t^N(y)-U_t(y)| &\leq |U_0^N(y)-U_0(y)| + C \int_0^t \int_\Omega |\psi^N(\partial_u H(U_s^N,Q_s^N)(x))||A_s^N(x,y)-A(x,y)|d\mu(x)ds\\
            &\quad + C\int_0^t \int_\Omega |\psi^N(\partial_u H(U_s^N,Q_s^N)(x))-\psi^N(\partial_u H(U_s,Q_s)(x))||A(x,y)|d\mu(x)ds\\
            &\quad + C\int_0^t \int_\Omega |\psi^N(\partial_u H(U_s,Q_s)(x))-\partial_u H(U_s,Q_s)(x)||A(x,y)|d\mu(x)ds\\
            &\leq |U_0^N(y)-U_0(y)| + C\int_0^t \int_\Omega N^\delta |A_s^N(x,y)-A(x,y)|d\mu(x)ds\\
            &\quad + C\int_0^t \int_\Omega |\partial_u H(U_s^N,Q_s^N)(x)-\partial_u H(U_s,Q_s)(x)||A(x,y)|d\mu(x)ds\\
            &\quad + C \int_0^t \int_\Omega |\partial_u H(U_s,Q_s)(x)|\mathbf{1}_{\{|\partial_u H(U_s,Q_s)|\geq N^\delta\}}(x)|A(x,y)|d\mu(x)ds.
        \end{split}
    \end{align}
    Notice that by Proposition \ref{proposition_lipschitz} and Lemma \ref{QQ},
    \begin{align*}
        &|\partial_u H(U_s^N,Q_s^N)(x)-\partial_u H(U_s,Q_s)(x)|\\ &\leq |\partial_u H(U_s^N,Q_s^N)(x)-\partial_u H(U_s^N,Q_s)(x)|+|\partial_u H(U_s^N,Q_s)(x)-\partial_u H(U_s,Q_s)(x)|\\
        &\leq C \sum_{|\alpha|\leq 2}|D_\alpha(Q_s^N-Q_s)(x)|+ C|U_s^N(x)-U_s(x)|\paren{1+\sum_{|\alpha|\leq 2}|D_\alpha Q_s(x)|} \\
        &\leq C \paren{\sum_{|\alpha|\leq 2}|D_\alpha(Q_s^N-Q_s)(x)|+ |U_s^N(x)-U_s(x)|}.
    \end{align*}
    Using this and applying the bound $|A(x,y)| \leq C$, we have from \eqref{U_diff_growth} that
    \begin{align} \label{U_diff_bound}
        \begin{split}
            \modu{U_t^N(y)-U_t(y)}&\leq \modu{U_0^N(y)-U_0(y)} + C\int_0^t \int_\Omega N^\delta\modu{A_s^N(x,y)-A(x,y)}d\mu(x)ds\\
            &\quad + C \int_0^t\int_\Omega \sum_{|\alpha|\leq 2}\modu{D_\alpha(Q_s^N-Q_s)(x)}+\modu{U_s^N(x)-U_s(x)}d\mu(x)ds\\
            &\quad + C \int_0^t \int_\Omega \modu{\partial_u H(U_s,Q_s)(x)}\mathbf{1}_{\{|\partial_u H(U_s,Q_s)|\geq N^\delta\}}(x)d\mu(x)ds.
        \end{split}
    \end{align}
    Summing up the inequalities \eqref{Q_split} for all multi-indices $\modu{\alpha} \leq 2$ and adding \eqref{U_diff_bound} gives
    \begin{align} \label{Xi_bound}
        \Xi_t^N(y) \leq \Bar{J}_t^N(y)+C\int_0^t \int_\Omega \Xi^N_s(x) d\mu(x)ds,
    \end{align}
    where we define
    \begin{align}
        \Xi_t^N(y):=|U_t^N(y)-U_t(y)|+ \sum_{|\alpha|\leq 2}|D_\alpha(Q_t^N-Q_t)(y)|
    \end{align}
    and a residual term
    \begin{align}
        \begin{split}
            \Bar{J}_t^N(y):&=C \bigg[ \int_0^t \int_\Omega N^\delta \modu{A_s^N(x,y)-A(x,y)}d\mu(x)ds + \int_0^t\int_\Omega N^\delta S_s^N(x,y)d\mu(x) ds\\
            &\quad + \int_0^t \int_\Omega \modu{\mathcal{L}^{U_s}Q_s(x)}\mathbf{1}_{\curlbrac{|\mathcal{L}^{U_s}Q_s|\geq N^{\delta}}}(x)d\mu(x)ds \\
            &\quad+ \int_0^t \int_\Omega\modu{\partial_u H(U_s,Q_s)(x)}\mathbf{1}_{\curlbrac{|\partial_u H(U_s,Q_s)|\geq N^\delta}}(x) d\mu(x) ds+\Xi_0^N(y) \bigg].
        \end{split}
    \end{align}
    Here, $S_s^N(x,y):=\sum_{|\alpha|\leq 2} |D_\alpha^y B(x,y)-D_\alpha^y B_s^N(x,y)|$. Multiplying by $\Xi_t^N(y)$ on both sides of \eqref{Xi_bound}, integrating, and using the Cauchy and Jensen inequalities gives
    \begin{align*}
        \int_\Omega \Xi_t^N(y)^2 d\mu(y) & \leq \int_\Omega \Xi_t^N(y)\Bar{J}_t^N(y) d\mu(y) + \int_0^t \int_{\Omega^2} \Xi_s^N(x)\Xi_t^N(y)d\mu(x)d\mu(y)ds \\
        & \leq C\paren{\int_\Omega \Xi_t^N(y)^2 d\mu(y)}^{\frac{1}{2}}\brac{J_T^N + \int_0^t \paren{\int_\Omega \Xi_s^N(y)^2 d\mu(y)}^{\frac{1}{2}} ds},
    \end{align*}
    where repeatedly using the Cauchy inequality and that $t \leq T$ inspires setting
    \begin{align}
        \begin{split}
            J_T^N&= \paren{\int_\Omega \Xi_0^N(y)^2 d\mu(y)}^{\frac{1}{2}} + \int_0^T\paren{ \int_{\Omega^2} N^{2\delta}|A_s^N(x,y)-A(x,y)|^2d\mu(x)d\mu(y)ds}^{\frac{1}{2}}ds\\
            &\quad +\int_0^T \paren{\int_{\Omega^2} N^{2\delta} S_s^N(x,y)^2 d\mu(x)d\mu(y)}^{\frac{1}{2}}ds+\int_0^T \int_\Omega \modu{\mathcal{L}^{U_s}Q_s(x)}\mathbf{1}_{\{|\mathcal{L}^{U_s}Q_s|\geq N^{\delta}\}}(x)d\mu(x) ds\\
            &\quad +\int_0^T \int_\Omega \modu{\partial_u H(U_s,Q_s)(x)}\mathbf{1}_{\{|\partial_u H(U_s,Q_s)|\geq N^\delta\}}(x) d\mu(x) ds.
        \end{split}
    \end{align}
    Therefore
    \begin{align*}
        \paren{\int_\Omega \Xi_t^N(y)^2 d\mu(y)}^{\frac{1}{2}} \leq J_T^N + C \int_0^t \paren{\int_\Omega \Xi_s^N(y)^2 d\mu(y)}^{\frac{1}{2}} ds,
    \end{align*}
    and so by Gronwall's inequality,
    \begin{align*}
        \paren{\int_\Omega \Xi_t^N(y)^2 d\mu(y)}^{\frac{1}{2}} \leq J_T^N e^{Ct}.
    \end{align*}
    Squaring this relation and taking expectations with respect to the parameter initialization gives
    \begin{align*}
        \E\brac{\int_\Omega \Xi_t^N(y)^2 d\mu(y)} \leq \mathbb{E}\brac{(J_T^N)^2}e^{2Ct}.
    \end{align*}
    Observe that by squaring $J_T^N$, using Jensen's inequality, and taking expectations, we have
    \begin{align}
    \begin{split}
        \E\brac{(J_T^N)^2} &\leq  C\E\bigg[\int_\Omega \Xi_0^N(y)^2 d\mu(y)+\int_0^T \int_{\Omega^2} N^{2\delta}|A_s^N(x,y)-A(x,y)|^2d\mu(x)d\mu(y)ds\\
        &\quad +\int_0^T \int_{\Omega^2} N^{2\delta} S_s^N(x,y)^2 d\mu(x)d\mu(y) ds+\int_0^T \int_\Omega \modu{\mathcal{L}^{U_s}Q_s(x)}^2\mathbf{1}_{\{|\mathcal{L}^{U_s}Q_s|\geq N^{\delta}\}}(x)\\
        &\quad +\modu{\partial_u H(U_s,Q_s)(x)}^2\mathbf{1}_{\{|\partial_u H(U_s,Q_s)|\geq N^\delta\}}(x) d\mu(x) ds\bigg].
    \end{split}
    \end{align}
    Then as $N\to\infty$, the first integral disappears by Lemma \ref{l7}, the second by Lemma \ref{l8}, the third by Lemma \ref{l9}, and the fourth by Lemma \ref{lemma_dominated_convergence}. Thus
    \begin{align*}
        \lim_{N \to \infty} \E\brac{\int_\Omega \Xi_t^N(y)^2 d\mu(y)} = 0.
    \end{align*}
    Further,
    \begin{align*}
        \|U_t^N-U_t\|_{L^2}^2+\|Q_t^N-Q_t\|_{\mathcal{H}^2}^2 &= \int_\Omega \bigg( |U_t^N(y)-U_t(y)|^2+ \sum_{|\alpha|\leq 2}|D_\alpha(Q_t^N-Q_t)(y)|^2\bigg) d\mu(y) \\
        &\leq C \int_\Omega \Xi_t^N(y)^2 d\mu(y).
    \end{align*}
    Therefore
    \begin{align*}
        \lim_{N\to\infty}\E\brac{\nor{U_t^N-U_t}_{L^2}^2+\nor{Q_t^N-Q_t}_{\h^2}^2} = 0.
    \end{align*}
    Since the original training time bound $T\geq t$ is arbitrary, we conclude that this holds for any $t \geq 0$. A final round of applying Jensen's inequality then gives \eqref{eq_limit_ode_convergence}.
\end{proof}
\section{Acknowledgments}
    The authors acknowledge support from His Majesty’s Government in the development of this research. SC acknowledges the support of the UKRI Prosperity Partnership Scheme (FAIR) under EPSRC Grant EP/V056883/1. DJ's research was supported by the EPSRC Centre for Doctoral Training in Industrially Focused Mathematical Modelling (InFoMM) (EP/L015803/1).

    The order of authorship in this paper was determined alphabetically.

    Artificial intelligence tools were used for general proofreading and generating/editing a portion of the code of this paper. The authors assume responsibility for all content.
    
\newpage

\bibliography{refs}
\bibliographystyle{plain}

\newpage

\appendix
\section{Measuring actor-critic agreement with Monte Carlo}

One might be interested in finding a measure of how closely the estimated value function $Q_\phi^{N^*}$ corresponds with the control given by the actor $U_\theta^N =: u$. Thankfully, this is easy to check. Since we have access to a simulation environment and a candidate actor, we may use an Euler--Maruyama scheme to simulate paths with the dynamics 
\begin{align*}
    dX^{u}_t = b(X_t, u(X_t^{u}))dt + \Phi(X_t, u(X_t^{u}))dW_t, \quad X^{u}_0 = x
\end{align*}
and numerically estimate
\begin{align*}
    V^u(x) = \E\brac{\int_0^{\tau_u} c(X_s^u, u_s) e^{-\gamma s} ds + g(X_{\tau_u}^u)e^{-\gamma\tau_u} \Big\rvert X_0^u = x}.
\end{align*}
The hope is to have $$V(x) \approx V^u(x) \approx Q^{N^*}_\phi(x), \quad\quad \text{for each $x \in \Omega$}.$$
To evaluate whether this is the case, for the above constructed problems, we select 1000 random points in the domain and calculate $$V(x) - V^u(x), \quad V(x) - Q_\phi^{N^*}(x), \quad \text{and} \quad Q_\phi^{N^*}(x) - V^u(x),$$
where each $V^u(x)$ is estimated from 2000 sample paths. We then provide a histogram for each of these metrics and calculate their mean square values. That is to say, we give a visualization of the distribution of each error metric and estimate
\begin{align*}
E_1 &:= \frac{1}{\mu(\Omega)}\int_\Omega \paren{V(x) - V^u(x)}^2 d\mu(x) \\
E_2 &:= \frac{1}{\mu(\Omega)}\int_\Omega \paren{V(x) - Q_\phi^{N^*}(x)}^2 d\mu(x) \\
E_3 &:= \frac{1}{\mu(\Omega)}\int_\Omega \paren{Q_\phi^{N^*}(x) - V^u(x)}^2 d\mu(x).
\end{align*}
Note that the figures involving Monte Carlo estimates for $V^u(x)$ are affected by our choice of time increment in the Euler-Maruyama scheme. Here, we use $\Delta t = 0.001$. This is low enough that the introduced discretization error should be negligible, though it is of course inferior to possessing the exact stochastic integrals in question, and we would expect any discretization error to add cost on average.

Table \ref{AC_disagreement_means} shows $E_1$, $E_2$, and $E_3$ for each of the six problems considered previously. Figures \ref{problem_1_AC} -- \ref{problem_5_AC} display histograms of the (signed) differences between the value function, critic, and Monte Carlo estimated value using the actor. Overall, the critic and the Monte Carlo estimated value function seem to match well. There are a few oddities in the data. For example, the difference between the true value function and the Monte Carlo estimate of the value function using the actor is greater for Problem 2B than for Problem 2A ($\zeta^*$) despite the former having a far lower actor mean square error. This is likely because the diffusion coefficient in the latter is larger than the former and so forces the simulated particles to the boundary faster, giving less time for error to accumulate.

\begin{table}[]
\centering
\begin{tabular}{@{}llll@{}}
\toprule
Metric                & $E_1$ (True -- MC) & $E_2$ (True -- Critic) & $E_3$ (Critic -- MC) \\ \midrule
Problem 1 & $2.57 \times 10^{-4}$  & $1.13 \times 10^{-4}$  & $3.67 \times 10^{-4}$ \\
Problem 2A ($\zeta$) & $1.55 \times 10^{-2}$ & $5.35 \times 10^{-4}$ & $1.31 \times 10^{-2}$ \\
Problem 2A ($\zeta^*$) & $7.58 \times 10^{-3}$ & $2.28\times 10^{-4}$ & $9.29 \times 10^{-3}$ \\
Problem 2B & $1.44 \times 10^{-2}$ & $1.78 \times 10^{-4}$ & $1.46 \times 10^{-2}$ \\
Problem 3 (Mod. Ham.) & $1.55 \times 10^{-2}$ & $4.79 \times 10^{-3}$ & $2.70 \times 10^{-2}$ \\
Problem 4 & $1.20 \times 10^{-6}$ & $5.30 \times 10^{-5}$ & $5.57 \times 10^{-5}$ \\
Problem 5 & $1.14 \times 10^{-4}$ & $7.79 \times 10^{-3}$ & $9.07 \times 10^{-3}$
\\ \bottomrule
\end{tabular}
\caption{Actor-critic disagreement square means for each problem}
\label{AC_disagreement_means}
\end{table}

\begin{figure}
\centering
\begin{subfigure}{0.325\textwidth}
   \includegraphics[width=1\linewidth]{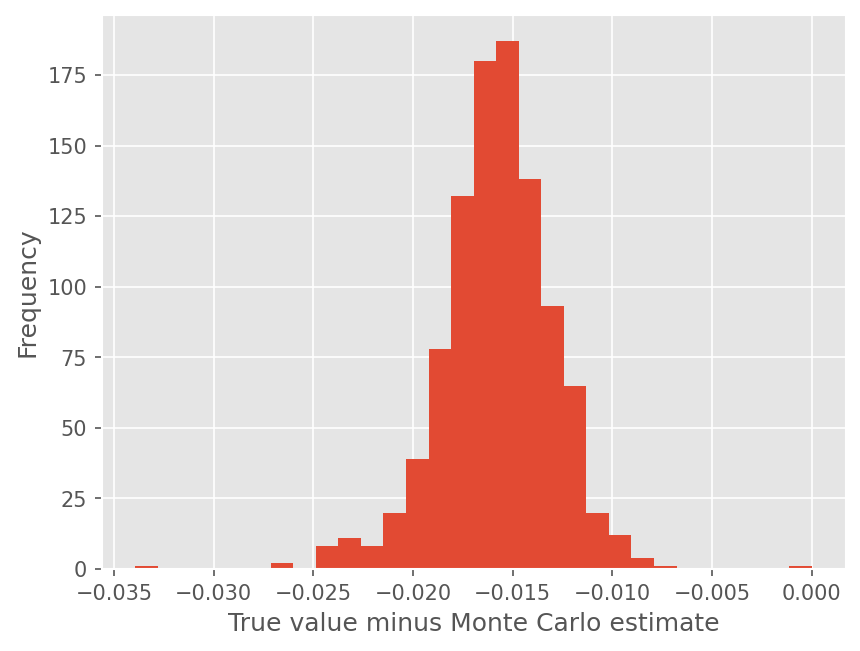}
\end{subfigure}
\begin{subfigure}{0.325\textwidth}
   \includegraphics[width=1\linewidth]{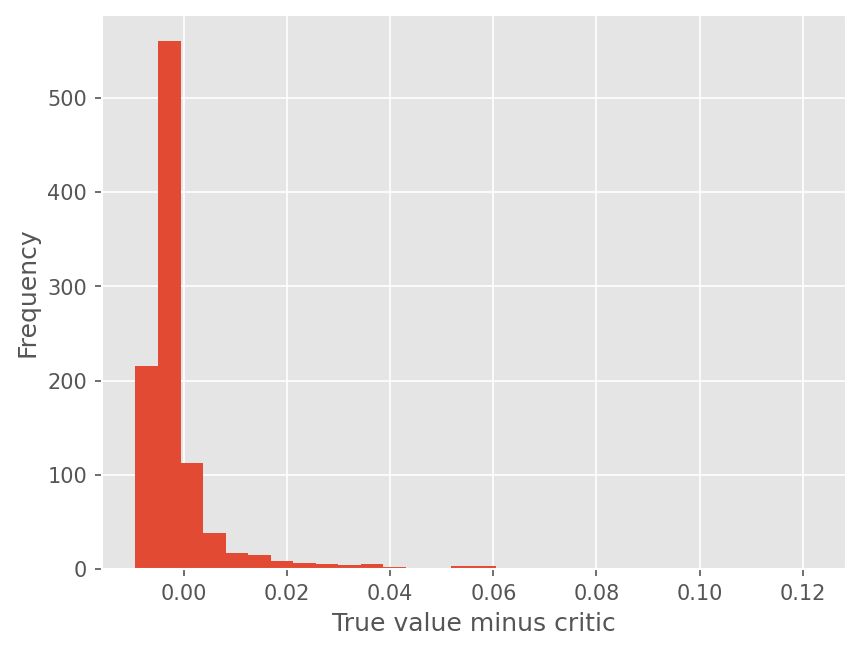}
\end{subfigure}
\begin{subfigure}{0.325\textwidth}
   \includegraphics[width=1\linewidth]{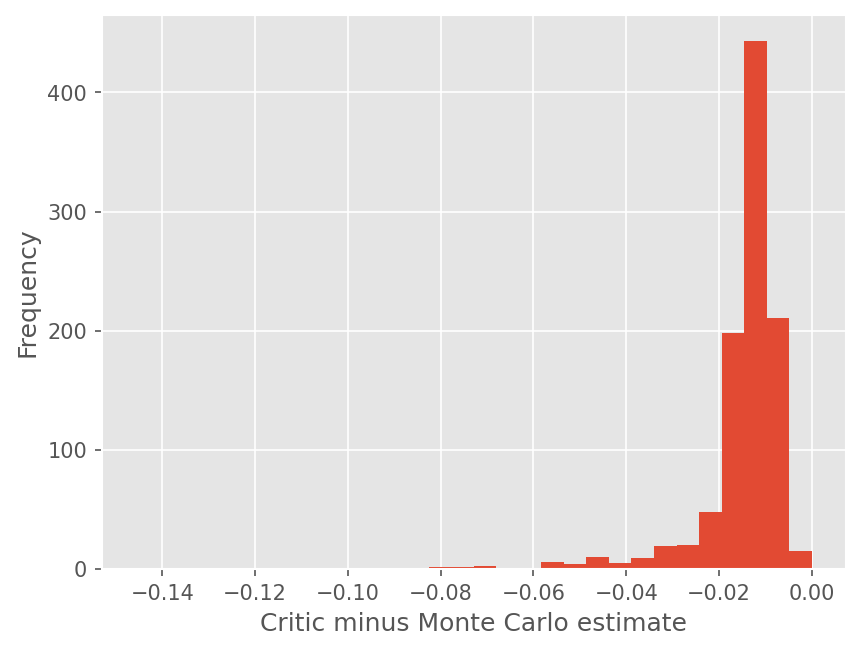}
\end{subfigure}
\caption{Actor-critic disagreement metrics for Problem 1}
\label{problem_1_AC}
\end{figure}

\begin{figure}
\centering
\begin{subfigure}{0.325\textwidth}
   \includegraphics[width=1\linewidth]{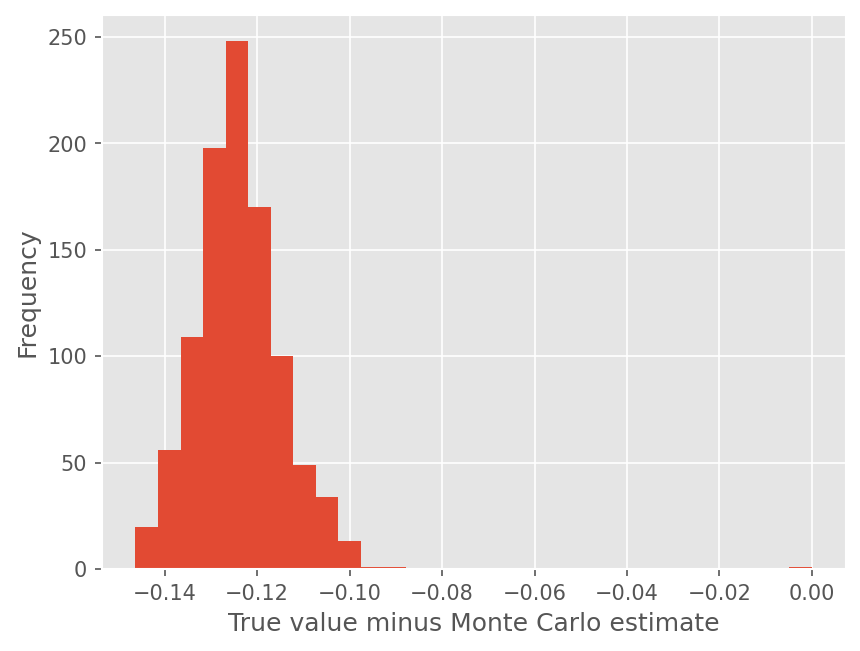}
\end{subfigure}
\begin{subfigure}{0.325\textwidth}
   \includegraphics[width=1\linewidth]{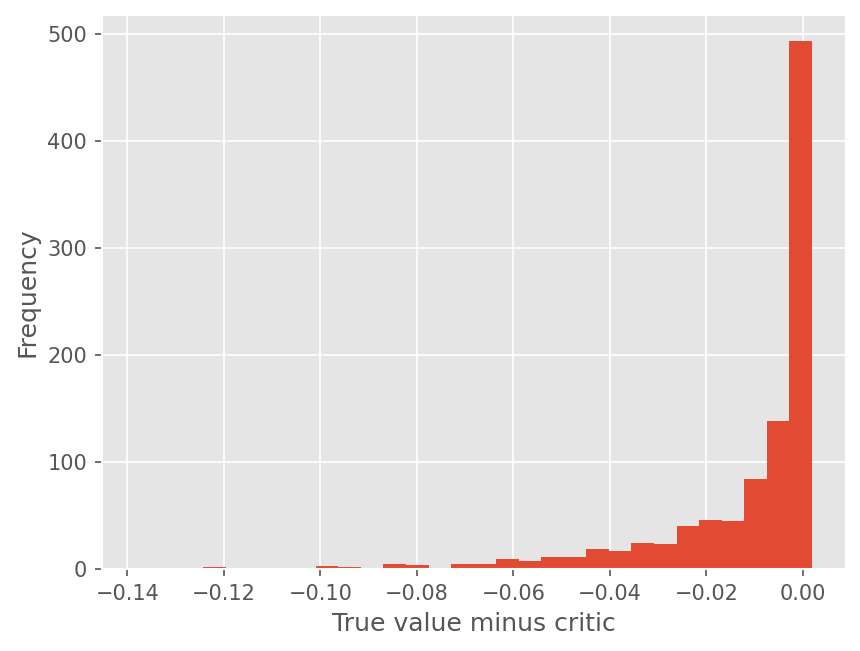}
\end{subfigure}
\begin{subfigure}{0.325\textwidth}
   \includegraphics[width=1\linewidth]{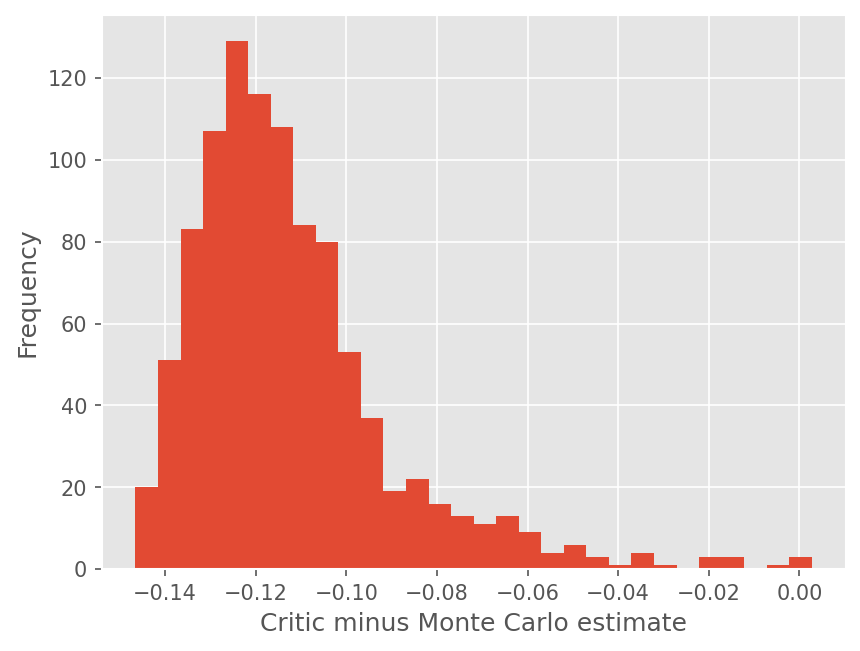}
\end{subfigure}
\caption{Actor-critic disagreement metrics for Problem 2A ($\zeta$)}
\label{problem_2Az_AC}
\end{figure}

\begin{figure}
\centering
\begin{subfigure}{0.325\textwidth}
   \includegraphics[width=1\linewidth]{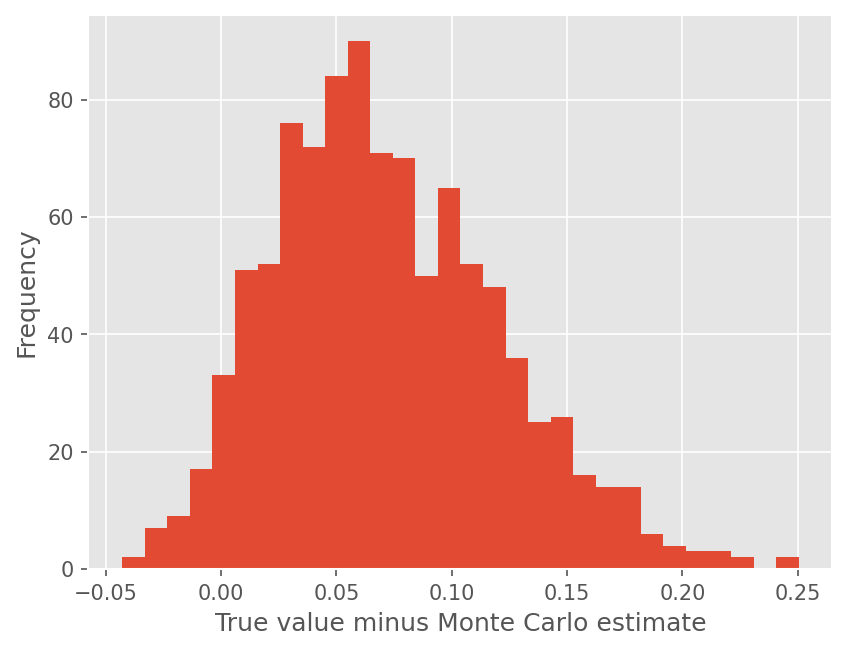}
\end{subfigure}
\begin{subfigure}{0.325\textwidth}
   \includegraphics[width=1\linewidth]{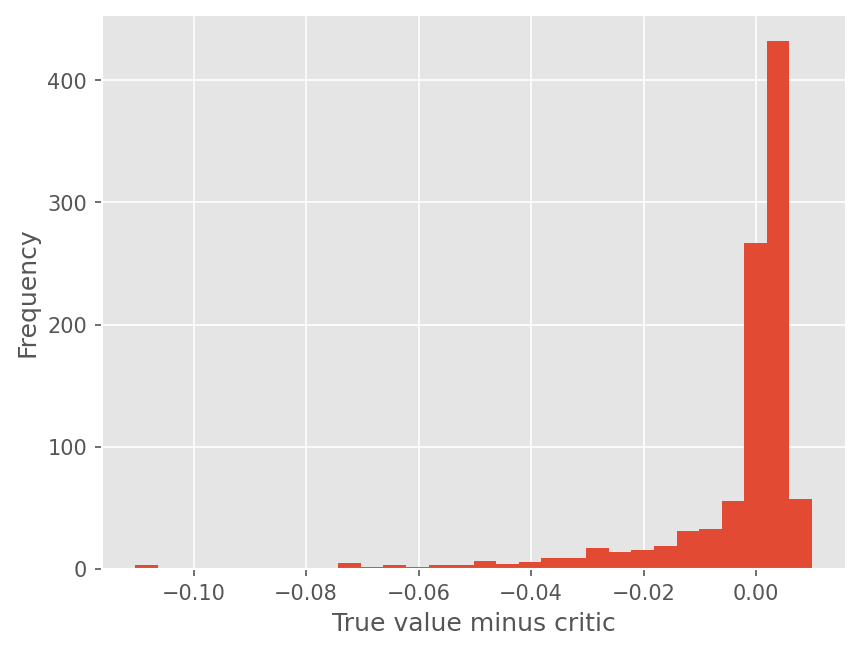}
\end{subfigure}
\begin{subfigure}{0.325\textwidth}
   \includegraphics[width=1\linewidth]{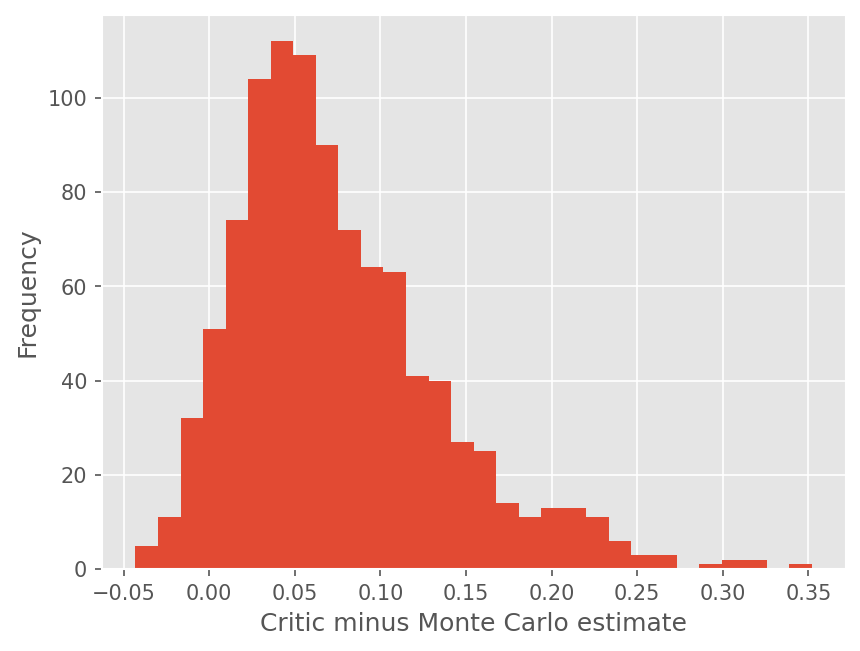}
\end{subfigure}
\caption{Actor-critic disagreement metrics for Problem 2A ($\zeta^*$)}
\label{problem_2A_AC}
\end{figure}

\begin{figure}
\centering
\begin{subfigure}{0.325\textwidth}
   \includegraphics[width=1\linewidth]{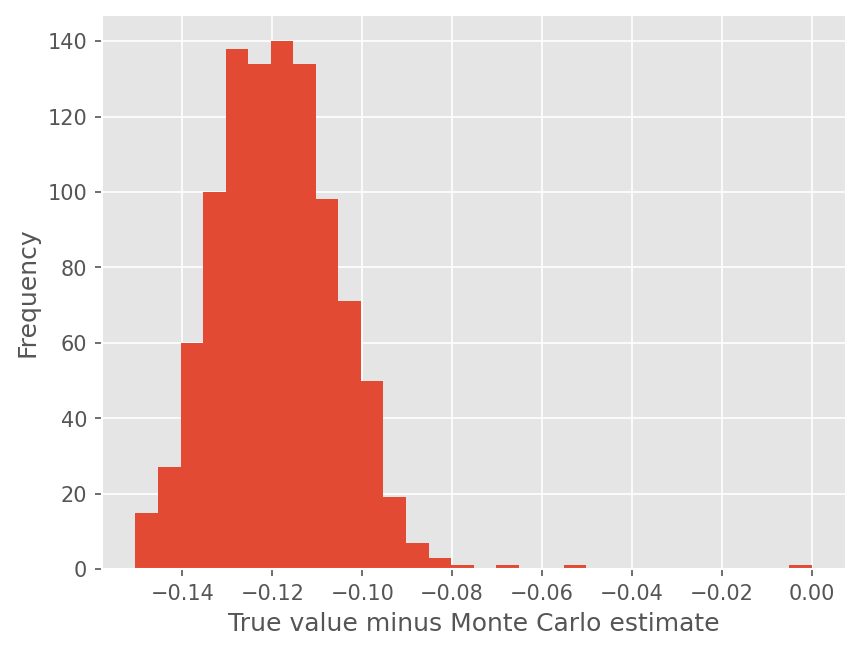}
\end{subfigure}
\begin{subfigure}{0.325\textwidth}
   \includegraphics[width=1\linewidth]{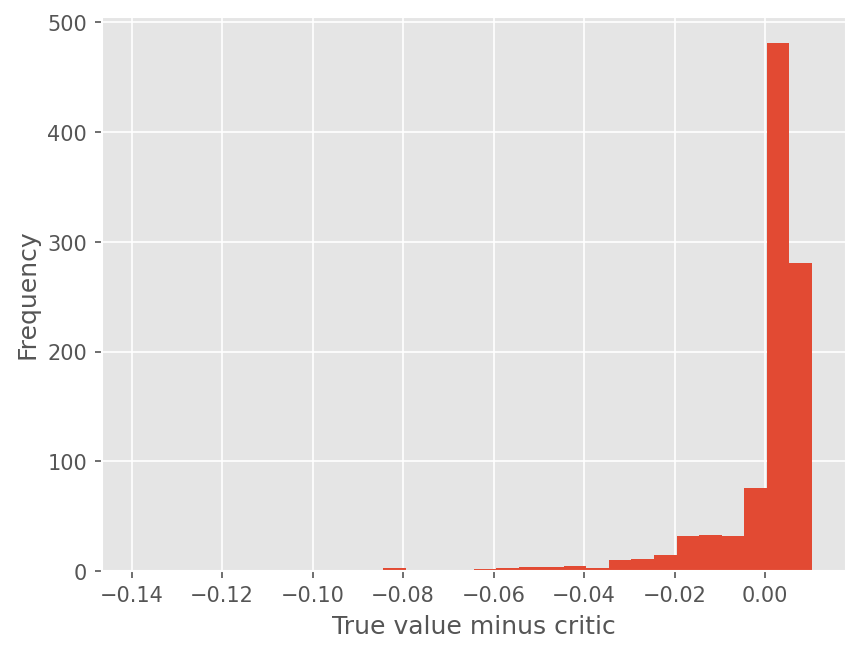}
\end{subfigure}
\begin{subfigure}{0.325\textwidth}
   \includegraphics[width=1\linewidth]{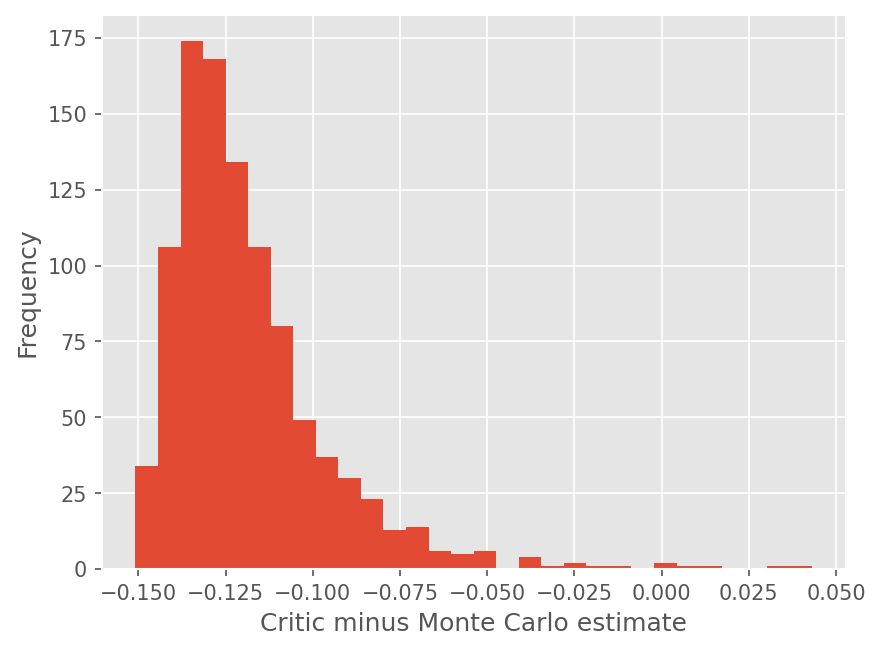}
\end{subfigure}
\caption{Actor-critic disagreement metrics for Problem 2B}
\label{problem_2B_AC}
\end{figure}

\begin{figure}
\centering
\begin{subfigure}{0.325\textwidth}
   \includegraphics[width=1\linewidth]{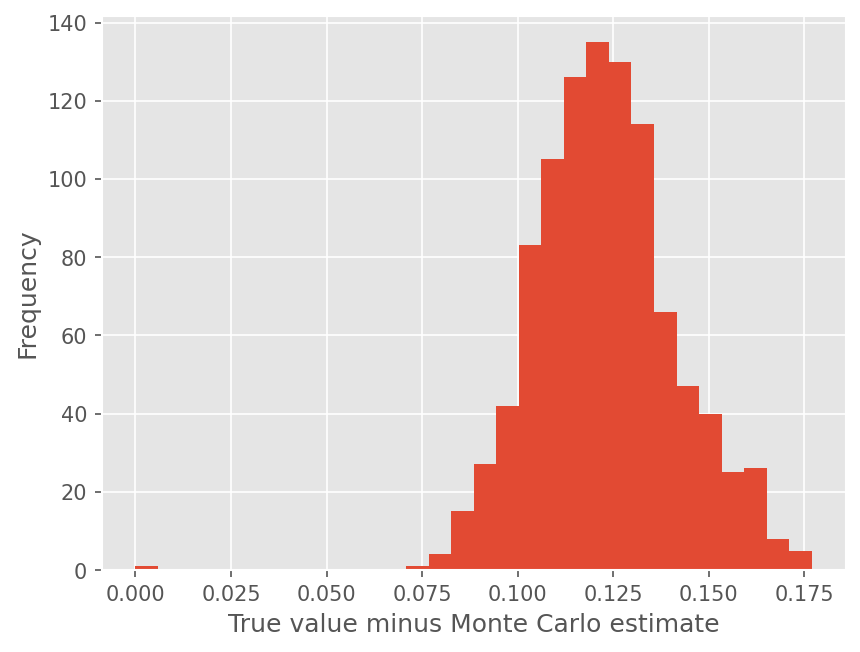}
\end{subfigure}
\begin{subfigure}{0.325\textwidth}
   \includegraphics[width=1\linewidth]{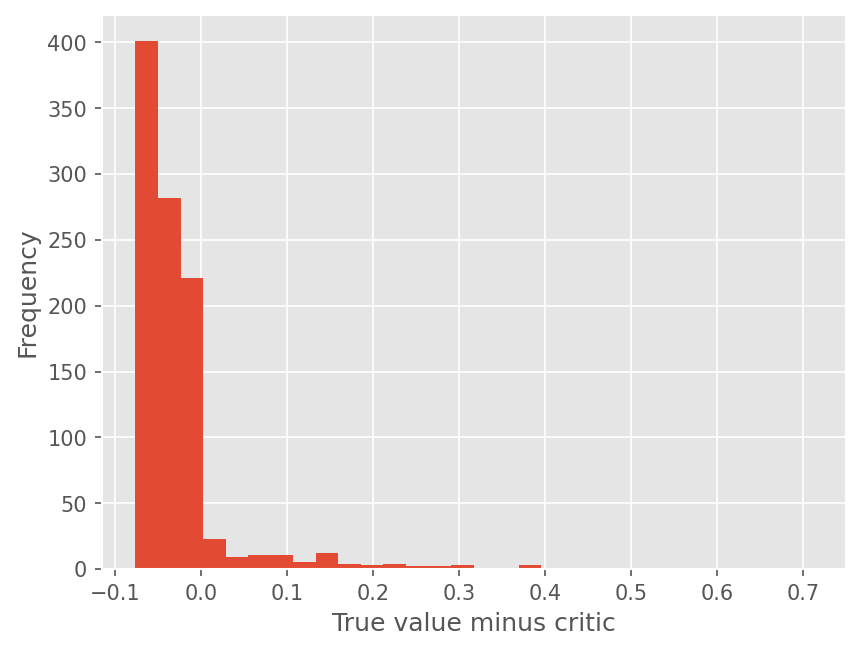}
\end{subfigure}
\begin{subfigure}{0.325\textwidth}
   \includegraphics[width=1\linewidth]{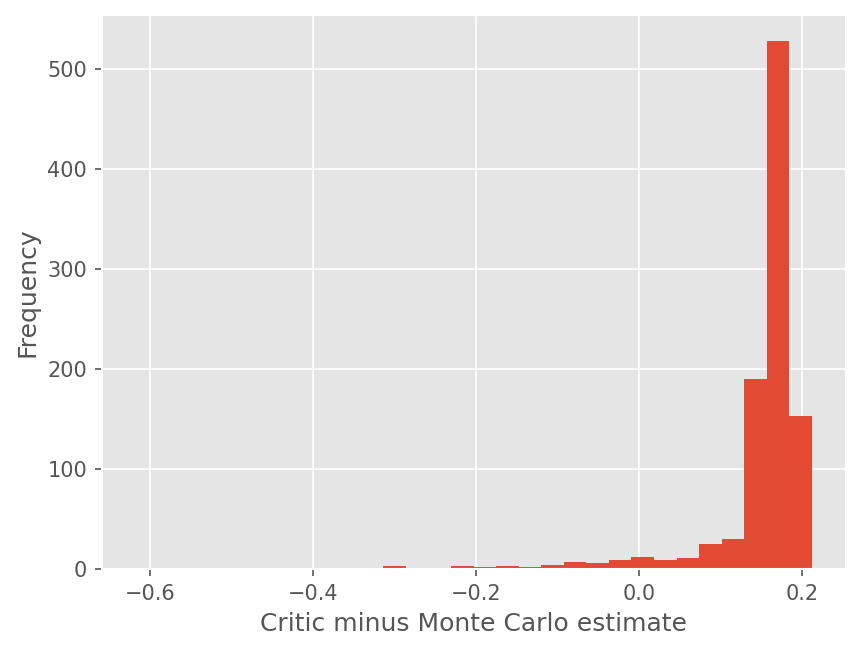}
\end{subfigure}
\caption{Actor-critic disagreement metrics for Problem 3}
\label{problem_3_AC}
\end{figure}

\begin{figure}
\centering
\begin{subfigure}{0.325\textwidth}
   \includegraphics[width=1\linewidth]{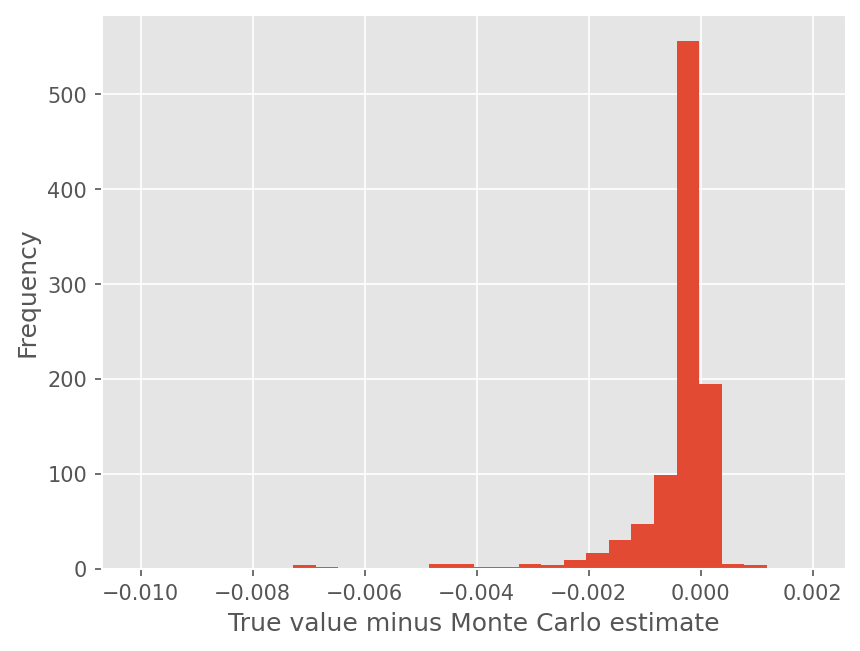}
\end{subfigure}
\begin{subfigure}{0.325\textwidth}
   \includegraphics[width=1\linewidth]{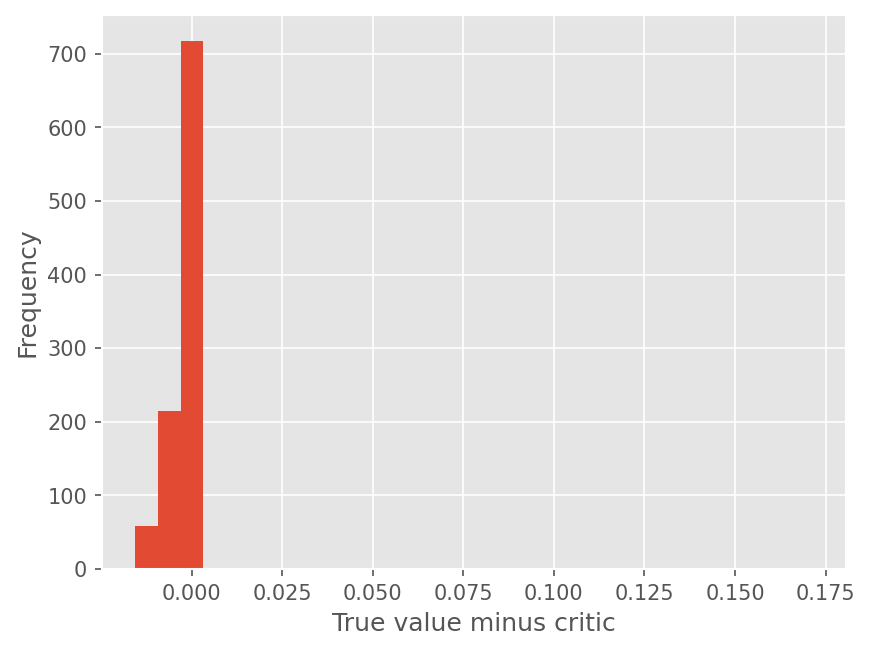}
\end{subfigure}
\begin{subfigure}{0.325\textwidth}
   \includegraphics[width=1\linewidth]{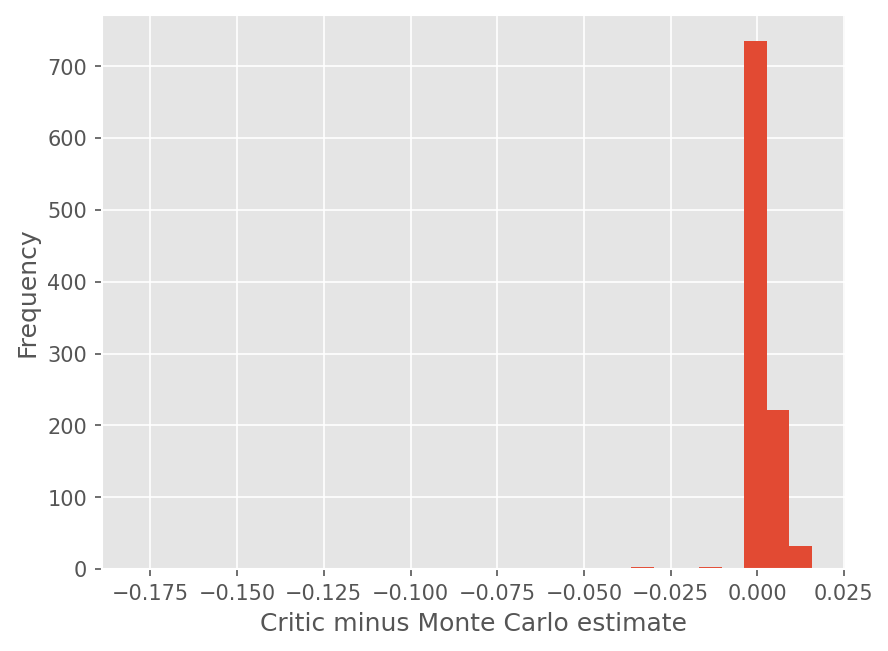}
\end{subfigure}
\caption{Actor-critic disagreement metrics for Problem 4}
\label{problem_4_AC}
\end{figure}

\begin{figure}
\centering
\begin{subfigure}{0.325\textwidth}
   \includegraphics[width=1\linewidth]{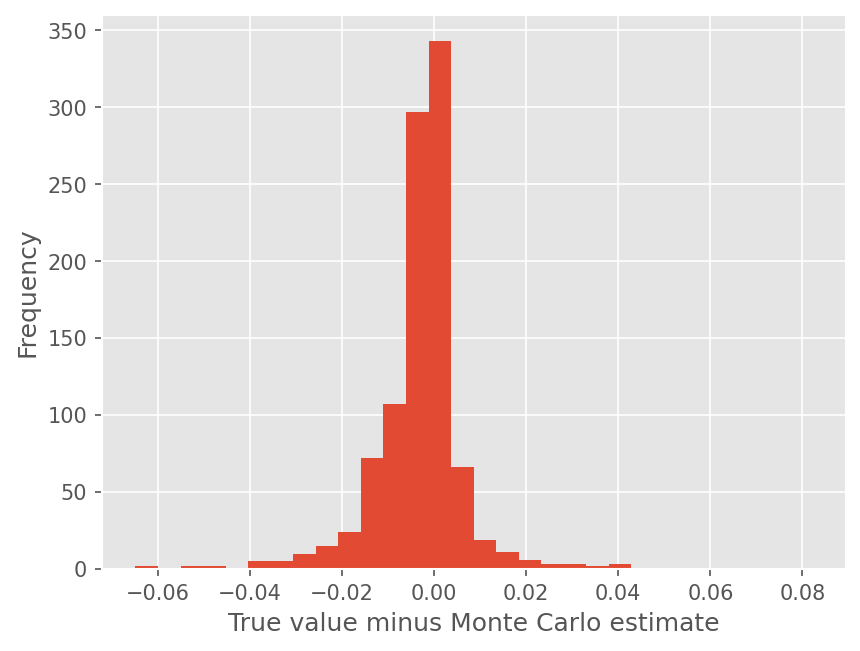}
\end{subfigure}
\begin{subfigure}{0.325\textwidth}
   \includegraphics[width=1\linewidth]{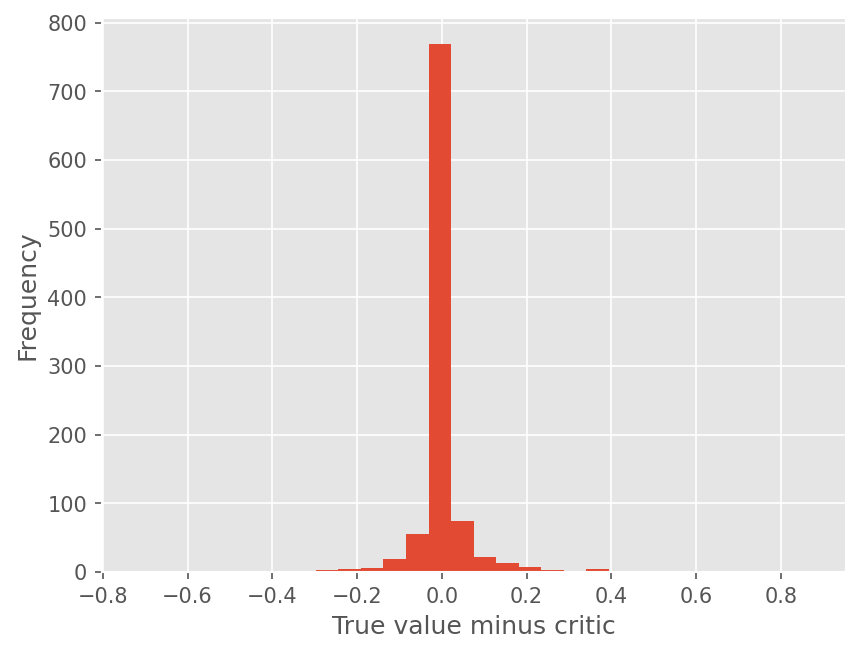}
\end{subfigure}
\begin{subfigure}{0.325\textwidth}
   \includegraphics[width=1\linewidth]{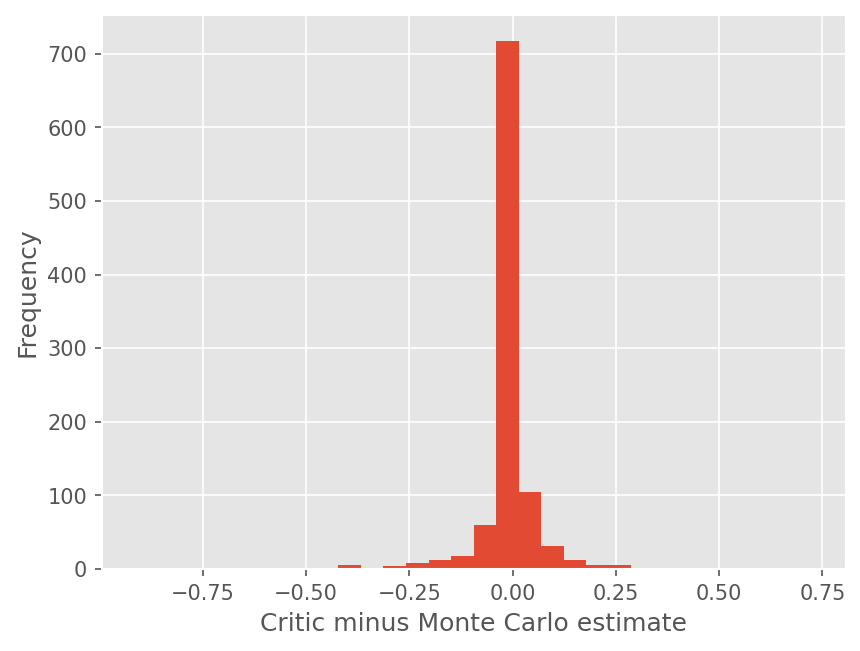}
\end{subfigure}
\caption{Actor-critic disagreement metrics for Problem 5}
\label{problem_5_AC}
\end{figure}
\end{document}